\documentclass[a4j,10pt,showkeys]{article}

\usepackage{amsmath,amsthm,amssymb}
\usepackage{geometry}
\usepackage{hyperref}
\usepackage[capitalize]{cleveref}
\usepackage{cancel}
\usepackage{mathtools}
\usepackage{bm}
\usepackage{color}
\usepackage{mathrsfs}
\usepackage[dvipsnames]{xcolor}
\hypersetup{
    colorlinks=true,
    citecolor=MidnightBlue,
    linkcolor=Red,
    urlcolor=OliveGreen,
}
\usepackage{tikz}
\usetikzlibrary{calc}
\tikzset{set label/.style={fill=white}}

\newtheorem{theo}{Theorem}[section]
\newtheorem{lemm}[theo]{Lemma}

\newtheorem{cor}[theo]{Corollary}
\theoremstyle{definition}
\newtheorem{defi}[theo]{Definition}
\newtheorem{rem}[theo]{Remark}
\newtheorem{assum}{Assumption}
\newtheorem{exam}[theo]{Example}

\newcommand{\bE}{\mathbb{E}}
\newcommand{\bF}{\mathbb{F}}

\newcommand{\bN}{\mathbb{N}}

\newcommand{\bP}{\mathbb{P}}
\newcommand{\bQ}{\mathbb{Q}}
\newcommand{\bR}{\mathbb{R}}

\newcommand{\bW}{\mathbb{W}}

\newcommand{\cB}{\mathcal{B}}

\newcommand{\cD}{\mathcal{D}}

\newcommand{\cF}{\mathcal{F}}
\newcommand{\cG}{\mathcal{G}}
\newcommand{\cH}{\mathcal{H}}

\newcommand{\cK}{\mathcal{K}}
\newcommand{\cL}{\mathcal{L}}

\newcommand{\cN}{\mathcal{N}}

\newcommand{\cP}{\mathcal{P}}

\newcommand{\cS}{\mathcal{S}}

\newcommand{\cU}{\mathcal{U}}
\newcommand{\cV}{\mathcal{V}}

\newcommand{\cY}{\mathcal{Y}}

\newcommand{\sC}{\mathscr{C}}

\newcommand{\ep}{\varepsilon}
\newcommand{\diff}{\mathrm{d}}
\newcommand{\dmu}{\,\mu(\diff\theta)}
\newcommand{\dmud}{\,\mu(\diff\theta')}
\newcommand{\supp}{{\mathrm{supp}\,\mu}}
\newcommand{\LG}{\mathrm{LG}}
\newcommand{\BDG}{\mathrm{BDG}}
\newcommand{\Bb}{\cB_\mathrm{b}(\cH_\mu)}
\newcommand{\Cb}{C_\mathrm{b}(\cH_\mu)}

\newcommand{\dual}[2]{{}_{\cV^*_\mu}\langle{#1},{#2}\rangle_{\cV_\mu}}

\newcommand{\relmiddle}[1]{\mathrel{}\middle#1\mathrel{}}
\newcommand{\1}{\mbox{\rm{1}}\hspace{-0.25em}\mbox{\rm{l}}}

\providecommand{\keywords}[1]{\textbf{Keywords:} #1}
\makeatletter

\@addtoreset{equation}{section}
\makeatother
\def\widebar{\accentset{{\cc@style\underline{\mskip10mu}}}}
\numberwithin{equation}{section}
\allowdisplaybreaks

\title{Markovian lifting and asymptotic log-Harnack inequality for stochastic Volterra integral equations}

\author{
Yushi Hamaguchi\footnote{Graduate School of Engineering Science, Department of Systems Innovation, Osaka University. 1-3, Machikaneyama, Toyonaka, Osaka, Japan. Email: \href{mailto:hmgch2950@gmail.com}{hmgch2950@gmail.com}}\ \footnote{The author was supported by JSPS KAKENHI Grant Number 22K13958.}
}

\begin{document}
\maketitle

%% Abstract

\begin{abstract}
We introduce a new framework of Markovian lifts of stochastic Volterra integral equations (SVIEs for short) with completely monotone kernels. We define the state space of the Markovian lift as a separable Hilbert space which incorporates the singularity or regularity of the kernel into the definition. We show that the solution of an SVIE is represented by the solution of a lifted stochastic evolution equation (SEE for short) defined on the Hilbert space and prove the existence, uniqueness and Markov property of the solution of the lifted SEE. Furthermore, we establish an asymptotic log-Harnack inequality and some consequent properties for the Markov semigroup associated with the Markovian lift via the asymptotic coupling method.
\end{abstract}

%% Keywords

\keywords
Stochastic Volterra integral equation; Markovian lift; asymptotic log-Harnack inequality.

%% MSC


%60H20 Stochastic integral equations
%60H15 Stochastic partial differential equations (aspects of stochastic analysis)
%60G22 Fractional processes, including fractional Brownian motion
%37A25 Ergodicity, mixing, rates of mixing

%45D05 Volterra integral equations
%45A05 Linear integral equations
%93E20 Optimal stochastic control
%93B52 Feedback control
%49K45 Optimality conditions for problems involving randomness
%45G05 Singular nonlinear integral equations
%49N15 Duality theory (optimization)
%45B05 Fredholm integral equations
%34A08 Fractional ordinary differential equations and fractional differential inclusions
%26A33 Fractional derivatives and integrals

%%%%%%%%%%%%%%%%%%%%%%%%%%%%%%%%%%
%%%%%%%%%%%%%%%%%%%%%%%%%%%%%%%%%%
%% Section
%%%%%%%%%%%%%%%%%%%%%%%%%%%%%%%%%%
%%%%%%%%%%%%%%%%%%%%%%%%%%%%%%%%%%

\section{Introduction}\label{intro}

We consider the following stochastic Volterra integral equation (SVIE for short):
\begin{equation}\label{intro_eq_SVIE}
	X_t=x(t)+\int^t_0K(t-s)b(X_s)\,\diff s+\int^t_0K(t-s)\sigma(X_s)\,\diff W_s,\ t>0,
\end{equation}
where $W$ is a $d$-dimensional Brownian motion, $b:\bR^n\to\bR^n$ and $\sigma:\bR^n\to\bR^{n\times d}$ are deterministic maps, $x$ is an $\bR^n$-valued deterministic function called the forcing term, and $K:(0,\infty)\to[0,\infty)$ is a deterministic function called the kernel. On the one hand, if $K$ and $x$ are constants or of the exponential forms $K(t)=e^{-\beta t}$ and $x(t)=e^{-\beta t}y$ for some $\beta>0$ and $y\in\bR^n$, then the SVIE \eqref{intro_eq_SVIE} reduces to a standard stochastic differential equation (SDE for short). On the other hand, the case of the fractional kernel $K(t)=\frac{1}{\Gamma(\alpha)}t^{\alpha-1}$ with $\alpha\in(\frac{1}{2},1)$ corresponds to a kind of time-fractional SDE (cf.\ \cite{SaKiMa87}), where $\Gamma(\alpha)$ denotes the Gamma function. SVIEs were first studied by Berger and Mizel \cite{BeMi80a,BeMi80b}, and then applied in many research areas including population dynamics, tumor growth, and volatility models in mathematical finance. SVIEs provide suitable models for dynamics with hereditary properties, memory effects and the power-law property of the volatility (cf.\ \cite{BaBeVe11,GrLoSt90,Sc06,GaJaRo18}).

It is well-known that the solutions of SVIEs are neither Markovian nor semimartingales in general, and thus we cannot apply It\^{o}'s stochastic calculus to SVIEs directly. However, if we ``lift'' the state space of SVIEs to an infinite dimensional space, we can recover a kind of Markov property. To see this phenomenon, let us consider the following Volterra process:
\begin{equation*}
	B_t:=\int^t_0K(t-s)\,\diff W_s,\ t>0,
\end{equation*}
with a $1$-dimensional Brownian motion $W$ and the fractional kernel $K(t)=\frac{1}{\Gamma(\alpha)}t^{\alpha-1}$ of order $\alpha\in(\frac{1}{2},1)$. The Volterra process $B$ is a Riemann--Liouville type fractional Brownian motion of the Hurst parameter $H=\alpha-\frac{1}{2}\in(0,\frac{1}{2})$. Carmona and Coutin \cite{CaCu98} proposed a Markovian representation approach for $B$ of the following form:
\begin{equation}\label{intro_eq_representation1}
	B_t=\int_{[0,\infty)}Y_t(\theta)\dmu,\ t>0,
\end{equation}
where $\dmu=\frac{1}{\Gamma(\alpha)\Gamma(1-\alpha)}\theta^{-\alpha}\1_{(0,\infty)}(\theta)\,\diff\theta$, and each $Y(\theta)$ is an Ornstein--Uhlenbeck process with the mean reverting parameter $\theta$, which solves the SDE
\begin{equation*}
	\begin{dcases}
	\diff Y_t(\theta)=-\theta Y_t(\theta)\,\diff t+\diff W_t,\ t>0,\\
	Y_0(\theta)=0.
	\end{dcases}
\end{equation*}
The above representation means that the fractional Brownian motion $B$ can be represented as a superposition of Ornstein--Uhlenbeck processes with various mean reverting parameters. Although the fractional Brownian motion $B$ is neither Markovian nor a semimartingale, its lift $(Y(\theta))_{\theta\in(0,\infty)}$ can be seen as a Markov process in a function space. This approach was adopted and applied to the studies of tractable financial models with fractional features \cite{AbiJa19,AbiJaEu19,HaSt19}, optimal control problems for SVIEs \cite{AbiJaMiPh21}, and efficient numerical approximations of SVIEs \cite{AlKe21,BaBr23}, among others.

In this paper, motivated by \cite{CaCu98}, we formulate a new framework of Markovian lifting of SVIEs and investigate some asymptotic behaviours of the associated Markov semigroup. Here, let us give a formal observation. We assume that the kernel $K$ is completely monotone, that is, $K$ is infinitely differentiable and satisfies $(-1)^k\frac{\diff^k}{\diff t^k}K(t)\geq0$ for any $t\in(0,\infty)$ and any nonnegative integers $k$. By Bernstein's theorem, there exists a unique Radon measure $\mu$ on $[0,\infty)$ such that
\begin{equation*}
	K(t)=\int_{[0,\infty)}e^{-\theta t}\dmu=\int_\supp e^{-\theta t}\dmu,\ t>0,
\end{equation*}
where $\supp\subset[0,\infty)$ denotes the support of $\mu$. Furthermore, assume that the forcing term $x$ is of the form
\begin{equation*}
	x(t)=(\cK y)(t):=\int_\supp e^{-\theta t}y(\theta)\dmu,\ t>0,
\end{equation*}
for some Borel measurable function $y:\supp\to\bR^n$. Under suitable integrability conditions, by the (stochastic) Fubini theorem, the SVIE \eqref{intro_eq_SVIE} becomes
\begin{align*}
	X_t&=(\cK y)(t)+\int^t_0K(t-s)b(X_s)\,\diff s+\int^t_0K(t-s)\sigma(X_s)\,\diff W_s\\
	&=\int_\supp e^{-\theta t}y(\theta)\dmu+\int^t_0\int_\supp e^{-\theta(t-s)}\dmu\,b(X_s)\,\diff s+\int^t_0\int_\supp e^{-\theta(t-s)}\dmu\,\sigma(X_s)\,\diff W_s\\
	&=\int_\supp\Big\{e^{-\theta t}y(\theta)+\int^t_0e^{-\theta(t-s)}b(X_s)\,\diff s+\int^t_0e^{-\theta(t-s)}\sigma(X_s)\,\diff W_s\Big\}\dmu.
\end{align*}
Thus, $X$ can be represented by
\begin{equation}\label{intro_eq_representation2}
	X_t=\int_\supp Y_t(\theta)\dmu,\ t>0,
\end{equation}
with $Y:\Omega\times[0,\infty)\times\supp\to\bR^n$ given by
\begin{equation*}
	Y_t(\theta)=e^{-\theta t}y(\theta)+\int^t_0e^{-\theta(t-s)}b(X_s)\,\diff s+\int^t_0e^{-\theta(t-s)}\sigma(X_s)\,\diff W_s,\ t\geq0,\ \theta\in\supp,
\end{equation*}
or equivalently, in the differential form
\begin{equation*}
	\begin{dcases}
	\diff Y_t(\theta)=-\theta Y_t(\theta)\,\diff t+b(X_t)\,\diff t+\sigma(X_t)\,\diff W_t,\ t>0,\ \theta\in\supp,\\
	Y_0(\theta)=y(\theta),\ \theta\in\supp.
	\end{dcases}
\end{equation*}
The above formal observation means that the solution of the SVIE \eqref{intro_eq_SVIE} can be represented as a superposition of the collection of the parametrized processes $Y(\theta)$, $\theta\in\supp$. Then, we can guess that the function-valued lifted process $Y=(Y(\theta))_{\theta\in\supp}$ solves a kind of infinite dimensional stochastic evolution equation (SEE for short) and enjoys the Markov property as an infinite dimensional process. From this formal observation, the following natural questions arise:
\begin{itemize}
\item[Q\,1.]
What is the reasonable state space in which the lifted process $Y$ becomes a Markov process?
\item[Q\,2.]
If we can get a reasonable state space of the Markov process, what can we say more about the associated Markov semigroup?
\end{itemize}

In this paper, we address the above questions. The following are our main contributions in this paper.
\begin{itemize}
\item[(1)]
\emph{We propose and formulate a state space of lifted processes as a separable Hilbert space, and investigate its fundamental properties.} See \cref{lift-space}.
\item[(2)]
\emph{We rigorously prove the equivalence (in some sense) between the SVIE \eqref{intro_eq_SVIE} and the associated Hilbert-valued lifted SEE \eqref{lift_eq_SEE}.} See \cref{lift_theo_equivalence}.
\item[(3)]
\emph{Under suitable assumptions, we show that the lifted SEE has a unique solution which satisfies the Markov and Feller properties.} See \cref{lift_theo_Markov}.
\item[(4)]
\emph{Under suitable assumptions, we derive an asymptotic log-Harnack inequality for the associated Markov semigroup.} See \cref{Harnack_theo_Harnack}.
\end{itemize}

Let us make some comments about the above questions and contributions.

Concerning the question Q\,1, a first possible candidate of the state space of lifts would be $L^1(\mu)$. Indeed, Carmona and Coutin \cite{CaCu98}, Harms and Stefanovits \cite{HaSt19} and Abi Jaber, Miller and Pham \cite{AbiJaMiPh21} adopted the space $L^1(\mu)$ as the state space of lifts. This choice is natural in view of the integral representations \eqref{intro_eq_representation1} and \eqref{intro_eq_representation2}. However, it has a drawback that each non-zero constant vector does not belong to $L^1(\mu)$ in general. Indeed, $\mu$ becomes an infinite measure when the kernel $K$ is singular (in the sense that $\lim_{t\downarrow0}K(t)=\infty$). As a consequence, the $\bR^n$-valued drift coefficient $b$ or the $\bR^{n\times d}$-valued diffusion coefficient $\sigma$ cannot be realized as $L^1(\mu)$-valued maps. Moreover, $L^1(\mu)$ is not a Hilbert space, which makes the infinite dimensional analysis of the lifted equation difficult. A second candidate would be $L^2(\mu)$ which was treated in \cite{HaSt19}. In this case, although it is a Hilbert space, the coefficients $b$ and $\sigma$ cannot be realized as $L^2(\mu)$-valued maps by the same reasons as above. Furthermore, when the kernel $K$ is singular, the measure becomes an infinite measure, and thus $L^2(\mu)\setminus L^1(\mu)\neq\emptyset$ in general. Hence, the integral representations \eqref{intro_eq_representation1} or \eqref{intro_eq_representation2} might not make sense, and this would be one reason why \cite{AbiJaMiPh21,CaCu98} avoided to choose $L^2(\mu)$ as the state space; see \cite[Remark 2.4]{AbiJaMiPh21}. Cuchiero and Teichmann \cite{CuTe19,CuTe20} considered Markovian lifts of SVIEs in view of measure-valued stochastic partial differential equations (SPDEs for short), and formulated a framework of associated Markov processes called generalized Feller processes. In our notations, the measure-valued process $\lambda_t(\diff\theta):=Y_t(\theta)\dmu$ corresponds to the Markovian lift considered therein. In their framework, the state space of lifts is taken as the space of finite signed (or vector-valued) regular Borel measures on the extended real half-line $[0,\infty]$ endowed with its weak-$*$ topology. They applied an abstract framework of measure-valued SPDEs to affine Volterra processes and obtained the affine transformation formula in the infinite dimensional point of view. However, since the state space is not a Polish space\footnote{The continuous dual space of an infinite dimensional Banach space is not metrizable with respect to the weak-$*$ topology.}, one cannot apply general theory of Markov processes on Polish spaces to their framework directly.

In contrast to the aforementioned works, we propose a \emph{Hilbert-valued Markovian lifting approach} for SVIEs. Our framework covers SVIEs with (completely monotone) singular kernels such as the fractional kernel $K(t)=\frac{1}{\Gamma(\alpha)}t^{\alpha-1}$ and the Gamma kernel $K(t)=\frac{1}{\Gamma(\alpha)}e^{-\beta t}t^{\alpha-1}$ with $\alpha\in(\frac{1}{2},1)$ and $\beta>0$. Our idea is to incorporate the singularity/regularity of the kernel into the definition of the state space explicitly as a weight of an $L^2$-space. We demonstrate that the state space has good structures, and the lifted SEE fits into the well-established framework of SPDEs in Hilbert spaces. In this new framework, we show that the solution of an SVIE is represented by the solution of a lifted SEE. Then, we show the existence, uniqueness and the Markov and Feller properties of the solutions of lifted SEEs. Our general framework can be a useful tool to study SVIEs by means of It\^{o}'s stochastic calculus in Hilbert spaces, and it has potential applications to mathematical finance, optimal control theory, ergodic theory and numerical analysis for SVIEs.

In our new framework, we address the question Q\,2 mentioned above. Specifically, we are concerned with the long-time asymptotic behaviour and the ergodicity of the Hilbert-valued Markovian lifts of SVIEs. There are only a few previous works in this direction. Friesen and Jin \cite{FrJi22} studied Volterra square-root processes and showed the existence of the long-time limiting distributions under suitable integrability conditions of the kernel. They applied their results to the Volterra CIR process with a Gamma kernel of the form $K(t)=\frac{1}{\Gamma(\alpha)}e^{-\beta t}t^{\alpha-1}$ with $\alpha\in(\frac{1}{2},1)$ and $\beta>0$ and characterized its stationary distribution.
%
%Jacquier, Pannier and Spiliopoulos \cite{JaPaSp22} studied the ergodic behaviour of the Markovian lifts of affine Volterra processes. They adopted the framework of measure-valued SPDEs introduced by \cite{CuTe19,CuTe20} and showed the existence of an invariant probability measure for the associated Markovian lift; the uniqueness of the invariant probability measure is left open though.
%
Benth, Detering and Kr\"{u}hner \cite{BeDeKr22} investigated long-time limiting distributions of SVIEs by means of abstract results of SPDEs. Although this work covers rather general classes of SVIEs, it requires that the kernel $K$ is sufficiently regular which excludes singular kernels such as the fractional kernel or the Gamma kernel.

In the last half of this paper, by means of the Hilbert-valued Markovian lifting approach which we introduce in the first half, we study the long-time asymptotic behaviour and the ergodicity for the lifted SEE of the SVIE \eqref{intro_eq_SVIE}. We begin with detailed analysis of a Gaussian Volterra process as the simplest case of our general framework. This corresponds to the (one-dimensional) SVIE \eqref{intro_eq_SVIE} with $b=0$ and $\sigma=1$, and the associated lifted SEE becomes an Ornstein--Uhlenbeck process on the Hilbert space. In this case, we characterize the invariant probability measure for the associated lifted SEE and construct a stationary process for the original SVIE. The existence and uniqueness of the invariant probability measure are completely characterized by the integrability condition of the kernel $K$; see \cref{Gauss_theo_IPM}. Concerning the general case with the coefficients having non-linear dependence on the current state, we focus on the so-called (asymptotic) log-Harnack inequality and show some important consequent properties including the uniqueness of the invariant probability measure for the lifted SEE.

The dimension-free Harnack inequality was initiated by Wang \cite{Wa97} for elliptic diffusion semigroups on Riemannian manifolds, and he also introduced the log-Harnack inequality in \cite{Wa10}. For a Markov semigroup $\{P_t\}_{t\geq0}$ on the class $\cB_\mathrm{b}(E)$ of real-valued bounded Borel measurable functions on a Polish space $E$, the log-Harnack inequality is of type
\begin{equation*}
	P_t\log f(\bar{x})\leq\log P_tf(x)+\Phi(x,\bar{x})
\end{equation*}
for any $t>0$, $x,\bar{x}\in E$ and any $f\in\cB_\mathrm{b}(E)$ such that $f\geq1$, where $\Phi$ is a nonnegative measurable function on $E\times E$ such that $\lim_{\bar{x}\to x}\Phi(x,\bar{x})=\lim_{\bar{x}\to x}\Phi(\bar{x},x)=0$ for any $x\in E$. This inequality implies gradient estimates (hence, the strong Feller property), heat kernel estimates, the uniqueness of invariant probability measures, and the irreducibility of the Markov semigroup $\{P_t\}_{t\geq0}$; see the book \cite{Wa13} for more detailed theory of Harnack inequalities and applications to SDEs, SPDEs and stochastic functional (partial) differential equations. Unfortunately, it turns out that the strong Feller property is invalid for the lifted SEE associated with a Gaussian Volterra process, which is the simplest example of the SVIE \eqref{intro_eq_SVIE}, and hence the log-Harnack inequality does not hold in our framework; see \cref{Gauss_theo_strong-Feller}. A main difficulty of our framework is that the lifted SEE is \emph{highly degenerate}; the state space is (typically) infinite dimensional, while the underlying Brownian motion is finite dimensional. Similar difficulties arise in SPDEs with degenerate noise such as degenerate stochastic 2D Navier--Stokes equations. In order to prove the ergodicity for such degenerate stochastic systems which do not satisfy the strong Feller property, Hairer and Mattingly \cite{HaMa06} introduced the notion of the \emph{asymptotic strong Feller property} which is weaker than the strong Feller property. Then, as an alternative to the log-Harnack inequality, Xu \cite{Xu11} and Bao, Wang and Yuan \cite{BaWaYu19} investigated the \emph{asymptotic log-Harnack inequality} (which was introduced in \cite{Xu11} by the name ``modified log-Harnack inequality'', whereas the ``asymptotic log-Harnack inequality'' was named by \cite{BaWaYu19} taking into account the long-time asymptotic behaviour of the inequality). For a Markov semigroup $\{P_t\}_{t\geq0}$ on $\cB_\mathrm{b}(E)$, the asymptotic log-Harnack inequality is of type
\begin{equation*}
	P_t\log f(\bar{x})\leq\log P_tf(x)+\Phi(x,\bar{x})+\Psi_t(x,\bar{x})\|\nabla\log f\|_\infty
\end{equation*}
for any (sufficiently large) $t>0$, $x,\bar{x}\in E$ and any $f\in\cB_\mathrm{b}(E)$ such that $f\geq1$ and $\|\nabla\log f\|_\infty<\infty$. Here, $\|\nabla\log f\|_\infty$ represents the Lipschitz constant of $\log f$ with respect to a compatible metric on $E$, and $\Phi,\Psi_t:E\times E\to[0,\infty)$ are measurable maps such that $\lim_{t\to\infty}\Psi_t(x,\bar{x})=0$ for any $x,\bar{x}\in E$. Under some additional conditions for $\Phi$ and $\Psi$, the asymptotic log-Harnack inequality implies asymptotic gradient estimates (hence, the asymptotic strong Feller property), asymptotic heat kernel estimates, the uniqueness of invariant probability measures, and the asymptotic irreducibility of the Markov semigroup $\{P_t\}_{t\geq0}$; see \cite[Theorem 2.1]{BaWaYu19}. By means of the asymptotic coupling method, the asymptotic log-Harnack inequalities were established for the 2D stochastic Navier--Stokes equation driven by highly degenerate but essentially elliptic noise \cite{Xu11}, SDEs with infinite memory \cite{BaWaYu19,WaWuYiZh22}, SPDEs with degenerate multiplicative noise \cite{HoLiLi20}, monotone SPDEs with multiplicative noise \cite{Li20}, the stochastic Cahn--Hilliard equation with the logarithmic free energy with highly degenerate but essentially elliptic noise \cite{GoXi20}, the 3D stochastic Leray-$\alpha$ model with degenerate noise \cite{HoLiLi21a}, and 2D hydrodynamical-type systems with degenerate noise \cite{HoLiLi21b}.

In this paper, following the asymptotic coupling method adopted in the aforementioned papers, we derive an asymptotic log-Harnack inequality for the Hilbert-valued Markovian lift of the SVIE \eqref{intro_eq_SVIE}. We note that, although the lifted SEE can be seen as a kind of monotone SPDEs, it is beyond the frameworks of previous works mentioned above; see \cref{Harnack_rem_previous}. Thus, the construction of an asymptotic coupling is a non-trivial task. By means of the special structure of the lifted SEE, we succeed to obtain a positive result.

We have to point out that, however, our results on the asymptotic log-Harnack inequality for the lifted SEE (\cref{Harnack_theo_Harnack} and \cref{Harnack_cor_application}) are informative only for the case $\beta:=\inf\supp>0$. Loosely speaking, the number $\beta$ corresponds to the exponential decay rate of the kernel $K(t)$ as $t\to\infty$. On the one hand, our results are valid for the case of the Gamma kernel $K(t)=\frac{1}{\Gamma(\alpha)}e^{-\beta t}t^{\alpha-1}$ with $\alpha\in(\frac{1}{2},1)$ and $\beta>0$, which is a kind of singular (as $t\downarrow0$) kernel with the exponential damping property (as $t\to\infty$). On the other hand, the condition $\beta:=\inf\supp>0$ excludes the case of the fractional kernel (which is the Gamma kernel with $\beta=0$). We need further considerations for the case $\beta=0$, which we leave to the future research. Also, in this paper, we do not study the existence of the invariant probability measure for the lifted SEE associated with a general SVIE, and we leave this important topic to the future research as well.

The rest of the paper is organized as follows. In \cref{lift}, we introduce our new framework of Hilbert-valued Markovian lifting of SVIEs. We investigate fundamental properties of the state space of Markovian lifts in \cref{lift-space}. Then, we show the equivalence between SVIEs and lifted SEEs in \cref{lift-equivalence}, the well-posedness of SEEs in \cref{lift-well-posedness} (whose proofs are in \hyperref[appendix]{Appendix}), and the Markov property for the lifted SEE in \cref{lift-Markov}. \cref{Gauss} is devoted to studies on Gaussian Volterra processes and their Markovian lifts; we characterize the invariant probability measures for the lifted SEEs and show that the strong Feller property does not hold in typical cases. In \cref{Harnack}, we derive an asymptotic log-Harnack inequality for the lifted SEE. \hyperref[appendix]{Appendix} contains complete proofs of the results in \cref{lift-well-posedness}.

\subsection*{Notations}

$(\Omega,\cF,\bF,\bP)$ is a complete filtered probability space, where $\bF=(\cF_t)_{t\geq0}$ is a filtration satisfying the usual conditions. $W$ is a $d$-dimensional Brownian motion on $(\Omega,\cF,\bF,\bP)$ with $d\in\bN$. $\bE$ denotes the expectation under $\bP$. The expectation under another probability measure $\bQ$ on $(\Omega,\cF)$ is denoted by $\bE_\bQ$. For each set $A$, $\1_A$ denotes the indicator function. For a topological space $E$, we denote by $\cB(E)$ the Borel $\sigma$-algebra on $E$.

We fix $n\in\bN$, which represents the dimension of the solutions of SVIEs considered in this paper. $\bR^n$ is the Euclidean space of $n$-dimensional vectors, and we denote by $|v|$ and $\langle v_1,v_2\rangle$ the usual Euclidean norm and the inner product for vectors $v,v_1,v_2\in\bR^n$. $\bR^{n\times d}$ is the space of $(n\times d)$-matrices endowed with the Frobenius norm $|A|:=\sqrt{\mathrm{tr}\,(A A^\top)}$ for $A\in\bR^{n\times d}$, where $A^\top\in\bR^{d\times n}$ denotes the transpose of $A$, and $\mathrm{tr}(\cdot)$ denotes the trace.

For a Banach space $E$, $E^*$ denotes the continuous dual space of $E$, and ${}_{E^*}\langle\cdot,\cdot\rangle_E$ denotes the duality pairing. For a bounded linear operator $T:E_1\to E_2$ from a Banach space $E_1$ to another Banach space $E_2$, $T^*:E^*_2\to E^*_1$ denotes the adjoint operator.

We sometimes use the following notations without further explanations:
\begin{itemize}
\item
For two separable Hilbert spaces $H_1$ and $H_2$, we denote by $L_2(H_1;H_2)$ the space of Hilbert--Schmidt operators from $H_1$ to $H_2$ endowed with the Hilbert--Schmidt norm $\|\cdot\|_{L_2(H_1;H_2)}$.
\item
For each $T>0$, $L^2(0,T;\bR^n)$ denotes the space of equivalent classes of measurable and square-integrable functions $\varphi:[0,T]\to\bR^n$ endowed with the norm $\|\varphi\|_{L^2(0,T;\bR^n)}:=\big(\int^T_0|\varphi(t)|^2\,\diff t\big)^{1/2}$. The space $L^2(0,T;\bR^{n\times d})$ is defined by the same manner.
\item
For each $T>0$, $L^2_\bF(0,T;\bR^n)$ denotes the space of equivalent classes of predictable and $\bP\otimes\diff t$-square-integrable stochastic processes $\varphi:\Omega\times[0,T]\to\bR^n$ endowed with the norm $\|\varphi\|_{L^2_\bF(0,T;\bR^n)}:=\bE\big[\int^T_0|\varphi(t)|^2\,\diff t\big]^{1/2}$. The space $L^2_\bF(0,T;\bR^{n\times d})$ is defined by the same manner.
\item
For each $\sigma$-field $\cG\subset\cF$ and a separable Hilbert space $(H,\|\cdot\|_H)$, $L^2_\cG(H)$ denotes the space of equivalent classes of $\cG/\cB(H)$-measurable and $\bP$-square-integrable random variables $\varphi:\Omega\to H$ endowed with the norm $\|\varphi\|_{L^2_\cG(H)}:=\bE[\|\varphi\|^2_H]^{1/2}$.
\item
For a Borel measure $\mu$ on $[0,\infty)$ and $p\geq1$, $L^p(\mu)$ denotes the space of equivalent classes of $\bR^n$-valued measurable functions $\varphi$ on $[0,\infty)$ such that $\int_\supp|\varphi(\theta)|^p\dmu<\infty$ endowed with the norm $\|\varphi\|_{L^p(\mu)}:=\big(\int_\supp|\varphi(\theta)|^p\dmu\big)^{1/p}$. Here, $\supp$ denotes the support of $\mu$.
\end{itemize}
Other notations will be defined later.

%%%%%%%%%%%%%%%%%%%%%%%%%%%%%%%%%%
%%%%%%%%%%%%%%%%%%%%%%%%%%%%%%%%%%
%% Section
%%%%%%%%%%%%%%%%%%%%%%%%%%%%%%%%%%
%%%%%%%%%%%%%%%%%%%%%%%%%%%%%%%%%%

\section{Hilbert-valued Markovian lifting of SVIEs}\label{lift}

In this section, we introduce a suitable framework of Markovian lifts of SVIEs. We first analyse some basic properties of completely monotone kernels. Let $K:(0,\infty)\to[0,\infty)$ be a completely monotone kernel. By Bernstein's theorem, there exists a unique Radon measure $\mu$ on $[0,\infty)$ such that $K(t)=\int_{[0,\infty)}e^{-\theta t}\dmu$ for any $t>0$. We say that the kernel $K$ is \emph{regular} if $\lim_{t\downarrow0}K(t)<\infty$ and that $K$ is \emph{singular} if it is not regular.

%% Lemma

\begin{lemm}
Let $K:(0,\infty)\to[0,\infty)$ be a completely monotone kernel, and let $\mu$ be the corresponding Radon measure on $[0,\infty)$. Then the following assertions hold:
\begin{itemize}
\item[(i)]
$K$ is regular if and only if $\mu([1,\infty))<\infty$.
\item[(ii)]
Let $\alpha\in(0,1]$. Then $\int^1_0t^{\alpha-1}K(t)\,\diff t<\infty$ if and only if $\int_{[1,\infty)}\theta^{-\alpha}\dmu$. Furthermore, these conditions imply $\int^1_0K(t)^{1/\alpha}\,\diff t<\infty$.
\end{itemize}
\end{lemm}

%% Proof

\begin{proof}
The assertion (i) is clear. We show (ii). Suppose $\int_{[1,\infty)}\theta^{-a}\dmu<\infty$. Then
\begin{align*}
	\int^1_0t^{\alpha-1}K(t)\,\diff t&=\int^1_0t^{\alpha-1}\Big(\int_{[0,\infty)}e^{-\theta t}\dmu\Big)\,\diff t=\int_{[0,\infty)}\Big(\int^1_0t^{\alpha-1}e^{-\theta t}\,\diff t\Big)\dmu\\
	&\leq\alpha^{-1}\mu([0,1))+\Gamma(\alpha)\int_{[1,\infty)}\theta^{-\alpha}\dmu<\infty,
\end{align*}
and, by Minkowski's integral inequality,
\begin{align*}
	\Big(\int^1_0K(t)^{1/\alpha}\,\diff t\Big)^\alpha&=\Big(\int^1_0\Big(\int_\supp e^{-\theta t}\dmu\Big)^{1/\alpha}\,\diff t\Big)^\alpha\\
	&\leq\int_\supp\Big(\int^1_0e^{-\theta t/\alpha}\,\diff t\Big)^\alpha\dmu\\
	&\leq\mu([0,1))+\int_{[1,\infty)}\theta^{-\alpha}\dmu<\infty.
\end{align*}
Conversely, suppose $\int^1_0t^{\alpha-1}K(t)\,\diff t<\infty$. Then
\begin{align*}
	\int_{[1,\infty)}\theta^{-\alpha}\dmu&=\frac{1}{\Gamma(a)}\int_{[1,\infty)}\Big(\int^\infty_0t^{\alpha-1}e^{-\theta t}\,\diff t\Big)\dmu=\frac{1}{\Gamma(\alpha)}\int^\infty_0t^{\alpha-1}\underbrace{\Big(\int_{[1,\infty)}e^{-\theta t}\dmu\Big)}_{\leq K(t)\wedge(e^{-t/2}K(t/2))}\,\diff t\\
	&\leq\frac{1}{\Gamma(\alpha)}\int^1_0t^{\alpha-1}K(t)\,\diff t+\frac{K(1/2)}{\Gamma(\alpha)}\int^\infty_1t^{\alpha-1}e^{-t/2}\,\diff t<\infty.
\end{align*}
This completes the proof.
\end{proof}

The above lemma shows that the regularity/singularity of the kernel $K$ is characterized by the lightness/heaviness of the tail of the measure $\mu$. On the other hand, the value of $\inf\supp\geq0$ characterizes the asymptotic behaviour of $K(t)$ as $t\to\infty$. Indeed, for each $\beta\geq0$,
\begin{equation*}
	e^{\beta t}K(t)=\int_\supp e^{-(\theta-\beta)t}\dmu\to
	\begin{cases}
	\infty\ &\text{if $\beta>\inf\supp$,}\\
	\mu(\{\beta\})\ &\text{if $\beta=\inf\supp$,}\\
	0\ &\text{if $\beta<\inf\supp$,}
	\end{cases}
\end{equation*}
as $t\to\infty$. In particular, if $\inf\supp>0$, then $K(t)$ decays exponentially fast as $t\to\infty$.

With the above observations in mind, throughout this paper, we impose the following assumption on the kernel $K$ appearing in the SVIE \eqref{intro_eq_SVIE}.

%% Assumption

\begin{assum}\label{lift_assum_kernel}
$K:(0,\infty)\to[0,\infty)$ is completely monotone. Furthermore, the corresponding Radon measure $\mu$ on $[0,\infty)$ satisfies
\begin{equation*}
	\int_\supp r(\theta)\dmu<\infty
\end{equation*}
for some non-increasing Borel measurable function $r:[0,\infty)\to(0,\infty)$ such that $1\wedge(\theta^{-1/2})\leq r(\theta)\leq1$ for any $\theta\in[0,\infty)$.
\end{assum}

%% Remark

\begin{rem}
The function $r$ determines the regularity/singularity of the kernel $K$. In particular, the above assumption implies $\int^1_0t^{-1/2}K(t)\,\diff t<\infty$ and $\int^1_0K(t)^2\,\diff t<\infty$.
\end{rem}

Here, we give some important examples of kernels $K$ which satisfy \cref{lift_assum_kernel}.

%% Example

\begin{exam}\label{lift_exam_kernel}
\begin{itemize}
\item[(i)]
Let $K_{\exp}$ be a finite combination of exponential functions:
\begin{equation*}
	K_{\exp}(t):=\sum^k_{i=1}c_i e^{-\theta_i t},\ t>0,
\end{equation*}
with $k\in\bN$, $c_i>0$ and $\theta_i\in[0,\infty)$ for $i=1,\dots,k$. Then $K_{\exp}$ is completely monotone, and the corresponding Radon measure is
\begin{equation*}
	\mu_{\exp}(\diff\theta)=\sum^k_{i=1}c_i\delta_{\theta_i}(\diff\theta),\ \supp_{\exp}=\{\theta_1,\dots,\theta_k\},
\end{equation*}
where $\delta_{\theta_i}$ denotes the Dirac measure at point $\theta_i$. Clearly, $K_{\exp}$ is regular, and we can take $r\equiv1$ in \cref{lift_assum_kernel}.
\item[(ii)]
Let $K_\mathrm{frac}$ be the fractional kernel of order $\alpha\in(\frac{1}{2},1)$:
\begin{equation*}
	K_\mathrm{frac}(t):=\frac{1}{\Gamma(\alpha)}t^{\alpha-1},\ t>0.
\end{equation*}
It is well-known that $K_\mathrm{frac}$ is completely monotone, and the corresponding Radon measure is
\begin{equation*}
	\mu_\mathrm{frac}(\theta)=\frac{1}{\Gamma(\alpha)\Gamma(1-\alpha)}\theta^{-\alpha}\,\diff\theta,\ \supp_\mathrm{frac}=[0,\infty).
\end{equation*}
The fractional kernel $K_\mathrm{frac}$ is singular, and we can take $r(\theta)=1\wedge(\theta^{-1/p})$ in \cref{lift_assum_kernel} for any $p\in[2,\frac{1}{1-\alpha})$.
\item[(iii)]
Let $K_\mathrm{gamma}$ be the Gamma kernel of the form
\begin{equation*}
	K_\mathrm{gamma}(t):=\frac{1}{\Gamma(\alpha)}e^{-\beta t}t^{\alpha-1},\ t>0,
\end{equation*}
for some $\beta>0$ and $\alpha\in(\frac{1}{2},1)$. $K_\mathrm{gamma}$ is completely monotone, and the corresponding Radon measure is
\begin{equation*}
	\mu_\mathrm{gamma}(\diff\theta)=\frac{1}{\Gamma(\alpha)\Gamma(1-\alpha)}(\theta-\beta)^{-\alpha}\1_{(\beta,\infty)}\diff\theta,\ \supp_\mathrm{gamma}=[\beta,\infty).
\end{equation*}
The Gamma kernel $K_\mathrm{gamma}$ is singular, and we can take $r(\theta)=1\wedge(\theta^{-1/p})$ in \cref{lift_assum_kernel} for any $p\in[2,\frac{1}{1-\alpha})$.
\item[(iv)]
More generally, for any $\beta>0$ and any kernel $K$ satisfying \cref{lift_assum_kernel}, the exponentially damped kernel
\begin{equation*}
	K_\mathrm{damp}(t):=e^{-\beta t}K(t),\ t>0,
\end{equation*}
satisfies \cref{lift_assum_kernel} with the same function $r$. The corresponding Radon measure is
\begin{equation*}
	\mu_\mathrm{damp}(\diff\theta)=\1_{[\beta,\infty)}(\theta)\mu(\diff\theta-\beta),\ \supp_\mathrm{damp}=\{\theta\in[\beta,\infty)\,|\,\theta-\beta\in\supp\}.
\end{equation*}
\item[(v)]
For any $\delta>0$ and any completely monotone kernel $K$ with the corresponding Radon measure $\mu$, the sifted kernel
\begin{equation*}
	K_\mathrm{shift}(t):=K(t+\delta),\ t>0,
\end{equation*}
satisfies \cref{lift_assum_kernel} with the function $r\equiv1$. In particular, the kernel $K_\mathrm{shift}$ is regular. The corresponding Radon measure is
\begin{equation*}
	\mu_\mathrm{shift}(\diff\theta)=e^{-\delta\theta}\dmu,\ \supp_\mathrm{shift}=\supp.
\end{equation*}
\end{itemize}
\end{exam}

Let \cref{lift_assum_kernel} hold. Consider the SVIE \eqref{intro_eq_SVIE} on $\bR^n$, which we write here again:
\begin{equation}\label{lift_eq_SVIE}
	X_t=x(t)+\int^t_0K(t-s)b(X_s)\,\diff s+\int^t_0K(t-s)\sigma(X_s)\,\diff W_s,\ t>0.
\end{equation}
Consider the case of $x(t)=(\cK y)(t)=\int_\supp e^{-\theta t}y(\theta)\dmu$ for some function $y:\supp\to\bR^n$. As a lift of the above SVIE, we introduce the following SEE defined on some function space:
\begin{equation}\label{lift_eq_SEE}
	\begin{dcases}
	\diff Y_t(\theta)=-\theta Y_t(\theta)\,\diff t+b(\mu[Y_t])\,\diff t+\sigma(\mu[Y_t])\,\diff W_t,\ t>0,\ \theta\in\supp,\\
	Y_0(\theta)=y(\theta),\ \theta\in\supp,
	\end{dcases}
\end{equation}
where $\mu[\eta]\text{``$=$''}\int_\supp \eta(\theta)\,\dmu$ for a suitable function $\eta:\supp\to\bR^n$; see \eqref{lift_eq_mu-map} for the precise definition of the map $\mu[\cdot]$. We introduce a suitable state space of the lift $Y$ in which the lifted SEE \eqref{lift_eq_SEE} is well-posed, and rigorously prove the equivalence of the SVIE \eqref{lift_eq_SVIE} and the lifted SEE \eqref{lift_eq_SEE}. Furthermore, we show that the solution $Y$ of the lifted SEE \eqref{lift_eq_SEE} is Markovian in the proposed state space.

%%%%%%%%%%%%%%
%% Subsection
%%%%%%%%%%%%%%

\subsection{State space of Markovian lifts}\label{lift-space}

First, we introduce the state space of the lifted process $Y$. Under \cref{lift_assum_kernel}, for each Borel measurable function $y:\supp\to\bR^n$, define
\begin{equation*}
	\|y\|_{\cH_\mu}:=\Big(\int_\supp r(\theta)|y(\theta)|^2\dmu\Big)^{1/2}
\end{equation*}
and
\begin{equation*}
	\|y\|_{\cV_\mu}:=\Big(\int_\supp (\theta+1)r(\theta)|y(\theta)|^2\dmu\Big)^{1/2}.
\end{equation*}
We denote by $\cH_\mu$ the space of equivalent classes of Borel measurable functions $y:\supp\to\bR^n$ such that $\|y\|_{\cH_\mu}<\infty$. Similarly, we denote by $\cV_\mu$ the space of equivalent classes of Borel measurable functions $y:\supp\to\bR^n$ such that $\|y\|_{\cV_\mu}<\infty$. Then both $(\cH_\mu,\|\cdot\|_{\cH_\mu})$ and $(\cV_\mu,\|\cdot\|_{\cV_\mu})$ are separable reflexive Banach spaces. Furthermore, we equip $\cH_\mu$ with the natural inner product
\begin{equation*}
	\langle y_1, y_2\rangle_{\cH_\mu}:=\int_\supp r(\theta)\langle y_1(\theta),y_2(\theta)\rangle\dmu,\ y_1,y_2\in\cH_\mu,
\end{equation*}
hence $(\cH_\mu,\|\cdot\|_{\cH_\mu},\langle\cdot,\cdot\rangle_{\cH_\mu})$ is a separable Hilbert space. It is easy to see that $\cV_\mu\subset\cH_\mu$ continuously and densely. Denote the continuous dual spaces of $\cH_\mu$ and $\cV_\mu$ by $\cH^*_\mu$ and $\cV^*_\mu$, respectively. By means of the Riesz isomorphism, we identify $\cH_\mu$ with $\cH^*_\mu$. Then we see that $\cH_\mu\subset\cV^*_\mu$ continuously and densely. The duality pairing $\dual{\cdot}{\cdot}$ is then compatible with the inner product on $\cH_\mu$ in the sense that $\dual{y'}{y}=\langle y',y\rangle_{\cH_\mu}$ whenever $y\in\cV_\mu\subset\cH_\mu$ and $y'\in\cH_\mu\subset\cV^*_\mu$. From this compatibility, it can be easily shown that $(\cV^*_\mu,\|\cdot\|_{\cV^*_\mu})$ is identified with the space of equivalent classes of Borel measurable functions $y:\supp\to\bR^n$ such that $\|y\|_{\cV^*_\mu}<\infty$, where
\begin{equation*}
	\|y\|_{\cV^*_\mu}=\Big(\int_\supp (\theta+1)^{-1}r(\theta)|y(\theta)|^2\dmu\Big)^{1/2}.
\end{equation*}
By the above constructions, we obtain a Gelfand triplet of Hilbert spaces:
\begin{equation*}
	\cV_\mu\hookrightarrow\cH_\mu\hookrightarrow\cV^*_\mu.
\end{equation*}
We adopt the Hilbert space $\cH_\mu$ (or the Gelfand triplet $\cV_\mu\hookrightarrow\cH_\mu\hookrightarrow\cV^*_\mu$) as the state space of Markovian lifts of SVIEs.

%% Remark

\begin{rem}
\begin{itemize}
\item[(i)]
Although $\cV_\mu$ and $\cV^*_\mu$ become Hilbert spaces with suitable inner products, these Hilbert space-structures are unnecessary in this paper.
\item[(ii)]
As we mentioned in \hyperref[intro]{Introduction}, regarding the state space of Markovian lifts of SVIEs, Carmona and Coutin \cite{CaCu98} and Abi Jaber, Miller and Pham \cite{AbiJaMiPh21} adopted the space $L^1(\mu)$, while Cuchiero and Teichmann \cite{CuTe19,CuTe20} adopted the space of signed (or vector-valued) Borel measures. In their approaches, the (stochastic) calculus of lifted equations are difficult due to the lack of Hilbert space-structure. In contrast, it will be turned out that our new framework fits into the well-established framework of SPDEs and It\^{o}'s calculus in Hilbert spaces.  We remark that, unlike \cite{AbiJaMiPh21,CaCu98,CuTe19,CuTe20}, the definitions of the spaces $\cH_\mu$, $\cV_\mu$ and $\cV^*_\mu$ take into account the singularity/regularity of the kernel via the weight function $r$.
\end{itemize}
\end{rem}

Here, we show some important properties of the function spaces defined above.

%% Lemma

\begin{lemm}\label{lift_lemm_space}
Let \cref{lift_assum_kernel} hold.
\begin{itemize}
\item[(i)]
Each constant vector $b\in\bR^n$ belongs to $\cH_\mu$ and satisfies $\|b\|_{\cH_\mu}=\big(\int_\supp r(\theta)\dmu\big)^{1/2}|b|$. Each constant matrix $\sigma\in\bR^{n\times d}$ belongs to the space $L_2(\bR^d;\cH_\mu)$ of Hilbert--Schmidt operators from $\bR^d$ to $\cH_\mu$ and satisfies $\|\sigma\|_{L_2(\bR^d;\cH_\mu)}=\big(\int_\supp r(\theta)\dmu\big)^{1/2}|\sigma|$.
\item[(ii)]
The map $e^{-\cdot t}:y\mapsto(\theta\mapsto e^{-\theta t}y(\theta))$ is a linear contraction operator on $\cH_\mu$ for any $t\geq0$. Furthermore, $\{e^{-\cdot t}\}_{t\geq0}$ is a $C_0$-semigroup on $\cH_\mu$ with the infinitesimal generator $A:\cD(A)\to\cH_\mu$ given by
\begin{equation*}
	\begin{dcases}
	\cD(A)=\left\{y\in\cH_\mu\relmiddle|\int_\supp \theta^2r(\theta)|y(\theta)|^2\dmu<\infty\right\},\\
	(Ay)(\theta)=-\theta y(\theta),\ \theta\in\supp,\ y\in\cD(A).
	\end{dcases}
\end{equation*}
The map $A:\cD(A)\to\cH_\mu\hookrightarrow\cV^*_\mu$ can be uniquely extended to a linear contraction operator (again denoted by $A$) from $\cV_\mu$ to $\cV^*_\mu$. Furthermore, we have
\begin{equation*}
	\dual{Ay}{y}=-\|y\|^2_{\cV_\mu}+\|y\|^2_{\cH_\mu}\leq0
\end{equation*}
for any $y\in\cV_\mu$.
\item[(iii)]
Any $y\in\cV_\mu$ is $\mu$-integrable and satisfies
\begin{equation*}
	\int_\supp |y(\theta)|\dmu\leq\Big(\int_\supp r(\theta)\dmu\Big)^{1/2}\|y\|_{\cV_\mu}.
\end{equation*}
Furthermore, for any $\ep>0$, there exists a constant $M=M(\mu,\ep)>0$ such that
\begin{equation*}
	\Big(\int_\supp |y(\theta)|\dmu\Big)^2\leq\ep\|y\|^2_{\cV_\mu}+M\|y\|^2_{\cH_\mu}
\end{equation*}
for any $y\in\cV_\mu$.
\end{itemize}
\end{lemm}

%% Proof

\begin{proof}
The assertion (i) is trivial from the definition.

For the proof of (ii), noting that $\supp\subset[0,\infty)$, we have
\begin{equation*}
	\|e^{-\cdot t}y\|_{\cH_\mu}=\Big(\int_\supp r(\theta)e^{-2\theta t}|y(\theta)|^2\dmu\Big)^{1/2}\leq\Big(\int_\supp r(\theta)|y(\theta)|^2\dmu\Big)^{1/2}=\|y\|_{\cH_\mu}
\end{equation*}
for any $y\in\cH_\mu$ and $t\geq0$. Thus, $e^{-\cdot t}$ is contractive on $\cH_\mu$ for any $t\geq0$. The semigroup property of $\{e^{-\cdot t}\}_{t\geq0}$ is trivial. For any $y\in\cH_\mu$, by the dominated convergence theorem,
\begin{equation*}
	\lim_{t\downarrow0}\|e^{-\cdot t}y-y\|_{\cH_\mu}=\lim_{t\downarrow 0}\Big(\int_\supp r(\theta)(e^{-\theta t}-1)^2|y(\theta)|^2\dmu\Big)^{1/2}=0.
\end{equation*}
Therefore, the semigroup $\{e^{-\cdot t}\}_{t\geq0}$ is strongly continuous. Now let $A:\cD(A)\to\cH_\mu$ be the infinitesimal generator of $\{e^{-\cdot t}\}_{t\geq0}$, and define $\bar{A}:\cD(\bar{A})\to\cH_\mu$ by
\begin{equation*}
	\begin{dcases}
	\cD(\bar{A}):=\left\{y\in\cH_\mu\relmiddle|\int_\supp \theta^2r(\theta)|y(\theta)|^2\dmu<\infty\right\},\\
	(\bar{A}y)(\theta):=-\theta y(\theta),\ \theta\in\supp,\ y\in\cD(\bar{A}).
	\end{dcases}
\end{equation*}
Let $y\in\cD(\bar{A})$. By the dominated convergence theorem, we see that
\begin{align*}
	\Big\|\frac{1}{t}\{e^{-\cdot t}y-y\}-\bar{A}y\Big\|^2_{\cH_\mu}&=\int_\supp r(\theta)\Big|\frac{1}{t}\{e^{-\theta t}y(\theta)-y(\theta)\}+\theta y(\theta)\Big|^2\dmu\\
	&=\int_\supp\Big\{1-\frac{1}{t}\int^t_0e^{-\theta s}\,\diff s\Big\}^2\theta^2r(\theta)|y(\theta)|^2\dmu\to0
\end{align*}
as $t\downarrow0$. Thus, $y\in\cD(A)$ and $\bar{A}y=Ay$. Conversely, for any $y\in\cD(A)$, by Fatou's lemma,
\begin{align*}
	\int_\supp\theta^2r(\theta)|y(\theta)|^2\dmu&=\int_\supp r(\theta)\lim_{t\downarrow0}\Big|\frac{1}{t}\{e^{-\theta t}y(\theta)-y(\theta)\}\Big|^2\dmu\\
	&\leq\liminf_{t\downarrow0}\int_\supp r(\theta)\Big|\frac{1}{t}\{e^{-\theta t}y(\theta)-y(\theta)\}\Big|^2\dmu\\
	&=\int_\supp r(\theta)|(Ay)(\theta)|^2\dmu<\infty.
\end{align*}
Hence, $y\in\cD(\bar{A})$. Therefore, $\cD(A)=\cD(\bar{A})$ and $Ay=\bar{A}y$ for any $y\in\cD(A)$. Furthermore, for any $y\in\cD(A)$,
\begin{align*}
	\|Ay\|_{\cV^*_\mu}&=\Big(\int_\supp(\theta+1)^{-1}r(\theta)|-\theta y(\theta)|^2\dmu\Big)^{1/2}\leq\Big(\int_\supp(\theta+1)r(\theta)|y(\theta)|^2\dmu\Big)^{1/2}=\|y\|_{\cV_\mu}.
\end{align*}
Since $\cD(A)$ is dense in $(\cV_\mu,\|\cdot\|_{\cV_\mu})$, the operator $A$ (more precisely, the composition of $A:\cD(A)\to\cH_\mu$ and the embedding $\cH_\mu\hookrightarrow\cV^*_\mu$) can be uniquely extended to a contraction operator $A:\cV_\mu\to\cV^*_\mu$. Furthermore, for any $y\in\cD(A)\subset\cV_\mu$, we have
\begin{align*}
	\dual{Ay}{y}&=\langle Ay,y\rangle_{\cH_\mu}\\
	&=-\int_\supp\theta r(\theta)|y(\theta)|^2\dmu\\
	&=-\int_\supp(\theta+1)r(\theta)|y(\theta)|^2\dmu+\int_\supp r(\theta)|y(\theta)|^2\dmu\\
	&=-\|y\|^2_{\cV_\mu}+\|y\|^2_{\cH_\mu}.
\end{align*}
By the density of $\cD(A)$ in $\cV_\mu$ and the continuity of $A:\cV_\mu\to\cV^*_\mu$, we get the last assertion of (ii).

We show (iii). For any $y\in\cV_\mu$, by the Cauchy--Schwarz inequality, we have
\begin{equation*}
	\int_\supp|y(\theta)|\dmu\leq\Big(\int_\supp(\theta+1)^{-1}r(\theta)^{-1}\dmu\Big)^{1/2}\Big(\int_\supp(\theta+1)r(\theta)|y(\theta)|^2\dmu\Big)^{1/2}.
\end{equation*}
By the assumption $r(\theta)\geq1\wedge(\theta^{-1/2})$ and the following observation:
\begin{equation*}
	(\theta+1)^{-1}r(\theta)^{-1}\leq(\theta+1)^{-1}(1\wedge(\theta^{-1/2}))^{-1}\leq1\wedge(\theta^{-1/2})\leq r(\theta),
\end{equation*}
we have
\begin{equation*}
	\int_\supp|y(\theta)|\dmu\leq\Big(\int_\supp r(\theta)\dmu\Big)^{1/2}\|y\|_{\cV_\mu}<\infty.
\end{equation*}
Similarly, for any $y\in\cV_\mu$ and any $m\geq1$,
\begin{align*}
	\Big(\int_\supp|y(\theta)|\dmu\Big)^2&\leq\int_\supp(\theta+m)^{-1}r(\theta)^{-1}\dmu\int_\supp(\theta+m)r(\theta)|y(\theta)|^2\dmu\\
	&=\int_\supp(\theta+m)^{-1}r(\theta)^{-1}\dmu\big(\|y\|^2_{\cV_\mu}+(m-1)\|y\|^2_{\cH_\mu}\big).
\end{align*}
Since $(\theta+m)^{-1}r(\theta)^{-1}\leq(\theta+1)^{-1}r(\theta)^{-1}\leq r(\theta)$ and $\int_\supp r(\theta)\dmu<\infty$, by the dominated convergence theorem, we see that $\lim_{m\to\infty}\int_\supp(\theta+m)^{-1}r(\theta)^{-1}\dmu=0$. Thus, for any $\ep>0$, there exists a constant $m=m(\mu,\ep)\geq1$ such that
\begin{equation*}
	\int_\supp(\theta+m)^{-1}r(\theta)^{-1}\dmu\leq\ep.
\end{equation*}
Then, by letting $M:=\ep(m-1)$, we get the assertion (iii).
\end{proof}

%% Remark

\begin{rem}\label{lift_rem_space}
\begin{itemize}
\item[(i)]
We do not need the conditions that $r$ is non-increasing or $r\leq1$ for \cref{lift_lemm_space} (and \cref{lift_lemm_convolution} below). Such conditions will be crucial in \cref{Harnack}.
\item[(ii)]
By \cref{lift_lemm_space} (ii), we have the continuous embeddings $\cV_\mu\hookrightarrow \cH^1_\mu\hookrightarrow\cH_\mu$, where $(\cH^1_\mu,\|\cdot\|_{\cH^1_\mu})$ is a separable Banach space defined by $\cH^1_\mu:=\cH_\mu\cap L^1(\mu)$ and $\|\cdot\|_{\cH^1_\mu}:=\|\cdot\|_{\cH_\mu}+\|\cdot\|_{L^1_\mu}$. Also, noting that $r(\theta)\leq 1\leq(\theta+1)r(\theta)$, we have the continuous embeddings $\cV_\mu\hookrightarrow L^2(\mu)\hookrightarrow \cH_\mu$. The spaces $\cH_\mu$ and $L^1(\mu)$ have no inclusion-relation in general. On the one hand, if the kernel $K$ is singular, or equivalently if $\mu([1,\infty))=\infty$, then we have $\cH_\mu\setminus L^1(\mu)\neq\emptyset$. Indeed, any non-zero constant vectors belong to $\cH_\mu\setminus L^1(\mu)$. On the other hand, if there exists a Borel set $B\subset\supp$ such that $\mu(B)>0$ and $\mu$ is atomless on $B$, then it can be easily shown that $L^1(\mu)\setminus\cV^*_\mu\neq\emptyset$, and in particular $L^1(\mu)\setminus\cH_\mu\neq\emptyset$.
\item[(iii)]
It is easy to see that the closed operator $A:\cD(A)\to\cH_\mu$ is self-adjoint and nonpositive. The sets of spectra, point spectra and continuous spectra of the nonnegative operator $-A$ are given by
\begin{align*}
	&\mathrm{spec}\,(-A)=\supp,\\
	&\mathrm{spec}_\mathrm{p}(-A)=\{\theta\in\supp\,|\,\mu(\{\theta\})>0\}\ \text{and}\\
	&\mathrm{spec}_\mathrm{c}(-A)=\{\theta\in\supp\,|\,\mu(\{\theta\})=0\},
\end{align*}
respectively.
\item[(iv)]
We remark on some measurability issues. Since $\cH_\mu$ is a separable Hilbert space, by the Pettis measurability theorem, a function $f:S\to\cH_\mu$ defined on a measurable space $(S,\cS)$ is (strongly) measurable (with respect to $\cS$ and $\cB(\cH_\mu)$) if and only if it is weakly measurable (i.e., $s\mapsto\langle y,f(s)\rangle_{\cH_\mu}$ is measurable for any $y\in\cH_\mu$). Furthermore, since $\cV_\mu\subset\cH_\mu\subset\cV^*_\mu$ are Polish spaces with continuous embeddings, the Lusin--Suslin theorem yields that $\cV_\mu\in\cB(\cH_\mu)$, $\cH_\mu\in\cB(\cV^*_\mu)$, and $\cB(\cV_\mu)=\{A\cap\cV_\mu\,|\,A\in\cB(\cH_\mu)\}$, $\cB(\cH_\mu)=\{A\cap\cH_\mu\,|\,A\in\cB(\cV^*_\mu)\}$. In particular, the restriction map $y\mapsto\1_{\cV_\mu}(y)y$ from $\cH_\mu$ to $\cV_\mu$ is $\cB(\cH_\mu)/\cB(\cV_\mu)$-measurable. Also, the map $\|\cdot\|_{\cV_\mu}:\cH_\mu\to[0,\infty]$ (which is the norm on $\cV_\mu$ and takes $\infty$ on $\cH_\mu\setminus\cV_\mu$) is $\cB(\cH_\mu)$-measurable. Similarly, we have $\cH^1_\mu\in\cB(\cH_\mu)$ and $\cB(\cH^1_\mu)=\{A\cap L^1(\mu)\,|\,A\in\cB(\cH_\mu)\}$, and the restriction map $y\mapsto\1_{\cH^1_\mu}(y)y$ from $\cH_\mu$ to $\cH^1_\mu$ is $\cB(\cH_\mu)/\cB(\cH^1_\mu)$-measurable.
\end{itemize}
\end{rem}

In our studies of the SVIE \eqref{lift_eq_SVIE} and the lifted SEE \eqref{lift_eq_SEE}, the integral $\int_\supp y(\theta)\dmu$ of a function $y:\supp\to\bR^n$ with respect to $\mu$ plays a key role. Recall that $(\cH^1_\mu,\|\cdot\|_{\cH^1_\mu})$ is a separable Banach space defined by $\cH^1_\mu:=\cH_\mu\cap L^1(\mu)$ and $\|\cdot\|_{\cH^1_\mu}:=\|\cdot\|_{\cH_\mu}+\|\cdot\|_{L^1(\mu)}$. Clearly, the integral operator $\mu_0[\cdot]:\cH^1_\mu\to\bR^n$ defined by
\begin{equation*}
	\mu_0[y]:=\int_\supp y(\theta)\dmu,\ y\in\cH^1_\mu,
\end{equation*}
is a bounded linear operator from $\cH^1_\mu$ to $\bR^n$, and in particular it is $\cB(\cH^1_\mu)/\cB(\bR^n)$-measurable. We extend this map to $\cH_\mu$ by the following way: For each $y\in\cH_\mu$, we define $\mu[y]\in\bR^n$ by
\begin{equation}\label{lift_eq_mu-map}
	\mu[y]:=\mu_0[\1_{\cH^1_\mu}(y)y]=
	\begin{dcases}
	\int_\supp y(\theta)\dmu\ &\text{if $y\in L^1(\mu)$},\\
	0\ &\text{if $y\notin L^1(\mu)$}.
	\end{dcases}
\end{equation}
By the $\cB(\cH_\mu)/\cB(\cH^1_\mu)$-measurability of the restriction map $y\mapsto\1_{\cH^1_\mu}(y)y$ (see \cref{lift_rem_space} (iv)), we see that the map $\mu[\cdot]:\cH_\mu\to\bR^n$ is $\cB(\cH_\mu)/\cB(\bR^n)$-measurable. In particular, for any $\cH_\mu$-valued predictable process $Y$, the ``integrated process'' $\mu[Y]=(\mu[Y_t])_{t\geq0}$ is an $\bR^n$-valued predictable process. On the one hand, since $\cV_\mu\subset\cH^1_\mu$ continuously (see \cref{lift_lemm_space} (iii)), we have $\mu[\cdot]=\mu_0[\cdot]$ on $\cV_\mu$, and it is a bounded linear operator from $\cV_\mu$ to $\bR^n$. On the other hand, the measurable map $\mu[\cdot]:\cH_\mu\to\bR^n$ itself is neither continuous nor linear on $\cH_\mu$ in general.
%If we can take $r\equiv1$ in \cref{lift_assum_kernel} (which is possible if and only if the kernel $K$ is regular), then $\cH_\mu=L^2(\mu)\hookrightarrow L^1(\mu)$, and thus the map $\mu:\cH_\mu\to\bR^n$ becomes a bounded linear operator from $\cH_\mu$ to $\bR^n$.

Next, for any $y\in\cH_\mu$, we define
\begin{equation}\label{lift_eq_K-map}
	(\cK y)(t):=\mu[e^{-\cdot t}y],\ t\geq0.
\end{equation}
Then $\cK y:[0,\infty)\to\bR^n$ is Borel measurable. Note that, for any $t>0$ and any $y\in\cH_\mu$, we have
\begin{align*}
	\|e^{-\cdot t}y\|^2_{\cV_\mu}&=\int_\supp(\theta+1)r(\theta)e^{-2\theta t}|y(\theta)|^2\dmu\\
	&\leq\int_\supp \frac{\theta+1}{2\theta t+1}r(\theta)|y(\theta)|^2\dmu\\
	&\leq\Big(\frac{1}{2t}\vee1\Big)\int_\supp r(\theta)|y(\theta)|^2\dmu\\
	&=\Big(\frac{1}{2t}\vee1\Big)\|y\|^2_{\cH_\mu}<\infty.
\end{align*}
Hence, $e^{-\cdot t}$ is a bounded linear operator from $\cH_\mu$ to $\cV_\mu$, and $(\cK y)(t)=\mu_0[e^{-\cdot t}y]=\int_\supp e^{-\theta t}y(\theta)\dmu$ for any $t>0$ and any $y\in\cH_\mu$. Furthermore, for any $T>0$,
\begin{align*}
	\int^T_0\|e^{-\cdot t}y\|^2_{\cV_\mu}\,\diff t&=\int^T_0\int_\supp(\theta+1)r(\theta)e^{-2\theta t}|y(\theta)|^2\dmu\,\diff t\\
	&\leq\int_\supp\Big(\frac{1}{2\theta}\wedge T\Big)(\theta+1)r(\theta)|y(\theta)|^2\dmu\\
	&\leq\Big(\frac{1}{2}+T\Big)\int_\supp r(\theta)|y(\theta)|^2\dmu\\
	&=\Big(\frac{1}{2}+T\Big)\|y\|^2_{\cH_\mu}<\infty.
\end{align*}
Thus, by \cref{lift_lemm_space} (iii), we see that the map $\cK:y\mapsto(\cK y)(\cdot)$ is a bounded linear operator from $\cH_\mu$ to $L^2(0,T;\bR^n)$ for any $T>0$.

The next lemma shows that $\bR^n$-valued ``Volterra processes'' are represented by the corresponding (stochastic) convolutions on $\cH_\mu$ via the map $\mu[\cdot]$.

%% Lemma

\begin{lemm}\label{lift_lemm_convolution}
Let \cref{lift_assum_kernel} hold. Let $T>0$ be fixed.
\begin{itemize}
\item[(i)]
Let $b:\Omega\times[0,T]\to\bR^n$ be a predictable process such that $\int^T_0|b_t|\,\diff t<\infty$ a.s. For each $t\in[0,T]$, consider the convolution
\begin{equation*}
	I^1_t:=\int^t_0e^{-\cdot(t-s)}b_s\,\diff s,
\end{equation*}
where the integral is defined as a Bochner integral on $\cH_\mu$. Then $I^1=(I^1_t)_{t\in[0,T]}$ is an $\cH_\mu$-valued continuous adapted process and satisfies
\begin{equation*}
	\sup_{t\in[0,T]}\|I^1_t\|_{\cH_\mu}\leq\Big(\int_\supp r(\theta)\dmu\Big)^{1/2}\int^T_0|b_t|\,\diff t<\infty\ \text{a.s.}
\end{equation*}
and
\begin{equation*}
	\int^T_0\|I^1_t\|^2_{\cV_\mu}\,\diff t\leq\Big(\frac{1}{2}+T\Big)\int_\supp r(\theta)\dmu\,\Big(\int^T_0|b_t|\,\diff t\Big)^2<\infty\ \text{a.s.}
\end{equation*}
Furthermore, it holds that
\begin{equation*}
	I^1_t\in\cV_\mu\ \ \text{and}\ \ \mu[I^1_t]=\int^t_0K(t-s)b_s\,\diff s\ \ \text{a.s.\ for a.e.\ $t\in[0,T]$},
\end{equation*}
and
\begin{equation*}
	\int^T_0\big|\mu[I^1_t]\big|^2\,\diff t\leq\int^T_0K(t)^2\,\diff t\,\Big(\int^T_0|b_t|\,\diff t\Big)^2<\infty\ \text{a.s.}
\end{equation*}
\item[(ii)]
Let $\sigma:\Omega\times[0,T]\to\bR^{n\times d}$ be a predictable process such that $\int^T_0|\sigma_t|^2\,\diff t<\infty$ a.s. For each $t\in[0,T]$, consider the stochastic convolution
\begin{equation*}
	I^2_t:=\int^t_0e^{-\cdot(t-s)}\sigma_s\,\diff W_s,
\end{equation*}
where the integral is defined as a stochastic integral on the Hilbert space $\cH_\mu$. Then $I^2=(I^2_t)_{t\in[0,T]}$ is an $\cH_\mu$-valued adapted process, has an $\cH_\mu$-continuous version (again denoted by $I^2$), and satisfies
\begin{equation*}
	\int^T_0\|I^2_t\|^2_{\cV_\mu}\,\diff t<\infty\ \text{a.s.}
\end{equation*}
Furthermore, it holds that
\begin{equation*}
	I^2_t\in\cV_\mu\ \ \text{and}\ \ \mu[I^2_t]=\int^t_0K(t-s)\sigma_s\,\diff  W_s\ \ \text{a.s.\ for a.e.\ $t\in[0,T]$.}
\end{equation*}
If furthermore $\bE\big[\int^T_0|\sigma_t|^2\,\diff t\big]<\infty$, we have
\begin{align*}
	&\bE\Big[\sup_{t\in[0,T]}\|I^2_t\|^2_{\cH_\mu}\Big]\leq38\int_\supp r(\theta)\dmu\,\bE\Big[\int^T_0|\sigma_t|^2\,\diff t\Big]<\infty,\\
	&\bE\Big[\int^T_0\|I^2_t\|^2_{\cV_\mu}\,\diff t\Big]\leq\Big(\frac{1}{2}+T\Big)\int_\supp r(\theta)\dmu\,\bE\Big[\int^T_0|\sigma_t|^2\,\diff t\Big]<\infty,\ \text{and}\\
	&\bE\Big[\int^T_0\big|\mu[I^2_t]\big|^2\,\diff t\Big]\leq\int^T_0K(t)^2\,\diff t\,\bE\Big[\int^T_0|\sigma_t|^2\,\diff t\Big]<\infty.
\end{align*}
\end{itemize}
\end{lemm}

%% Remark

\begin{rem}\label{lift_rem_measurable-convolution}
Under the same assumptions and notations as above, for any $\theta\in\supp$ and $t\in[0,T]$, we can define the integrals $\int^t_0e^{-\theta(t-s)}b_s\,\diff s$ and $\int^t_0e^{-\theta(t-s)}\sigma_s\,\diff W_s$ as Lebesgue/stochastic integrals on $\bR^n$. Furthermore, by \cite[Theorem 63 and its corollary]{Pr05}, there exist two maps $\tilde{I}^1,\tilde{I}^2:\Omega\times[0,T]\times\supp\to\bR^n$ such that
\begin{itemize}
\item
$\tilde{I}^1$ and $\tilde{I}^2$ are (jointly) measurable with respect to the product $\sigma$-algebra generated by the predictable $\sigma$-algebra on $\Omega\times[0,T]$ and the Borel $\sigma$-algebra on $\supp$;
\item
for any $\theta\in\supp$, the $\bR^n$-valued predictable processes $(\tilde{I}^1(\omega,t,\theta))_{t\in[0,T]}$ and $(\tilde{I}^2(\omega,t,\theta))_{t\in[0,T]}$ have continuous paths a.s.;
\item
for any $\theta\in\supp$ and any $t\in[0,T]$, $\tilde{I}^1(\omega,t,\theta)=\int^t_0e^{-\theta(t-s)}b_s\,\diff s$ and $\tilde{I}^2(\omega,t,\theta)=\int^t_0e^{-\theta(t-s)}\sigma_s\,\diff W_s$ a.s.
\end{itemize}
Furthermore, we have
\begin{equation*}
	I^1_t=\tilde{I}^1(\omega,t,\cdot)\ \text{and}\ I^2_t=\tilde{I}^2(\omega,t,\cdot)\ \text{in $\cH_\mu$ a.s. for any $t\in[0,T]$.}
\end{equation*}
Indeed, by the commutativity of bounded linear operators on $\cH_\mu$ and the Bochner/stochastic integrals on $\cH_\mu$, together with the (stochastic) Fubini theorem (cf.\ \cite{Ve12}), we have, for any $y\in\cH_\mu$,
\begin{align*}
	\langle y,I^1_t\rangle_{\cH_\mu}&=\int^t_0\langle y,e^{-\cdot(t-s)}b_s\rangle_{\cH_\mu}\,\diff s\\
	&=\int^t_0\int_\supp r(\theta)e^{-\theta(t-s)}\langle y(\theta),b_s\rangle\dmu\,\diff s\\
	&=\int_\supp r(\theta)\Big\langle y(\theta),\int^t_0e^{-\theta(t-s)}b_s\,\diff s\Big\rangle\dmu\\
	&=\langle y,\tilde{I}^1(\omega,t,\cdot)\rangle_{\cH_\mu}
\end{align*}
and
\begin{align*}
	\langle y,I^2_t\rangle_{\cH_\mu}&=\int^t_0\langle y,e^{-\cdot(t-s)}\sigma_s\,\diff W_s\rangle_{\cH_\mu}\\
	&=\sum^d_{j=1}\int^t_0\int_\supp r(\theta)e^{-\theta(t-s)}\langle y(\theta),\sigma^j_s\rangle\dmu\,\diff W^j_s\\
	&=\int_\supp r(\theta)\Big\langle y(\theta),\int^t_0e^{-\theta(t-s)}\sigma_s\,\diff W_s\Big\rangle\dmu\\
	&=\langle y,\tilde{I}^2(\omega,t,\cdot)\rangle_{\cH_\mu}
\end{align*}
a.s.\ for any $t\in[0,T]$. Since $\cH_\mu$ is a separable Hilbert space, we get the desired claims.
\end{rem}

%% Proof

\begin{proof}[Proof of \cref{lift_lemm_convolution}]
We show (i). The adaptedness of $I^1$ is trivial. By the strong continuity and contractivity of $\{e^{-\cdot t}\}_{t\geq0}$, it is easy to see that $I^1$ is $\cH_\mu$-continuous a.s. and satisfies
\begin{equation*}
	\sup_{t\in[0,T]}\|I^1_t\|_{\cH_\mu}\leq\int^T_0\|b_t\|_{\cH_\mu}\,\diff t=\Big(\int_\supp r(\theta)\dmu\Big)^{1/2}\,\int^T_0|b_t|\,\diff t<\infty\ \text{a.s.}
\end{equation*}
Furthermore,
\begin{align*}
	\int^T_0\|I^1_t\|^2_{\cV_\mu}\,\diff t&=\int^T_0\int_\supp(\theta+1)r(\theta)|I^1_t(\theta)|^2\dmu\,\diff t\\
	&\leq\int^T_0\int_\supp(\theta+1)r(\theta)\Big(\int^t_0e^{-\theta(t-s)}|b_s|\,\diff s\Big)^2\dmu\,\diff t\\
	&=\int_\supp(\theta+1)r(\theta)\int^T_0\Big(\int^t_0e^{-\theta(t-s)}|b_s|\,\diff s\Big)^2\,\diff t\dmu\\
	&\leq\int_\supp(\theta+1)r(\theta)\int^T_0e^{-2\theta t}\,\diff t\,\Big(\int^T_0|b_t|\,\diff t\Big)^2\dmu\\
	&\leq\Big(\frac{1}{2}+T\Big)\int_\supp r(\theta)\dmu\,\Big(\int^T_0|b_t|\,\diff t\Big)^2<\infty\ \text{a.s.},
\end{align*}
where we used Young's convolution inequality in the forth line. In particular, we have $I^1_t\in\cV_\mu$ for a.e.\ $t\in[0,T]$ a.s. By Tonelli's theorem, we have $I^1_t\in\cV_\mu$ a.s.\ for a.e.\ $t\in[0,T]$, and hence $\mu[I^1_t]=\int_\supp I^1_t(\theta)\dmu$ a.s.\ for a.e.\ $t\in[0,T]$. Furthermore, the Cauchy--Schwarz inequality and the above estimates yield that
\begin{align*}
	&\int^T_0\Big(\int_\supp \int^t_0e^{-\theta(t-s)}|b_s|\,\diff s\dmu\Big)^2\,\diff t\\
	&\leq\int_\supp(\theta+1)^{-1}r(\theta)^{-1}\dmu\int^T_0\int_\supp(\theta+1)r(\theta)\Big(\int^t_0e^{-\theta(t-s)}|b_s|\,\diff s\Big)^2\dmu\,\diff t<\infty\ \text{a.s.},
\end{align*}
and hence $\int_\supp\int^t_0e^{-\theta(t-s)}|b_s|\,\diff s\dmu<\infty$ a.s.\ for a.e.\ $t\in[0,T]$.
Thus, by Fubini's theorem, we have
\begin{equation*}
	\mu[I^1_t]=\int_\supp I^1_t(\theta)\dmu=\int^t_0\int_\supp e^{-\theta(t-s)}\dmu\,b_s\,\diff s=\int^t_0K(t-s)b_s\,\diff s
\end{equation*}
a.s.\ for a.e.\ $t\in[0,T]$. From this representation, by Young's convolution inequality, we obtain
\begin{equation*}
	\int^T_0\big|\mu[I^1_t]\big|^2\,\diff t=\int^T_0\Big|\int^t_0K(t-s)b_s\,\diff s\Big|^2\,\diff t\leq\int^T_0K(t)^2\,\diff t\,\Big(\int^T_0|b_t|\,\diff t\Big)^2\ \text{a.s.}
\end{equation*}
Since $\int^T_0K(t)^2\,\diff t<\infty$, which follows from the assumption $r(\theta)\geq1\wedge(\theta^{-1/2})$, the last assertion of (i) holds.

We show (ii). First, we assume that $\bE\big[\int^T_0|\sigma_t|^2\,\diff t\big]<\infty$. Then $I^2$ is well-defined as an $\cH_\mu$-valued adapted process. Noting that $\{e^{-\cdot t}\}_{t\geq0}$ is a contraction semigroup on $\cH_\mu$, by \cite[Theorem 6.10 and its proof]{DaPrZa14}\footnote{In \cite[the last line of page 167]{DaPrZa14}, we can replace the number $12$ by $3$, and then we can take $\ep=\frac{1}{6}$. As a consequence, we get $\bE\sup_{s\in[0,t]}|Z(s)|^2\leq38\bE\int^t_0\|\Phi(s)\|^2_{L^0_2}\,\diff s$ in their notations.}, $I^2$ has an $\cH_\mu$-continuous version (again denoted by $I^2$) and satisfies
\begin{equation*}
	\bE\Big[\sup_{t\in[0,T]}\|I^2_t\|^2_{\cH_\mu}\Big]\leq38\bE\Big[\int^T_0\|\sigma_t\|^2_{L_2(\bR^d;\cH_\mu)}\,\diff t\Big]=38\int_\supp r(\theta)\dmu\,\bE\Big[\int^T_0|\sigma_t|^2\,\diff t\Big]<\infty.
\end{equation*}
Furthermore,
\begin{align*}
	\bE\Big[\int^T_0\|I^2_t\|^2_{\cV_\mu}\,\diff t\Big]&=\bE\Big[\int^T_0\int_\supp(\theta+1)r(\theta)|I^2_t(\theta)|^2\dmu\,\diff t\Big]\\
	&=\bE\Big[\int^T_0\int_\supp(\theta+1)r(\theta)\int^t_0e^{-2\theta(t-s)}|\sigma_s|^2\,\diff s\dmu\,\diff t\Big]\\
	&\leq\bE\Big[\int_\supp(\theta+1)r(\theta)\int^T_0e^{-2\theta t}\,\diff t\,\int^T_0|\sigma_t|^2\,\diff t\dmu\Big]\\
	&\leq\Big(\frac{1}{2}+T\Big)\int_\supp r(\theta)\dmu\,\bE\Big[\int^T_0|\sigma_t|^2\,\diff t\Big]<\infty.
\end{align*}
In particular, $I^2_t\in\cV_\mu$ a.s.\ for a.e.\ $t\in[0,T]$, and hence $\mu[I^2_t]=\int_\supp I^2_t(\theta)\dmu$ a.s.\ for a.e.\ $t\in[0,T]$. Furthermore, noting that $1\leq(\theta+1)r(\theta)$, the above estimates yield that
\begin{align*}
	&\bE\Big[\int^T_0\int_\supp\int^t_0e^{-2\theta(t-s)}|\sigma_s|^2\,\diff s\dmu\,\diff t\Big]\\
	&\leq\bE\Big[\int^T_0\int_\supp(\theta+1)r(\theta)\int^t_0e^{-2\theta(t-s)}|\sigma_s|^2\,\diff s\dmu\,\diff t\Big]<\infty,
\end{align*}
and hence $\bE[\int_\supp\int^t_0e^{-2\theta(t-s)}|\sigma_s|^2\,\diff s\dmu]<\infty$ for a.e.\ $t\in[0,T]$. Thus, by the stochastic Fubini theorem (cf.\ \cite{Ve12}), we have
\begin{equation*}
	\mu[I^2_t]=\int_\supp I^2_t(\theta)\dmu=\int^t_0\int_\supp e^{-\theta(t-s)}\dmu\,\sigma_s\,\diff W_s=\int^t_0K(t-s)\sigma_s\,\diff W_s
\end{equation*}
a.s.\ for a.e.\ $t\in[0,T]$. From this representation, we obtain
\begin{align*}
	\bE\Big[\int^T_0\big|\mu[I^2_t]\big|^2\,\diff t\Big]&=\bE\Big[\int^T_0\Big|\int^t_0K(t-s)\sigma_s\,\diff W_s\Big|^2\,\diff t\Big]\\
	&=\bE\Big[\int^T_0\int^t_0K(t-s)^2|\sigma_s|^2\,\diff s\,\diff t\Big]\\
	&\leq\int^T_0K(t)^2\,\diff t\,\bE\Big[\int^T_0|\sigma_t|^2\,\diff t\Big]<\infty.
\end{align*}
Now we consider the general case. We can define the $\cH_\mu$-valued stochastic integral $I^2_t=\int^t_0e^{-\cdot(t-s)}\sigma_s\,\diff W_s$ for each $t\in[0,T]$. Define $\tau_k:=\inf\{t\in[0,T]\,|\,\int^t_0|\sigma_s|^2\,\diff s>k\}$ and $\sigma^k_s:=\sigma_s\1_{[0,\tau_k)}(s)$ for each $k\in\bN$. Then $\{\tau_k\}_{k\in\bN}$ is a sequence of stopping times and satisfies $\tau_k\leq\tau_{k+1}$ for any $k\in\bN$ and $\lim_{k\to\infty}\tau_k=T$ a.s. By the above observations, for each $k\in\bN$, the stochastic convolution $I^{2,k}_t:=\int^t_0e^{-\cdot(t-s)}\sigma^k_s\,\diff W_s$ is an $\cH_\mu$-valued adapted process and has an $\cH_\mu$-continuous version (againe denoted by $I^{2,k}$). It is easy to see that $I^{2,k}_t=I^{2,\ell}_t$ (in $\cH_\mu$) for any $t\in[0,\tau_k]$ and any $1\leq k\leq\ell$ a.s. Thus, the limit $I^{2,\infty}_t:=\lim_{k\to\infty}I^{2,k}_t$ (in $\cH_\mu$) exists for any $t\in[0,T]$ a.s., and $I^{2,\infty}=(I^{2,\infty}_t)_{t\in[0,T]}$ is $\cH_\mu$-continuous a.s. Furthermore, $I^2_t\1_{\{t\leq\tau_k\}}=I^{2,k}_t\1_{\{t\leq\tau_k\}}=I^{2,\infty}_t\1_{\{t\leq\tau_k\}}$ a.s.\ for any $t\in[0,T]$ and any $k\in\bN$, hence we have $I^2_t=I^{2,\infty}_t$ a.s.\ for any $t\in[0,T]$. Thus, $I^{2,\infty}$ is an $\cH_\mu$-continuous version of $I^2$. In the following, we take this version and write $I^2=I^{2,\infty}$. Since $I^2_t=I^{2,k}_t$ for any $t\in[0,\tau_k]$ and $\int^T_0\|I^{2,k}_t\|^2_{\cV_\mu}\,\diff t<\infty$ a.s.\ for any $k\in\bN$, it can be easily shown that $\int^T_0\|I^2_t\|^2_{\cV_\mu}\,\diff t<\infty$ a.s. In particular, $I^2_t\in\cV_\mu$ for a.e.\ $t\in[0,T]$ a.s. By Tonelli's theorem, we have $I^2_t\in\cV_\mu$ a.s.\ for a.e.\ $t\in[0,T]$. Furthermore, we have
\begin{equation*}
	\mu[I^2_t]\1_{\{t\leq\tau_k\}}=\mu[I^{2,k}_t]\1_{\{t\leq\tau_k\}}=\int^t_0K(t-s)\sigma^k_s\,\diff W_s\1_{\{t\leq\tau_k\}}=\int^t_0K(t-s)\sigma_s\,\diff W_s\1_{\{t\leq\tau_k\}}
\end{equation*}
for any $k\in\bN$ a.s.\ for a.e.\ $t\in[0,T]$. This implies that $\mu[I^2_t]=\int^t_0K(t-s)\sigma_s\,\diff W_s$ a.s.\ for a.e.\ $t\in[0,T]$. This completes the proof.
\end{proof}

In the following, we will always assume that the process $(\int^t_0e^{-\cdot(t-s)}\sigma_s\,\diff W_s)_{t\in[0,T]}$ is $\cH_\mu$-continuous for each predictable process $\sigma:\Omega\times[0,T]\to\bR^{n\times d}$ such that $\int^T_0|\sigma_t|^2\,\diff t<\infty$ a.s.

%%%%%%%%%%%%%%
%% Subsection
%%%%%%%%%%%%%%

\subsection{Equivalence between the SVIE \eqref{lift_eq_SVIE} and the lifted SEE \eqref{lift_eq_SEE}}\label{lift-equivalence}

We define the solutions of the SVIE \eqref{lift_eq_SVIE} and the lifted SEE \eqref{lift_eq_SEE}.

%% Definition

\begin{defi}\label{lift_defi_solution}
Suppose that \cref{lift_assum_kernel} holds. Let $b:\bR^n\to\bR^n$ and $\sigma:\bR^n\to\bR^{n\times d}$ be measurable maps.
\begin{itemize}
\item
Let $x:\Omega\times(0,\infty)\to\bR^n$ be $\cF_0\otimes\cB((0,\infty))$-measurable with $\int^T_0|x(t)|^2\,\diff t<\infty$ a.s.\ for any $T>0$. We say that an $\bR^n$-valued process $X$ is a \emph{solution} of the SVIE \eqref{lift_eq_SVIE} if $X$ is predictable, the integrability conditions
\begin{equation}\label{lift_eq_X-integrability}
	\int^T_0|X_t|^2\,\diff t<\infty,\ \int^T_0|b(X_t)|\,\diff t<\infty\ \text{and}\ \int^T_0|\sigma(X_t)|^2\,\diff t<\infty
\end{equation}
hold a.s.\ for any $T>0$, and the equality in \eqref{lift_eq_SVIE} holds a.s.\ for a.e.\ $t>0$. We say that the (pathwise) uniqueness holds for the SVIE \eqref{lift_eq_SVIE} if for any two solutions $X^1$and $X^2$ of the SVIE \eqref{lift_eq_SVIE}, it holds that $X^1_t=X^2_t$ a.s.\ for a.e.\ $t>0$.
\item
Let $y:\Omega\to\cH_\mu$ be $\cF_0$-measurable. We say that an $\cH_\mu$-valued process $Y$ is a \emph{mild solution} of the lifted SEE \eqref{lift_eq_SEE} if $Y$ is $\cH_\mu$-continuous and adapted, the integrability conditions
\begin{equation}\label{lift_eq_Y-integrability}
	\int^T_0\|Y_t\|^2_{\cV_\mu}\,\diff t<\infty,\ \int^T_0|b(\mu[Y_t])|\,\diff t<\infty\ \text{and}\ \int^T_0|\sigma(\mu[Y_t])|^2\,\diff t<\infty
\end{equation}
hold a.s.\ for any $T>0$, and the equality
\begin{equation}\label{lift_eq_mild-sol}
	Y_t(\theta)=e^{-\theta t}y(\theta)+\int^t_0e^{-\theta(t-s)}b(\mu[Y_s])\,\diff s+\int^t_0e^{-\theta(t-s)}\sigma(\mu[Y_s])\,\diff W_s
\end{equation}
holds for $\mu$-a.e.\ $\theta\in\supp$ a.s., for any $t\geq0$. We say that the (pathwise) uniqueness holds for the lifted SEE \eqref{lift_eq_SEE} if for any two mild solutions $Y^1$ and $Y^2$ of the lifted SEE \eqref{lift_eq_SEE}, it holds that $Y^1_t(\theta)=Y^2_t(\theta)$ for $\mu$-a.e.\ $\theta\in\supp$ a.s., for any $t\geq0$.
\end{itemize}
\end{defi}

%% Remark

\begin{rem}\label{lift_rem_solution}
\begin{itemize}
\item[(i)]
We would like to emphasize that the solution $X$ of the SVIE \eqref{lift_eq_SVIE} is defined only almost everywhere with respect to the time parameter $t$, and no path-regularity condition is imposed. In order to ensure the path-regularity of $X$, we need further regularity assumptions on the kernel $K$; see for example \cite{AbiJaCuLaPu21,AbiJaMiPh21}. On the other hand, the mild solution $Y$ of the SEE \eqref{lift_eq_SEE} is defined as an $\cH_\mu$-continuous process.
\item[(ii)]
Since the map $\mu:\cH_\mu\to\bR^n$ is Borel measurable, for any $\cH_\mu$-valued continuous adapted process $Y$, the process $\mu[Y]$ is an $\bR^n$-valued predictable process. If $\int^T_0\|Y_t\|^2_{\cV_\mu}\,\diff t<\infty$ a.s. for some $T>0$, then $Y_t\in\cV_\mu$ and $\mu[Y_t]=\int_\supp Y_t(\theta)\dmu$ a.s.\ for a.e.\ $t\in[0,T]$.
\item[(iii)]
The integrals in the right-hand side of \eqref{lift_eq_mild-sol} can be defined as integrals on $\cH_\mu$ as well as integrals on $\bR^n$ parametrized by $\theta$, and these two notions coincide on $\cH_\mu$, i.e., in the sense of $\mu$-a.e.\ $\theta\in\supp$; see \cref{lift_rem_measurable-convolution}. In particular, for a mild solution $Y$ of the lifted SEE \eqref{lift_eq_SEE}, there exists a jointly measurable map $\Omega\times[0,\infty)\times\supp\ni(\omega,t,\theta)\mapsto\tilde{Y}_t(\theta)(\omega)\in\bR^n$ (with respect to the predictable $\sigma$-algebra on $\Omega\times[0,\infty)$ and the Borel $\sigma$-algebra on $\supp$) such that $Y_t=\tilde{Y}_t$ in $\cH_\mu$ a.s.\ for any $t\geq0$ and that, for each $\theta\in\supp$, the process $\tilde{Y}(\theta)=(\tilde{Y}_t(\theta))_{t\geq0}$ is an $\bR^n$-valued It\^{o} process satisfying
\begin{equation*}
	\begin{dcases}
	\diff \tilde{Y}_t(\theta)=-\theta \tilde{Y}_t(\theta)\,\diff t+b(\mu[\tilde{Y}_t])\,\diff t+\sigma(\mu[\tilde{Y}_t])\,\diff W_t,\ t>0,\\
	\tilde{Y}_0(\theta)=y(\theta),
	\end{dcases}
\end{equation*}
in the usual It\^{o}'s sense.
\end{itemize}
\end{rem}

Now we state our first main result which shows the equivalence of the SVIE \eqref{lift_eq_SVIE} and the lifted SEE \eqref{lift_eq_SEE}. Recall the definitions of the maps $\mu[\cdot]$ and $\cK$ introduced in \eqref{lift_eq_mu-map} and \eqref{lift_eq_K-map}, respectively.

%% Theorem

\begin{theo}\label{lift_theo_equivalence}
Suppose that \cref{lift_assum_kernel} holds. Let $b:\bR^n\to\bR^n$ and $\sigma:\bR^n\to\bR^{n\times d}$ be measurable, and let $y:\Omega\to\cH_\mu$ be an $\cF_0$-measurable random variable. Then the following assertions hold:
\begin{itemize}
\item
If $X$ is a solution of the SVIE \eqref{lift_eq_SVIE} with the forcing term $x=\cK y$, then the $\cH_\mu$-valued process $Y$ defined by
\begin{equation*}
	Y_t:=e^{-\cdot t}y(\cdot)+\int^t_0e^{-\cdot(t-s)}b(X_s)\,\diff s+\int^t_0e^{-\cdot(t-s)}\sigma(X_s)\,\diff W_s,\ t\geq0,
\end{equation*}
is a mild solution of the lifted SEE \eqref{lift_eq_SEE} with the initial condition $y$. Furthermore, it holds that
\begin{equation*}
	X_t=\mu[Y_t]
\end{equation*}
a.s.\ for a.e.\ $t>0$.
\item
If $Y$ is a mild solution of the lifted SEE \eqref{lift_eq_SEE} with the initial condition $y$, then the $\bR^n$-valued process $X$ defined by
\begin{equation*}
	X_t:=\mu[Y_t],\ t>0,
\end{equation*}
is a solution of the SVIE \eqref{lift_eq_SVIE} with the forcing term $x=\cK y$. Furthermore, it holds that
\begin{equation*}
	Y_t=e^{-\cdot t}y(\cdot)+\int^t_0e^{-\cdot(t-s)}b(X_s)\,\diff s+\int^t_0e^{-\cdot(t-s)}\sigma(X_s)\,\diff W_s
\end{equation*}
a.s.\ for any $t\geq0$.
\end{itemize}
In particular, the uniqueness holds for the SVIE \eqref{lift_eq_SVIE} with the forcing term $x=\cK y$ if and only if the uniqueness holds for the lifted SEE \eqref{lift_eq_SEE} with the initial condition $y$.
\end{theo}

%% Proof

\begin{proof}
Assume that $X$ is a solution of the SVIE \eqref{lift_eq_SVIE} with $x=\cK y$. Define
\begin{equation*}
	Y_t:=e^{-\cdot t}y(\cdot)+\int^t_0e^{-\cdot(t-s)}b(X_s)\,\diff s+\int^t_0e^{-\cdot(t-s)}\sigma(X_s)\,\diff W_s,\ t\geq0.
\end{equation*}
Noting the integrability condition \eqref{lift_eq_X-integrability} (for $X$), by \cref{lift_lemm_convolution}, $Y$ is an $\cH_\mu$-valued continuous adapted process and satisfies $\int^T_0\|Y_t\|^2_{\cV_\mu}\,\diff t<\infty$ a.s.\ for any $T>0$. Furthermore, noting that the map $\mu[\cdot]$ is linear on $\cV_\mu$, it holds that
\begin{align*}
	\mu[Y_t]&=\mu\Big[e^{-\cdot t}y(\cdot)+\int^t_0e^{-\cdot(t-s)}b(X_s)\,\diff s+\int^t_0e^{-\cdot(t-s)}\sigma(X_s)\,\diff W_s\Big]\\
	&=(\cK y)(t)+\int^t_0K(t-s)b(X_s)\,\diff s+\int^t_0K(t-s)\sigma(X_s)\,\diff W_s=X_t
\end{align*}
a.s.\ for a.e.\ $t>0$. Thus, the integrability condition \eqref{lift_eq_Y-integrability} (for $Y$) and the equality \eqref{lift_eq_mild-sol} hold. Therefore, $Y$ is a mild solution of the lifted SEE \eqref{lift_eq_SEE}. The above observations show that $X_t=\mu[Y_t]$ a.s.\ for a.e.\ $t>0$.

Conversely, assume that $Y$ is a mild solution of the lifted SEE \eqref{lift_eq_SEE}, and define $X_t:=\mu[Y_t]$, $t>0$. Since $Y$ is an $\cH_\mu$-valued predictable process, and since the map $\mu:\cH_\mu\to\bR^n$ is Borel measurable, we see that $X$ is an $\bR^n$-valued predictable process. By the integrability condition \eqref{lift_eq_Y-integrability} (for $Y$) and \cref{lift_lemm_space} (iii), we see that the integrability condition \eqref{lift_eq_X-integrability} (for $X$) holds. Then, by the equality \eqref{lift_eq_mild-sol} and \cref{lift_lemm_convolution}, we have
\begin{align*}
	X_t&=\mu\Big[e^{-\cdot t}y(\cdot)+\int^t_0e^{-\cdot(t-s)}b(\mu[Y_s])\,\diff s+\int^t_0e^{-\cdot(t-s)}\sigma(\mu[Y_s])\,\diff W_s\Big]\\
	&=(\cK y)(t)+\int^t_0K(t-s)b(\mu[Y_s])\,\diff s+\int^t_0K(t-s)\sigma(\mu[Y_s])\,\diff W_s\\
	&=x(t)+\int^t_0K(t-s)b(X_s)\,\diff s+\int^t_0K(t-s)\sigma(X_s)\,\diff W_s
\end{align*}
a.s.\ for a.e.\ $t>0$. Therefore, $X$ is a solution of the SVIE \eqref{lift_eq_SVIE} with $x=\cK y$. Clearly, we have
\begin{equation*}
	Y_t=e^{-\cdot t}y(\cdot)+\int^t_0e^{-\cdot(t-s)}b(X_s)\,\diff s+\int^t_0e^{-\cdot(t-s)}\sigma(X_s)\,\diff W_s
\end{equation*}
a.s.\ for any $t\geq0$. This completes the proof.
\end{proof}

%% Remark

\begin{rem}
The lifted SEE \eqref{lift_eq_SEE} can be recast as an evolution equation defined in the Gelfand triplet $\cV_\mu\hookrightarrow\cH_\mu\hookrightarrow\cV^*_\mu$. Indeed, for each $y\in\cV_\mu$, define
\begin{equation*}
	\hat{b}(y)(\theta):=-\theta y(\theta)+b(\mu[y])=(Ay)(\theta)+b(\mu_0[y]),\ \theta\in\supp,
\end{equation*}
and
\begin{equation*}
	(\hat{\sigma}(y)w)(\theta):=\sigma(\mu[y])w=\sigma(\mu_0[y])w,\ \theta\in\supp,\ w\in\bR^d.
\end{equation*}
Then \eqref{lift_eq_SEE} can be written as follows:
\begin{equation}\label{lift_eq_variational}
	\begin{dcases}
	\diff Y_t=\hat{b}(Y_t)\,\diff t+\hat{\sigma}(Y_t)\,\diff W_t,\ t>0,\\
	Y_0=y\in\cH_\mu.
	\end{dcases}
\end{equation}
By \cref{lift_lemm_space}, the maps $\hat{b}:\cV_\mu\to\cV^*_\mu$ and $\hat{\sigma}:\cV_\mu\to L_2(\bR^d;\cH_\mu)$ are well-defined and measurable. Furthermore, if $b:\bR^n\to\bR^n$ and $\sigma:\bR^n\to\bR^{n\times d}$ are continuous, then $\hat{b}:\cV_\mu\to\cV^*_\mu$ and $\hat{\sigma}:\cV_\mu\to L_2(\bR^d;\cH_\mu)$ are continuous as well. Thus, the equation \eqref{lift_eq_variational} makes sense in the Gelfand triplet $\cV_\mu\hookrightarrow\cH_\mu\hookrightarrow\cV^*_\mu$. We say that $Y$ is a \emph{variational solution} (or an \emph{analytically strong solution}) of \eqref{lift_eq_variational} if $Y$ is an $\cH_\mu$-continuous adapted process satisfying the integrability condition \eqref{lift_eq_Y-integrability} a.s.\ for any $T>0$ such that
\begin{equation*}
	Y_t=y+\int^t_0\hat{b}(Y_s)\,\diff s+\int^t_0\hat{\sigma}(Y_s)\,\diff W_s
\end{equation*}
a.s.\ for any $t\geq0$. Here, the integrands $\hat{b}(Y_s)$ and $\hat{\sigma}(Y_s)$ are evaluated at a $\cV_\mu$-valued predictable version of $Y$, for example $\1_{\cV_\mu}(Y)Y$, and the integral with respect to $\diff s$ is in the Bochner sense in $\cV^*_\mu$. In our framework, it is easy to see that the two concepts of the variational solution of \eqref{lift_eq_variational} and the mild solution of \eqref{lift_eq_SEE} coincide. For general frameworks of the variational approach to infinite dimensional SEEs, we refer the readers to \cite{DaPrZa14,GaMa10,KrRo79} and references cited therein.
\end{rem}

%%%%%%%%%%%%%%
%% Subsection
%%%%%%%%%%%%%%

\subsection{Well-posedness of SEEs}\label{lift-well-posedness}

Next, we show the existence and uniqueness of the mild solution of the lifted SEE \eqref{lift_eq_SEE}. For the later purpose, we consider the following generalized SEE:
\begin{equation}\label{lift_eq_general-SEE}
	\begin{dcases}
	\diff Y_t(\theta)=-\theta Y_t(\theta)\,\diff t+b\big(t,\mu[Y_t],Y_t\big)\,\diff t+\sigma\big(t,\mu[Y_t],Y_t\big)\,\diff W_t,\ t>0,\ \theta\in\supp,\\
	Y_0(\theta)=\eta(\theta),\ \theta\in\supp,
	\end{dcases}
\end{equation}
where $\eta:\Omega\to\cH_\mu$ is an $\cF_0$-measurable random variable, and $b:\Omega\times[0,\infty)\times\bR^n\times\cH_\mu\to\bR^n$ and $\sigma:\Omega\times[0,\infty)\times\bR^n\times\cH_\mu\to\bR^{n\times d}$ are measurable maps satisfying suitable assumptions specified below.

%% Assumption

\begin{assum}\label{lift_assum_general-SEE}
\begin{itemize}
\item[(i)]
\emph{(The measurability condition)}.
$b$ and $\sigma$ are measurable, and $(b(t,x,y))_{t\geq0}$ and $(\sigma(t,x,y))_{t\geq0}$ are predictable for any $(x,y)\in\bR^n\times\cH_\mu$.
\item[(ii)]
\emph{(The linear growth condition)}.
There exists a constant $c_\LG>0$ and a nonnegative predictable process $\varphi$ satisfying $\bE\big[\int^T_0\varphi^2_t\,\diff t\big]<\infty$ for any $T>0$ such that
\begin{equation*}
	|b(t,x,y)|+|\sigma(t,x,y)|\leq\varphi_t+c_\LG\big\{|x|+\|y\|_{\cH_\mu}\big\}
\end{equation*}
for any $(t,x,y)\in[0,\infty)\times\bR^n\times\cH_\mu$ a.s.
\item[(iii)]
\emph{(The global Lipschitz condition)}.
There exists a constant $L>0$ such that
\begin{equation*}
	|b(t,x,y)-b(t,\bar{x},\bar{y})|+|\sigma(t,x,y)-\sigma(t,\bar{x},\bar{y})|\leq L\big\{|x-\bar{x}|+\|y-\bar{y}\|_{\cH_\mu}\big\}
\end{equation*}
for any $t\in[0,\infty)$, $x,\bar{x}\in\bR^n$ and $y,\bar{y}\in\cH_\mu$ a.s.
\item[(iii)']
\emph{(The local Lipschitz condition)}.
There exist two sequences $\{\tau_k\}_{k\in\bN}$ and $\{L_k\}_{k\in\bN}$ such that:
\begin{itemize}
\item
each $\tau_k$ is a stopping time, $\tau_k\leq\tau_{k+1}$ for any $k\in\bN$, and $\lim_{k\to\infty}\tau_k=\infty$ a.s.;
\item
each $L_k$ is a positive constant;
\item
for any $k\in\bN$, it holds that
\begin{equation*}
	|b(t,x,y)-b(t,\bar{x},\bar{y})|+|\sigma(t,x,y)-\sigma(t,\bar{x},\bar{y})|\leq L_k\big\{|x-\bar{x}|+\|y-\bar{y}\|_{\cH_\mu}\big\}
\end{equation*}
for any $t\in[0,\tau_k)$, $x,\bar{x}\in\bR^n$ and $y,\bar{y}\in\cH_\mu$ such that $\|y\|_{\cH_\mu}\vee\|\bar{y}\|_{\cH_\mu}\leq k$ a.s. 
\end{itemize}
\end{itemize}
\end{assum}

%% Remark

\begin{rem}
In \cref{Harnack}, we have to deal with a generalized SEE with local Lipschitz coefficients (see \eqref{Harnack_eq_coupling}).
In \cref{lift_assum_general-SEE} (iii)' (the local Lipschitz condition), we impose the locality with respect to the time variable and the $y$-variable, while it is global with respect to the $x$-variable. Indeed, this is sufficient for our purpose.
\end{rem}

The mild solution of the generalized SEE \eqref{lift_eq_general-SEE} is defined by the same way as in \cref{lift_defi_solution}.

%% Definition

\begin{defi}
Suppose that \cref{lift_assum_kernel} holds. Let $b:\Omega\times[0,\infty)\times\bR^n\times\cH_\mu\to\bR^n$ and $\sigma:\Omega\times[0,\infty)\times\bR^n\times\cH_\mu\to\bR^{n\times d}$ satisfy \cref{lift_assum_general-SEE} (i) (the measurability condition), and let $\eta:\Omega\to\cH_\mu$ be $\cF_0$-measurable. We say that an $\cH_\mu$-valued process $Y$ is a \emph{mild solution} of the generalized SEE \eqref{lift_eq_general-SEE} if $Y$ is $\cH_\mu$-continuous and adapted, the integrability conditions
\begin{equation*}
	\int^T_0\|Y_t\|^2_{\cV_\mu}\,\diff t<\infty,\ \int^T_0|b(t,\mu[Y_t],Y_t)|\,\diff t<\infty\ \text{and}\ \int^T_0|\sigma(t,\mu[Y_t],Y_t)|^2\,\diff t<\infty
\end{equation*}
hold a.s.\ for any $T>0$, and the equality
\begin{equation*}
	Y_t(\theta)=e^{-\theta t}\eta(\theta)+\int^t_0e^{-\theta(t-s)}b(s,\mu[Y_s],Y_s)\,\diff s+\int^t_0e^{-\theta(t-s)}\sigma(s,\mu[Y_s],Y_s)\,\diff W_s
\end{equation*}
holds for $\mu$-a.e.\ $\theta\in\supp$ a.s., for any $t\geq0$. We say that the (pathwise) uniqueness holds for the generalized SEE \eqref{lift_eq_general-SEE} if for any two mild solutions $Y^1$ and $Y^2$ of the generalized SEE \eqref{lift_eq_general-SEE}, it holds that $Y^1_t(\theta)=Y^2_t(\theta)$ for $\mu$-a.e.\ $\theta\in\supp$ a.s., for any $t\geq0$.
\end{defi}

%% Remark

\begin{rem}\label{lift_rem_Ito}
As we mentioned in \cref{lift_rem_solution}, for any mild solution $Y$ of the generalized SEE \eqref{lift_eq_general-SEE}, there exists a jointly measurable map $\tilde{Y}:\Omega\times[0,\infty)\times\supp\to\bR^n$ such that $Y_t=\tilde{Y}_t$ on $\cH_\mu$ a.s.\ for any $t\geq0$ and that, for any $\theta\in\supp$, the process $\tilde{Y}(\theta)=(\tilde{Y}_t(\theta))_{t\geq0}$ is an $\bR^n$-valued It\^{o} process satisfying
\begin{equation*}
	\begin{dcases}
	\diff\tilde{Y}_t(\theta)=-\theta\tilde{Y}_t(\theta)\,\diff t+b(t,\mu[\tilde{Y}_t],\tilde{Y}_t)\,\diff t+\sigma(t,\mu[\tilde{Y}_t],\tilde{Y}_t)\,\diff W_t,\ t>0,\\
	\tilde{Y}_0(\theta)=\eta(\theta),
	\end{dcases}
\end{equation*}
in the usual It\^{o}'s sense.
\end{rem}

In the following, we show the existence and uniqueness results for the generalized SEE \eqref{lift_eq_general-SEE} under \cref{lift_assum_general-SEE}. Actually, the generalized SEE \eqref{lift_eq_general-SEE} fits into the well-established framework of monotone SPDEs (cf.\ \cite{KrRo79}). To see this, suppose that \cref{lift_assum_general-SEE} (i) (the measurability condition), (ii) (the linear growth condition) and (iii) (the global Lipschitz condition) hold. Define
\begin{equation*}
	\hat{b}(t,y)(\theta):=-\theta y(\theta)+b(t,\mu[y],y),\ \theta\in\supp,\ t\geq0,\ y\in\cV_\mu,
\end{equation*}
and
\begin{equation*}
	(\hat{\sigma}(t,y)w)(\theta):=\sigma(t,\mu[y],y)w,\ \theta\in\supp,\ w\in\bR^d,\ t\geq0,\ y\in\cV_\mu.
\end{equation*}
By \cref{lift_lemm_space}, the maps $\hat{b}:\Omega\times[0,\infty)\times\cV_\mu\to\cV^*_\mu$ and $\hat{\sigma}:\Omega\times[0,\infty)\times\cV_\mu\to L_2(\bR^d;\cH_\mu)$ are well-defined and satisfy suitable measurability conditions. Furthermore, it can be easily shown that there exist constants $\kappa>0$, $\gamma\in\bR$ and a nonnegative predictable process $\psi$ satisfying $\bE\big[\int^T_0\psi_t\,\diff t\big]<\infty$ for any $T>0$ such that the following conditions hold:
\begin{itemize}
\item
(Continuity). The maps $\cV_\mu\ni y\mapsto\hat{b}(t,y)\in\cV^*_\mu$ and $\cV_\mu\ni y\mapsto\hat{\sigma}(t,y)\in L_2(\bR^d;\cH_\mu)$ are continuous for any $t\geq0$ a.s.
\item
(Monotonicity). It holds that
\begin{equation*}
	2\dual{\hat{b}(t,y)-\hat{b}(t,\bar{y})}{y-\bar{y}}+\|\hat{\sigma}(t,y)-\hat{\sigma}(t,\bar{y})\|^2_{L_2(\bR^d;\cH_\mu)}\leq\gamma\|y-\bar{y}\|^2_{\cH_\mu}
\end{equation*}
for any $t\geq0$ and $y,\bar{y}\in\cV_\mu$ a.s.
\item
(Coercivity). It holds that
\begin{equation*}
	2\dual{\hat{b}(t,y)}{y}+\|\hat{\sigma}(t,y)\|^2_{L_2(\bR^d;\cH_\mu)}\leq-\kappa\|y\|^2_{\cV_\mu}+\gamma\|y\|^2_{\cH_\mu}+\psi_t
\end{equation*}
for any $(t,y)\in[0,\infty)\times\cV_\mu$ a.s.
\item
(Growth). It holds that
\begin{equation*}
	\|\hat{b}(t,y)\|^2_{\cV^*_\mu}\leq\psi_t+\gamma\|y\|^2_{\cV_\mu}
\end{equation*}
for any $(t,y)\in[0,\infty)\times\cV_\mu$ a.s.
\end{itemize}
Therefore, by the well-known result of monotone SPDEs (cf.\ \cite{KrRo79}), the equation
\begin{equation*}
	\begin{dcases}
	\diff Y_t=\hat{b}(t,Y_t)\,\diff t+\hat{\sigma}(t,Y_t)\,\diff W_t,\ t>0,\\
	Y_0=\eta,
	\end{dcases}
\end{equation*}
with $\eta\in L^2_{\cF_0}(\cH_\mu)$ has a unique variational solution, which is the mild solution of the generalized SEE \eqref{lift_eq_general-SEE} in our framework.

However, the well-posedness of the generalized SEE \eqref{lift_eq_general-SEE} under \cref{lift_assum_general-SEE} (iii)' (the local Lipschitz condition) is not so trivial from the context of monotone SPDEs. Although the well-posedness of SPDEs with locally (in the space variable) monotone coefficients was proved by Liu and R\"{o}ckner \cite{LiRo10}, we cannot apply their results to our setting; our assumption is local in both space (the $y$ variable) and time.

Thus, we need further observations for the well-posedness of the generalized SEE \eqref{lift_eq_general-SEE}, especially in the local Lipschitz case. Here, we state our results, which will be proved in \hyperref[appendix]{Appendix}.

%% Theorem

\begin{theo}[A priori bound of mild solutions]\label{lift_theo_bound}
Suppose that \cref{lift_assum_kernel} holds. Let $b:\Omega\times[0,\infty)\times\bR^n\times\cH_\mu\to\bR^n$ and $\sigma:\Omega\times[0,\infty)\times\bR^n\times\cH_\mu\to\bR^{n\times d}$ satisfy \cref{lift_assum_general-SEE} (i) (the measurability condition) and (ii) (the linear growth condition). Then there exists a constant $C=C(\mu,c_\LG)>0$ such that, for any $T>0$, stopping time $\tau$, initial condition $\eta\in L^2_{\cF_0}(\cH_\mu)$, and any mild solution $Y$ of the generalized SEE \eqref{lift_eq_general-SEE}, the following holds:
\begin{equation*}
	\bE\Big[\sup_{t\in[0,T]}\|Y_{t\wedge\tau}\|^2_{\cH_\mu}+\int^{T\wedge\tau}_0\|Y_t\|^2_{\cV_\mu}\,\diff t\Big]\leq Ce^{CT}\bE\Big[\|\eta\|^2_{\cH_\mu}+\int^{T\wedge\tau}_0\varphi^2_t\,\diff t\Big]<\infty.
\end{equation*}
%In particular, each mild solution of the generalized SEE \eqref{lift_eq_general-SEE} with an $L^2$-initial condition is an $L^2$-mild solution.
\end{theo}

%% Proof

\begin{proof}
See \cref{appendix_1}.
\end{proof}

%% Theorem

\begin{theo}[Existence, uniqueness and stability; the global Lipschitz case]\label{lift_theo_global-Lip}
Suppose that \cref{lift_assum_kernel} holds. Let $b:\Omega\times[0,\infty)\times\bR^n\times\cH_\mu\to\bR^n$ and $\sigma:\Omega\times[0,\infty)\times\bR^n\times\cH_\mu\to\bR^{n\times d}$ satisfy \cref{lift_assum_general-SEE} (i) (the measurability condition), (ii) (the linear growth condition) and (iii) (the global Lipschitz condition). Then the following hold:
\begin{itemize}
\item[(i)]
For any initial condition $\eta\in L^2_{\cF_0}(\cH_\mu)$, there exists a unique mild solution $Y$ to the generalized SEE \eqref{lift_eq_general-SEE}.
\item[(ii)]
There exists a constant $C=C(\mu,L)>0$ such that, for any $T>0$, stopping time $\tau$, and any $(\eta,\bar{\eta},\bar{b},\bar{\sigma})\in L^2_{\cF_0}(\cH_\mu)\times L^2_{\cF_0}(\cH_\mu)\times L^2_\bF(0,T;\bR^n)\times L^2_\bF(0,T;\bR^{n\times d})$, it holds that
\begin{align*}
	&\bE\Big[\sup_{t\in[0,T]}\|Y_{t\wedge\tau}-\bar{Y}_{t\wedge\tau}\|^2_{\cH_\mu}+\int^{T\wedge\tau}_0\|Y_t-\bar{Y}_t\|^2_{\cV_\mu}\,\diff t\Big]\\
	&\leq Ce^{CT}\bE\Big[\|\eta-\bar{\eta}\|^2_{\cH_\mu}+\int^{T\wedge\tau}_0\Big\{|b(t,\mu[\bar{Y}_t],\bar{Y}_t)-\bar{b}_t|^2+|\sigma(t,\mu[\bar{Y}_t],\bar{Y}_t)-\bar{\sigma}_t|^2\Big\}\,\diff t\Big],
\end{align*}
where $Y$ is the mild solution of the generalized SEE \eqref{lift_eq_general-SEE} with the initial condition $\eta$, and $\bar{Y}$ is defined by
\begin{equation*}
	\bar{Y}_t:=e^{-\cdot t}\bar{\eta}(\cdot)+\int^t_0e^{-\cdot(t-s)}\bar{b}_s\,\diff s+\int^t_0e^{-\cdot(t-s)}\bar{\sigma}_s\,\diff W_s,\ t\in[0,T].
\end{equation*}
\end{itemize}
\end{theo}

%% Proof

\begin{proof}
See \cref{appendix_2}.
\end{proof}

%% Theorem

\begin{theo}[Existence and uniqueness; the local Lipschitz case]\label{lift_theo_local-Lip}
Suppose that \cref{lift_assum_kernel} holds. Let $b:\Omega\times[0,\infty)\times\bR^n\times\cH_\mu\to\bR^n$ and $\sigma:\Omega\times[0,\infty)\times\bR^n\times\cH_\mu\to\bR^{n\times d}$ satisfy \cref{lift_assum_general-SEE} (i) (the measurability condition), (ii) (the linear growth condition) and (iii)' (the local Lipschitz condition). Then for any initial condition $\eta\in L^2_{\cF_0}(\cH_\mu)$, there exists a unique mild solution $Y$ to the generalized SEE \eqref{lift_eq_general-SEE}.
\end{theo}

%% Proof

\begin{proof}
See \cref{appendix_3}.
\end{proof}

%%%%%%%%%%%%%%
%% Subsection
%%%%%%%%%%%%%%

\subsection{Markov property of the mild solution of the lifted SEE \eqref{lift_eq_SEE}}\label{lift-Markov}

We turn to the SVIE \eqref{lift_eq_SVIE} and the lifted SEE \eqref{lift_eq_SEE} with deterministic coefficients $b:\bR^n\to\bR^n$ and $\sigma:\bR^n\to\bR^{n\times d}$. The following theorem shows that the mild solution $Y$ of the lifted SEE \eqref{lift_eq_SEE} is indeed a ``Markovian'' lift of the SVIE \eqref{lift_eq_SVIE}. The proof is standard (see for example \cite[Chapter 9]{DaPrZa14}), but we give a complete proof here for readers' convenience.

%% Theorem

\begin{theo}\label{lift_theo_Markov}
Let \cref{lift_assum_kernel} hold. Suppose that the maps $b:\bR^n\to\bR^n$ and $\sigma:\bR^n\to\bR^{n\times d}$ are globally Lipschitz continuous. Then for any initial condition $y\in\cH_\mu$, there exists a unique mild solution $Y^y$ of the lifted SEE \eqref{lift_eq_SEE}. The family $\{Y^y_t\}_{t\geq0,y\in\cH_\mu}$ of the mild solutions is a time-homogeneous Markov process on $\cH_\mu$ with the transition semigroup $\{P_t\}_{t\geq0}$ defined by
\begin{equation*}
	P_tf(y):=\bE[f(Y^y_t)],\ t\geq0,\ y\in\cH_\mu,\ f\in\Bb,
\end{equation*}
where $\Bb$ denotes the set of bounded Borel measurable functions $f:\cH_\mu\to\bR$. Furthermore, $\{Y^y_t\}_{t\geq0,y\in\cH_\mu}$ has the Feller property, namely, $P_tf\in\Cb$ for any $f\in\Cb$ and any $t\geq0$, where $\Cb$ denotes the set of bounded continuous funcions $f:\cH_\mu\to\bR$.
\end{theo}

%% Proof

\begin{proof}
By \cref{lift_theo_global-Lip}, there exists a unique mild solution $Y^y$ of the lifted SEE \eqref{lift_eq_SEE} for any initial condition $y\in\cH_\mu$. More generally, for any $t_0\geq0$, $\eta\in L^2_{\cF_{t_0}}(\cH_\mu)$ and $y\in\cH_\mu$, consider the following SEEs:
\begin{equation*}
	\begin{dcases}
	\diff \cY^{t_0,\eta}_t(\theta)=-\theta\cY^{t_0,\eta}_t(\theta)\,\diff t+b(\mu[\cY^{t_0,\eta}_t])\,\diff t+\sigma(\mu[\cY^{t_0,\eta}_t])\,\diff W^{t_0}_t,\ t>0,\ \theta\in\supp,\\
	\cY^{t_0,\eta}_0(\theta)=\eta(\theta),\ \theta\in\supp,
	\end{dcases}
\end{equation*}
and
\begin{equation*}
	\begin{dcases}
	\diff \bar{\cY}^{t_0,y}_t(\theta)=-\theta\bar{\cY}^{t_0,y}_t(\theta)\,\diff t+b(\mu[\bar{\cY}^{t_0,y}_t])\,\diff t+\sigma(\mu[\bar{\cY}^{t_0,y}_t])\,\diff W^{t_0}_t,\ t>0,\ \theta\in\supp,\\
	\bar{\cY}^{t_0,y}_0(\theta)=y(\theta),\ \theta\in\supp,
	\end{dcases}
\end{equation*}
where $W^{t_0}_t:=W_{t_0+t}-W_{t_0}$, $t\geq0$, which is a Brownian motion independent of $\cF_{t_0}$. By \cref{lift_theo_global-Lip}, there exist unique mild solutions $\cY^{t_0,\eta}$ and $\bar{\cY}^{t_0,y}$ to the above SEEs. Clearly, $Y^y=\cY^{0,y}=\bar{\cY}^{0,y}$ for any $y\in\cH_\mu$. Note that, for any $t_0,t\geq0$,
\begin{align*}
	Y^y_{t_0+t}(\theta)&=e^{-\theta(t_0+t)}y(\theta)+\int^{t_0+t}_0e^{-\theta(t_0+t-s)}b(\mu[Y^y_s])\,\diff s+\int^{t_0+t}_0e^{-\theta(t_0+t-s)}\sigma(\mu[Y^y_s])\,\diff W_s\\
	&=e^{-\theta t}\Big\{e^{-\theta t_0}y(\theta)+\int^{t_0}_0e^{-\theta(t_0-s)}b(\mu[Y^y_s])\,\diff s+\int^{t_0}_0e^{-\theta(t_0-s)}\sigma(\mu[Y^y_s])\,\diff W_s\Big\}\\
	&\hspace{0.5cm}+\int^{t_0+t}_{t_0}e^{-\theta(t_0+t-s)}b(\mu[Y^y_s])\,\diff s+\int^{t_0+t}_{t_0}e^{-\theta(t_0+t-s)}\sigma(\mu[Y^y_s])\,\diff W_s\\
	&=e^{-\theta t}Y^y_{t_0}(\theta)+\int^t_0e^{-\theta(t-s)}b(\mu[Y^y_{t_0+s}])\,\diff s+\int^t_0e^{-\theta(t-s)}\sigma(\mu[Y^y_{t_0+s}])\,\diff W^{t_0}_s
\end{align*}
for $\mu$-a.e.\ $\theta\in\supp$ a.s. Thus, by the uniqueness, we have $Y^y_{t_0+t}=\cY^{t_0,Y^y_{t_0}}_t$ for any $t\geq0$ a.s. Furthermore, by the uniqueness in law of the SEE \eqref{lift_eq_SEE}, which follows from the pathwise uniqueness and the Yamada--Watanabe theorem (see \cite{Ku14} for the general form of the Yamada--Watanabe theorem), we see that $Y^y$ and $\bar{\cY}^{t_0,y}$ have the same law in the space $C([0,\infty);\cH_\mu)$ for any $t_0\geq0$ and any $y\in\cH_\mu$. This implies that
\begin{equation*}
	\bE[f(\bar{\cY}^{t_0,y}_t)]=\bE[f(Y^y_t)]=P_tf(y)
\end{equation*}
for any $t,t_0\geq0$, $y\in\cH_\mu$ and $f\in\Bb$.

Now we show that:
\begin{itemize}
\item[(i)]
there exists a version of $\bar{\cY}^{t_0,y}$ for each $t_0\geq0$ and $y\in\cH_\mu$ such that the map $(\omega,t,y)\mapsto\bar{\cY}^{t_0,y}_t(\omega)\in\cH_\mu$ is $\bF^{t_0}$-predictable, and hence independent of $\cF_{t_0}$, where $\bF^{t_0}=(\cF^{t_0}_t)_{t\geq0}$ is the filtration generated by the Brownian motion $W^{t_0}$;
\item[(ii)]
for any $\eta\in L^2_{\cF_{t_0}}(\cH_\mu)$, we have $\cY^{t_0,\eta}_t(\omega)=\bar{\cY}^{t_0,\eta(\omega)}_t(\omega)$ for any $t\geq0$ for $\bP$-a.e.\ $\omega\in\Omega$.
\end{itemize}
The above two assertions imply that
\begin{equation*}
	\bE\left[f\big(Y^y_{t_0+t}\big)\relmiddle|\cF_{t_0}\right]=\bE\left[f\Big(\cY^{t_0,Y^y_{t_0}}_t\Big)\relmiddle|\cF_{t_0}\right]=\bE\left[f\Big(\bar{\cY}^{t_0,Y^y_{t_0}}_t\Big)\relmiddle|\cF_{t_0}\right]=P_tf\big(Y^y_{t_0}\big)\ \text{a.s.}
\end{equation*}
for any $t,t_0\geq0$, $y\in\cH_\mu$ and $f\in\Bb$, which means that $\{Y^y_t\}_{t\geq0,y\in\cH_\mu}$ is a time-homogeneous Markov process on $\cH_\mu$ with the transition semigroup $\{P_t\}_{t\geq0}$.

We prove the assertion (i). Let $t_0\geq0$ be fixed. For each $y\in\cH_\mu$, Define $\bar{\cY}^{t_0,y,0}_t:=0$, $t\geq0$, and
\begin{equation*}
	\bar{\cY}^{t_0,y,k}_t(\theta):=e^{-\theta t}y(\theta)+\int^t_0e^{-\theta(t-s)}b(\mu[\bar{\cY}^{t_0,y,k-1}_s])\,\diff s+\int^t_0e^{-\theta(t-s)}\sigma(\mu[\bar{\cY}^{t_0,y,k-1}_s])\,\diff W^{t_0}_s
\end{equation*}
for $t\geq0$, $\theta\in\supp$ and $k\in\bN$. Noting \cite[Theorem 63 and its corollary]{Pr05} for the joint measurability of parametrized stochastic integrals, by the induction, we see that each $\bar{\cY}^{t_0,y,k}$ has a version (again denoted by $\bar{\cY}^{t_0,y,k}$) such that the map $(\omega,t,y)\mapsto\bar{\cY}^{t_0,y,k}_t(\omega)\in\cH_\mu$ is $\bF^{t_0}$-predictable. Define
\begin{equation*}
	\bar{\cY}^{t_0,y,\infty}_t:=
	\begin{dcases}
	\lim_{k\to\infty}\bar{\cY}^{t_0,y,k}_t\ &\text{if the limit exists},\\
	0\ &\text{otherwise},
	\end{dcases}
\end{equation*}
where the limit is taken in $\cH_\mu$. Then $(\omega,t,y)\mapsto\bar{\cY}^{t_0,y,\infty}_t(\omega)$ is $\bF^{t_0}$-predictable. From the proof of \cref{lift_theo_global-Lip} (see also \cref{appendix_rem_fixed-point}), we see that $\lim_{k\to\infty}\|\bar{\cY}^{t_0,y}-\bar{\cY}^{t_0,y,k}\|_{T,\lambda,\kappa}=0$ for any $T>0$, where the constants $\lambda,\kappa>0$ and the norm $\|\cdot\|_{T,\lambda,\kappa}$ were defined in the proof of \cref{lift_theo_global-Lip}. This implies that $\bar{\cY}^{t_0,y}_t=\bar{\cY}^{t_0,y,\infty}_t$ a.s.\ for any $t\geq0$ and any $y\in\cH_\mu$. Therefore, $\bar{\cY}^{t_0,y,\infty}$ is a desired version of $\bar{\cY}^{t_0,y}$.

Next, we prove the assertion (ii). We first assume that $\eta\in L^2_{\cF_{t_0}}(\cH_\mu)$ is of the form $\eta(\theta,\omega)=\sum^\infty_{i=1}y_i(\theta)\1_{A_i}(\omega)$, where $y_i\in\cH_\mu$, $i\in\bN$, are deterministic and $\{A_i\}_{i\in\bN}$ is an $\cF_{t_0}$-measurable partition of $\Omega$. Define $\tilde{\cY}^{t_0,\eta}_t(\theta,\omega):=\bar{\cY}^{t_0,\eta(\omega)}_t(\theta,\omega)$. Then
\begin{align*}
	\tilde{\cY}^{t_0,\eta}_t(\theta)&=\sum^\infty_{i=1}\bar{\cY}^{t_0,y_i}_t(\theta)\1_{A_i}\\
	&=\sum^\infty_{i=1}\Big\{e^{-\theta t}y_i(\theta)+\int^t_0e^{-\theta(t-s)}b(\mu[\bar{\cY}^{t_0,y_i}_s])\,\diff s+\int^t_0e^{-\theta(t-s)}\sigma(\mu[\bar{\cY}^{t_0,y_i}_s])\,\diff W^{t_0}_s\Big\}\1_{A_i}\\
	&=e^{-\theta t}\sum^\infty_{i=1}y_i(\theta)\1_{A_i}+\int^t_0e^{-\theta(t-s)}\sum^\infty_{i=1}b(\mu[\bar{\cY}^{t_0,y_i}_s])\1_{A_i}\,\diff s+\int^t_0e^{-\theta(t-s)}\sum^\infty_{i=1}\sigma(\mu[\bar{\cY}^{t_0,y_i}_s])\1_{A_i}\,\diff W^{t_0}_s\\
	&=e^{-\theta t}\eta(\theta)+\int^t_0e^{-\theta(t-s)}b(\mu[\tilde{\cY}^{t_0,\eta}_s])\,\diff s+\int^t_0e^{-\theta(t-s)}\sigma(\mu[\tilde{\cY}^{t_0,\eta}_s])\,\diff W^{t_0}_s,
\end{align*}
for $\mu$-a.e.\ $\theta\in\supp$ a.s.\ for any $t\geq0$. Thus, by the uniqueness, we have $\tilde{\cY}^{t_0,\eta}_t=\cY^{t_0,\eta}_t$ for any $t\geq0$ a.s., which shows the assertion (ii) in the special case.

In the general case, we can take a sequence $\eta_k\in L^2_{\cF_{t_0}}(\cH_\mu)$, $k\in\bN$, such that $\eta_k$ are of the above special forms and $\|\eta_k-\eta\|_{\cH_\mu}\leq\frac{1}{k}$ a.s.\ for any $k\in\bN$. Then for each $k\in\bN$, we have $\cY^{t_0,\eta_k}_t(\omega)=\bar{\cY}^{t_0,\eta_k(\omega)}_t(\omega)$ for any $t\geq0$ a.s. By \cref{lift_theo_global-Lip} (ii), there exists a constant $C>0$ such that, for any $T>0$,
\begin{equation*}
	\bE\Big[\sup_{t\in[0,T]}\|\bar{\cY}^{t_0,y_1}_t-\bar{\cY}^{t_0,y_2}_t\|^2_{\cH_\mu}\Big]\leq Ce^{CT}\|y_1-y_2\|^2_{\cH_\mu}
\end{equation*}
for any $y_1,y_2\in\cH_\mu$, and
\begin{equation*}
	\bE\Big[\sup_{t\in[0,T]}\|\cY^{t_0,\eta_1}_t-\cY^{t_0,\eta_2}_t\|^2_{\cH_\mu}\Big]\leq Ce^{CT}\bE\big[\|\eta_1-\eta_2\|^2_{\cH_\mu}\big]
\end{equation*}
for any $\eta_1,\eta_2\in L^2_{\cF_{t_0}}(\cH_\mu)$. On the one hand, noting that the map $(\omega,t,y)\mapsto\bar{\cY}^{t_0,y}_t(\omega)$ is $\bF^{t_0}$-predictable and hence independent of $\cF_{t_0}$, we have
\begin{align*}
	\bE\Big[\sup_{t\in[0,T]}\|\bar{\cY}^{t_0,\eta(\omega)}_t(\omega)-\bar{\cY}^{t_0,\eta_k(\omega)}_t(\omega)\|^2_{\cH_\mu}\Big]&=\bE\Big[\bE\Big[\sup_{t\in[0,T]}\|\bar{\cY}^{t_0,y_1}_t-\bar{\cY}^{t_0,y_2}_t\|^2_{\cH_\mu}\Big]\Big|_{y_1=\eta(\omega),y_2=\eta_k(\omega)}\Big]\\
	&\leq Ce^{CT}\bE\big[\|\eta-\eta_k\|^2_{\cH_\mu}\big]\leq\frac{1}{k^2}Ce^{CT}
\end{align*}
for any $T>0$ and any $k\in\bN$. On the other hand, we have
\begin{equation*}
	\bE\Big[\sup_{t\in[0,T]}\|\cY^{t_0,\eta}_t-\cY^{t_0,\eta_k}_t\|^2_{\cH_\mu}\Big]\leq Ce^{CT}\bE\big[\|\eta-\eta_k\|^2_{\cH_\mu}\big]\leq\frac{1}{k^2}Ce^{CT}
\end{equation*}
for any $T>0$ and any $k\in\bN$. Thus, we have $\bE\big[\sup_{t\in[0,T]}\|\bar{\cY}^{t_0,\eta(\omega)}_t(\omega)-\cY^{t_0,\eta}_t(\omega)\|^2_{\cH_\mu}\big]=0$ for any $T>0$, and hence $\cY^{t_0,\eta}_t(\omega)=\bar{\cY}^{t_0,\eta(\omega)}_t(\omega)$ for any $t\geq0$ for $\bP$-a.e.\ $\omega\in\Omega$. This proves the assertion (ii).

It remains to show the Feller property. Again by \cref{lift_theo_global-Lip} (ii), there exists a constant $C>0$ such that
\begin{equation*}
	\bE\Big[\sup_{t\in[0,T]}\|Y^{y_1}_t-Y^{y_2}_t\|^2_{\cH_\mu}\Big]\leq Ce^{CT}\|y_1-y_2\|^2_{\cH_\mu}
\end{equation*}
for any $T>0$ and any $y_1,y_2\in\cH_\mu$. This implies that the map $y\mapsto Y^{y}_t$ is continuous in probability for any $t\geq0$. Thus, by the dominated convergence theorem, for any $f\in\Cb$, $t\geq0$, and any sequence $y_k\in\cH_\mu$, $k\in\bN$, such that $\lim_{k\to\infty}y_k=y\in\cH_\mu$, we have
\begin{equation*}
	\lim_{k\to\infty}P_tf(y_k)=\lim_{k\to\infty}\bE\big[f(Y^{y_k}_t)\big]=\bE\big[f(Y^y_t)\big]=P_tf(y).
\end{equation*}
Therefore, $\{P_t\}_{t\geq0}$ has the Feller property. This completes the proof.
\end{proof}

%%%%%%%%%%%%%%%%%%%%%%%%%%%%%%%%%%
%%%%%%%%%%%%%%%%%%%%%%%%%%%%%%%%%%
%% Section
%%%%%%%%%%%%%%%%%%%%%%%%%%%%%%%%%%
%%%%%%%%%%%%%%%%%%%%%%%%%%%%%%%%%%

\section{Lifting Gaussian Volterra processes}\label{Gauss}

In this section, we demonstrate the tractability of our framework with the simplest example. For simplicity of notation, in this section, we assume that $n=d=1$. Let \cref{lift_assum_kernel} hold, and consider the SVIE \eqref{lift_eq_SVIE} with $b=0$ and $\sigma=1$:
\begin{equation}\label{Gauss_eq_SVIE-OU}
	X_t=x(t)+\int^t_0K(t-s)\,\diff W_s,\ t>0,
\end{equation}
which defines a Gaussian Volterra process on $\bR$. Let $x=\cK y$ for some $y\in\cH_\mu$. The corresponding lifted SEE \eqref{lift_eq_SEE} becomes
\begin{equation}\label{Gauss_eq_SEE-OU}
	\begin{dcases}
	\diff Y_t(\theta)=-\theta Y_t(\theta)\,\diff t+\diff W_t,\ t>0,\ \theta\in\supp,\\
	Y_0(\theta)=y(\theta),\ \theta\in\supp.
	\end{dcases}
\end{equation}
The unique mild solution of the above SEE is given by
\begin{equation*}
	Y_t=e^{-\cdot t}y(\cdot)+\int^t_0e^{-\cdot(t-s)}\,\diff W_s,\ t\geq0.
\end{equation*}
Similar infinite dimensional processes were considered in \cite{CaCu98,HaSt19}, where $L^1(\mu)$ and $L^2(\mu)$ were adopted as the state spaces. However, in their settings, since the diffusion coefficient $1\notin L^1(\mu)\cup L^2(\mu)$ in general, one cannot use the usual framework, for example \cite{DaPrZa96,DaPrZa14}, of the infinite dimensional stochastic calculus.

In our framework, the process $Y$ is an Ornstein--Uhlenbeck process on the Hilbert space $\cH_\mu$. The associated transition semigroup $\{P_t\}_{t\geq0}$ is an Ornstein--Uhlenbeck semigroup on $\Bb$, and the transition probability $P_t(y,A):=P_t\1_A(y)$, $A\in\cB(\cH_\mu)$, $y\in\cH_\mu$, $t\geq0$, is given by $P_t(y,\cdot)=\cN(e^{-\cdot t}y,Q_t)$. Here, $\cN(e^{-\cdot t}y,Q_t)$ denotes a Gaussian measure on $\cH_\mu$ with the mean $e^{-\cdot t}y\in\cH_\mu$ and the covariance operator $Q_t:\cH_\mu\to\cH_\mu$ given by
\begin{equation*}
	Q_t:=\int^t_0e^{-\cdot s}\iota\iota^*(e^{-\cdot s})^*\,\diff s,
\end{equation*}
where $\iota:\bR\hookrightarrow\cH_\mu$ denotes the inclusion operator. It is easy to see that $e^{-\cdot t}:\cH_\mu\to\cH_\mu$ is self-adjoint and that $\iota^*y=\int_\supp r(\theta)y(\theta)\dmu$ for $y\in\cH_\mu$, and hence
\begin{equation*}
	Q_ty=\int^t_0e^{-\cdot s}\int_\supp r(\theta)e^{-\theta s}y(\theta)\dmu\,\diff s=\int^t_0(e^{-\cdot s}1)\otimes(e^{-\cdot s}1)(y)\,\diff s
\end{equation*}
for each $y\in\cH_\mu$. Here, for $y_1,y_2\in\cH_\mu$, $y_1\otimes y_2:\cH_\mu\to\cH_\mu$ denotes the tensor product.

In the rest of this section, we characterize the invariant probability measure and discuss the strong Feller property for the transition semigroup $\{P_t\}_{t\geq0}$ associated with the lifted SEE \eqref{Gauss_eq_SEE-OU}. We can apply the general theory of \cite{DaPrZa96,DaPrZa14} to our framework.

%%%%%%%%%%%%%%%%%
%% Subsection
%%%%%%%%%%%%%%%%%

\subsection{Characterization of the invariant probability measure}

We characterize the existence and uniqueness of the invariant probability measure of the transition semigroup $\{P_t\}_{t\geq0}$. Recall that a probability measure $\pi$ on $\cH_\mu$ is called an \emph{invariant probability measure} of $\{P_t\}_{t\geq0}$ if
\begin{equation*}
	\int_{\cH_\mu}P_t(y,A)\,\pi(\diff y)=\pi(A)
\end{equation*}
for any $A\in\cB(\cH_\mu)$ and any $t\geq0$. By means of the general theory on the ergodicity of infinite dimensional Ornstein--Uhlenbeck processes (cf.\ \cite[Section 11.3]{DaPrZa14}), we get the following results.

%% Theorem

\begin{theo}\label{Gauss_theo_IPM}
Let \cref{lift_assum_kernel} hold.
\begin{itemize}
\item[(i)]
If both $\mu(\{0\})=0$ and $\int_{\supp\cap(0,1]}\theta^{-1}\dmu<\infty$ hold (or equivalently if $\int^\infty_1K(t)\,\diff t<\infty$ holds), then there exists a unique invariant probability measure of the transition semigroup $\{P_t\}_{t\geq0}$ associated with the lifted SEE \eqref{Gauss_eq_SEE-OU}. The invariant probability measure is a centred Gaussian measure $\cN(0,Q)$ with the covariance operator $Q:\cH_\mu\to\cH_\mu$ given by
\begin{equation*}
	(Qy)(\theta)=\int_\supp\frac{r(\theta')}{\theta+\theta'}y(\theta')\dmud,\ \theta\in\supp,\ y\in\cH_\mu.
\end{equation*}
Furthermore, it holds that
\begin{equation*}
	\lim_{t\to\infty}P_t(y,\cdot)=\cN(0,Q)\ \text{weakly on $\cH_\mu$}
\end{equation*}
for any $y\in\cH_\mu$.
\item[(ii)]
If either $\mu(\{0\})>0$ or $\int_{\supp\cap(0,1]}\theta^{-1}\dmu=\infty$ holds (or equivalently if $\int^\infty_1K(t)\,\diff t=\infty$ holds), then the transition semigroup $\{P_t\}_{t\geq0}$ associated with the lifted SEE \eqref{Gauss_eq_SEE-OU} does not have invariant probability measures.
\end{itemize}
\end{theo}

%% Proof

\begin{proof}
By \cite[Theorem 11.17]{DaPrZa14}, there exists an invariant probability measure of $\{P_t\}_{t\geq0}$ if and only if $\sup_{t\in[0,\infty)}\mathrm{Tr}\,[Q_t]<\infty$, where $\mathrm{Tr}\,[Q_t]$ denotes the trace of the operator $Q_t:\cH_\mu\to\cH_\mu$. Observe that
\begin{equation*}
	\mathrm{Tr}\,[Q_t]=\int^t_0\|e^{-\cdot s}1\|^2_{\cH_\mu}\,\diff s=\int^t_0\int_\supp r(\theta)e^{-2\theta s}\dmu\,\diff s=t\mu(\{0\})+\int_{\supp\setminus\{0\}}\frac{1-e^{-2\theta t}}{2\theta}r(\theta)\dmu.
\end{equation*}
By the monotone convergence theorem, we have
\begin{equation*}
	\sup_{t\in[0,\infty)}\int_{\supp\setminus\{0\}}\frac{1-e^{-2\theta t}}{2\theta}r(\theta)\dmu=\lim_{t\to\infty}\int_{\supp\setminus\{0\}}\frac{1-e^{-2\theta t}}{2\theta}r(\theta)\dmu=\int_{\supp\setminus\{0\}}\frac{r(\theta)}{2\theta}\dmu.
\end{equation*}
Therefore, $\sup_{t\in[0,\infty)}\mathrm{Tr}\,[Q_t]<\infty$ if and only if both $\mu(\{0\})=0$ and $\int_{\supp\cap(0,1]}\theta^{-1}\dmu<\infty$ hold. It is easy to see that the latter is equivalent to $\int^\infty_1K(t)\,\diff t<\infty$.

Suppose that $\sup_{t\in[0,\infty)}\mathrm{Tr}\,[Q_t]<\infty$. Noting that $\mu(\{0\})=0$, by the dominated convergence theorem, we have
\begin{equation*}
	\lim_{t\to\infty}\|e^{-\cdot t}y\|_{\cH_\mu}=\lim_{t\to\infty}\Big(\int_\supp r(\theta)e^{-2\theta t}|y(\theta)|^2\dmu\Big)^{1/2}=0
\end{equation*}
for any $y\in\cH_\mu$. Therefore, by \cite[Theorems 11.17 and 11.20]{DaPrZa14}, the invariant probability measure of $\{P_t\}_{t\geq0}$ is unique, it is a centred Gaussian measure $\cN(0,Q)$ with the covariance operator $Q:\cH_\mu\to\cH_\mu$ given by
\begin{align*}
	(Qy)(\theta)&=\int^\infty_0\big(e^{-\cdot t}\iota\iota^*(e^{-\cdot t})^*y\big)(\theta)\,\diff t\\
	&=\int^\infty_0e^{-\theta t}\int_\supp r(\theta')e^{-\theta' t}y(\theta')\dmud\,\diff t\\
	&=\int_\supp\frac{r(\theta')}{\theta+\theta'}y(\theta')\dmud,\ \theta\in\supp,\ y\in\cH_\mu,
\end{align*}
and it holds that $\lim_{t\to\infty}P_t(y,\cdot)=\cN(0,Q)$ weakly on $\cH_\mu$ for any $y\in\cH_\mu$. This completes the proof.
\end{proof}

%% Remark

\begin{rem}
On the one hand, if $K=K_\mathrm{frac}$ is the fractional kernel, then $\int^\infty_1K(t)\,\diff t=\infty$, and thus the transition semigroup $\{P_t\}_{t\geq0}$ associated with the lifted SEE \eqref{Gauss_eq_SEE-OU} does not have invariant probability measures. On the other hand, if $K=K_\mathrm{gamma}$ is the Gamma kernel, then the associated transition semigroup has a unique invariant probability measure; see \cref{lift_exam_kernel}. The above result indicates that the existence and uniqueness of the invariant probability measure is more about the asymptotic behaviour of the kernel $K(t)$ as $t\to\infty$ than it is about the regularity/singularity at $t=0$.
\end{rem}

Suppose that both $\mu(\{0\})=0$ and $\int_{\supp\cap(0,1]}\theta^{-1}\dmu<\infty$ hold, and let $\pi=\cN(0,Q)$, which is the invariant probability measure of the transition semigroup $\{P_t\}_{t\geq0}$ associated with the lifted SEE \eqref{Gauss_eq_SEE-OU}. Since the mild solution $Y=Y^y$ of the lifted SEE \eqref{Gauss_eq_SEE-OU} with a given initial condition $y\in\cH_\mu$ satisfies $Y^y_t\in\cV_\mu$ a.s.\ for a.e.\ $t>0$, we have $P_t(y,\cV_\mu)=1$ for a.e.\ $t>0$ for any $y\in\cH_\mu$, and hence $\pi(\cV_\mu)=1$. Furthermore, by the Cauchy--Schwarz inequality, we have
\begin{align*}
	\|Qy\|^2_{\cV_\mu}&=\int_\supp(\theta+1)r(\theta)\Big|\int_\supp\frac{r(\theta')}{\theta+\theta'}y(\theta')\dmud\Big|^2\dmu\\
	&\leq\int_\supp(\theta+1)r(\theta)\int_\supp\frac{\theta'+1}{(\theta+\theta')^2}r(\theta')\dmud\,\int_\supp(\theta'+1)^{-1}r(\theta')|y(\theta')|^2\dmud\dmu\\
	&=\int_\supp\int_\supp\frac{(\theta+1)(\theta'+1)}{(\theta+\theta')^2}r(\theta)r(\theta')\dmu\dmud\,\|y\|^2_{\cV^*_\mu}
\end{align*}
for any $y\in\cH_\mu$. Noting that
\begin{align*}
	&\int_\supp\int_\supp\frac{(\theta+1)(\theta'+1)}{(\theta+\theta')^2}r(\theta)r(\theta')\dmu\dmud\\
	&\leq\int_\supp\int_\supp\Big(1+\frac{1}{\theta}\Big)\Big(1+\frac{1}{\theta'}\Big)r(\theta)r(\theta')\dmu\dmud\\
	&=\Big(\int_\supp\Big(1+\frac{1}{\theta}\Big)r(\theta)\dmu\Big)^2<\infty,
\end{align*}
we see that $Q(\cH_\mu)\subset\cV_\mu$, and $Q:\cH_\mu\to\cV_\mu$ can be extended to a bounded linear operator $Q:\cV^*_\mu\to\cV_\mu$. Thus, for any $y'\in\cH_\mu$,
\begin{align*}
	\int_{\cV_\mu}e^{\sqrt{-1}\dual{y'}{y}}\,\pi(\diff y)&=\int_{\cV_\mu}e^{\sqrt{-1}\langle y',y\rangle_{\cH_\mu}}\,\pi(\diff y)\\
	&=\int_{\cH_\mu}e^{\sqrt{-1}\langle y',y\rangle_{\cH_\mu}}\,\pi(\diff y)\\
	&=e^{-\frac{1}{2}\langle y',Qy'\rangle_{\cH_\mu}}\\
	&=e^{-\frac{1}{2}\dual{y'}{Qy'}}.
\end{align*}
Since $\cH_\mu$ is dense in $\cV^*_\mu$, by using the dominated convergence theorem and the continuity of $Q:\cV^*_\mu\to\cV_\mu$, we see that
\begin{equation*}
	\int_{\cV_\mu}e^{\sqrt{-1}\dual{y'}{y}}\,\pi(\diff y)=e^{-\frac{1}{2}\dual{y'}{Qy'}}
\end{equation*}
holds for any $y'\in\cV^*_\mu$. Therefore, $\pi=\cN(0,Q)$ can be seen as a centred Gaussian measure on $\cV_\mu$ with the covariance operator $Q:\cV^*_\mu\to\cV_\mu$.

By means of the invariant probability measure $\pi$ of the lifted SEE \eqref{Gauss_eq_SEE-OU}, we can construct a ``stationary solution'' $X$ of the SVIE \eqref{Gauss_eq_SVIE-OU}. Let $\eta:\Omega\to\cH_\mu$ be an initial condition (independent of $W$) with the law $\pi=\cN(0,Q)$, and let $\xi(t)=(\cK\eta)(t)=\mu[e^{-\cdot t}\eta]$, $t\geq0$. Recall that, for any $t>0$, $e^{-\cdot t}$ is a bounded linear operator from $\cH_\mu$ to $\cV_\mu$, and hence $\xi(t)=\mu_0[e^{-\cdot t}\eta]$. It is easy to see that the adjoint operator $\mu^*_0:\bR\to\cV^*_\mu$ of the integral operator $\mu_0:\cV_\mu\to\bR$ is given by $\mu^*_0[1](\theta)=r(\theta)^{-1}$, $\theta\in\supp$. Thus, the law of $\xi(t)$ (on $\bR$) is a centred Gaussian measure with the variance
\begin{equation*}
	\mu_0e^{-\cdot t}Q(\mu_0e^{-\cdot t})^*=\int_\supp\int_\supp\frac{e^{-(\theta+\theta')t}}{\theta+\theta'}\dmu\dmud=\int^\infty_0K(t+s)^2\,\diff s.
\end{equation*}
Let $\bar{Y}$ be the mild solution of the lifted SEE \eqref{Gauss_eq_SEE-OU} with the initial condition $\eta$, and consider the Volterra process
\begin{equation*}
	\bar{X}_t=\xi(t)+\int^t_0K(t-s)\,\diff W_s,\ t>0.
\end{equation*}
Then we have $\bar{Y}_t\in\cV_\mu$ and $\bar{X}_t=\mu[\bar{Y}_t]=\mu_0[\bar{Y}_t]$ a.s.\ for a.e.\ $t>0$. Thus, for a.e.\ $t>0$, the law of $\bar{X}_t$ (on $\bR$) is a centred Gaussian measure with the variance
\begin{equation*}
	\mu_0Q\mu^*_0=\int_\supp\int_\supp\frac{1}{\theta+\theta'}\dmu\dmud=\int^\infty_0K(s)^2\,\diff s,
\end{equation*}
which does not depend on $t$.

\subsection{Failure of the strong Feller property}

Recall that a Markov transition semigroup $\{P_t\}_{t\geq0}$ is said to be \emph{strongly Feller} at time $t>0$ if $P_tf$ is continuous for any $\bR$-valued bounded measurable function $f$. It is well-known that if the noise is nondegenerate (i.e.\ the diffusion is elliptic in some sense), the strong Feller property can be obtained as a consequence of the Bismut--Elworthy--Li formula (cf.\ \cite[Section 9.4]{DaPrZa14}) or the log-Harnack inequality (cf.\ \cite{Wa13}). However, since the noise of the lifted SEE \eqref{Gauss_eq_SEE-OU} is degenerate (i.e.\ $\iota\iota^*:\cH_\mu\to\cH_\mu$ is not elliptic), there is no reason that the strong Feller property holds for the associated transition semigroup $\{P_t\}_{t\geq0}$. Indeed, the following theorem shows that, unfortunately, the strong Feller property does not hold in a typical infinite dimensional situation.

%% Theorem

\begin{theo}\label{Gauss_theo_strong-Feller}
Let \cref{lift_assum_kernel} hold. Assume that there exists a Borel set $B\subset\supp$ such that $\mu(B)>0$ and $\mu$ is atomless on $B$. Then the transition semigroup $\{P_t\}_{t\geq0}$ associated with the lifted SEE \eqref{Gauss_eq_SEE-OU} is not strongly Feller. More precisely, for any $t>0$, there exists $f\in\Bb$ such that $y\mapsto P_tf(y)$ is not continuous on $\cH_\mu$.
\end{theo}

%% Proof

\begin{proof}
Let $t>0$ be fixed. By \cite[Theorem 7.2.1]{DaPrZa96}, $\{P_t\}_{t\geq0}$ is strongly Feller at time $t$ if and only if
\begin{equation}\label{Gauss_eq_strong-Feller}
	e^{-\cdot t}(\cH_\mu)\subset Q^{1/2}_t(\cH_\mu).
\end{equation}
Furthermore, noting that the operators $e^{-\cdot t}:\cH_\mu\to\cH_\mu$ and $Q^{1/2}_t:\cH_\mu\to\cH_\mu$ are self adjoint, by \cite[Proposition B.1]{DaPrZa14}, \eqref{Gauss_eq_strong-Feller} holds if and only if there exists a constant $k>0$ such that $\|e^{-\cdot t}y\|_{\cH_\mu}\leq k\|Q^{1/2}_ty\|_{\cH_\mu}$ for any $y\in\cH_\mu$. Therefore, to prove our assertion, it suffices to show that, for any $\ep>0$, there exists $y_\ep\in\cH_\mu$ such that $\|Q^{1/2}_ty_\ep\|_{\cH_\mu}<\ep\|e^{-\cdot t}y_\ep\|_{\cH_\mu}$.

Let $M\in(0,\infty)$ be such that $\mu(B\cap[0,M])>0$. Fix an arbitrary number $\ep>0$. By the assumption, there exists a Borel set $B_\ep\subset B\cap[0,M]$ such that
\begin{equation*}
	0<\int_{B_\ep}r(\theta)\dmu<\frac{2M\ep^2}{e^{2Mt}-1}.
\end{equation*}
Take $y_\ep:=\1_{B_\ep}\in\cH_\mu$. Note that
\begin{equation*}
	\|e^{-\cdot t}y_\ep\|_{\cH_\mu}=\Big(\int_{B_\ep}r(\theta)e^{-2\theta t}\dmu\Big)^{1/2}>0.
\end{equation*}
By the Cauchy--Schwarz inequality, we have
\begin{align*}
	\|Q^{1/2}_ty_\ep\|^2_{\cH_\mu}&=\langle Q_ty_\ep,y_\ep\rangle_{\cH_\mu}\\
	&=\int^t_0\Big|\int_\supp r(\theta)e^{-\theta s}y_\ep(\theta)\dmu\Big|^2\,\diff s\\
	&=\int^t_0\Big|\int_{B_\ep}r(\theta)e^{-\theta s}\dmu\Big|^2\,\diff s\\
	&\leq\int^t_0\int_{B_\ep}r(\theta)\dmu\,\int_{B_\ep}r(\theta)e^{-2\theta s}\dmu\,\diff s\\
	&\leq\int^t_0\int_{B_\ep}r(\theta)\dmu\,\int_{B_\ep}r(\theta)e^{-2\theta t}e^{2M(t-s)}\dmu\,\diff s\\
	&=\frac{e^{2Mt}-1}{2M}\int_{B_\ep} r(\theta)\dmu\,\|e^{-\cdot t}y_\ep\|^2_{\cH_\mu}\\
	&<\ep^2\|e^{-\cdot t}y_\ep\|^2_{\cH_\mu},
\end{align*}
and hence $\|Q^{1/2}_ty_\ep\|_{\cH_\mu}<\ep\|e^{-\cdot t}y_\ep\|_{\cH_\mu}$. This completes the proof.
\end{proof}

%% Remark

\begin{rem}
If $K$ is either the fractional kernel $K_\mathrm{frac}$ or the Gamma kernel $K_\mathrm{gamma}$, then the corresponding Radon measure $\mu$ is diffusive, and thus the associated transition semigroup of the lifted SEE \eqref{Gauss_eq_SEE-OU} is not strongly Feller; see \cref{lift_exam_kernel}.
\end{rem}

The above negative result indicates that it is difficult to capture the ``smoothing effect'' even in the simple example. The difficulty comes from the fact that the lifted SEE \eqref{Gauss_eq_SEE-OU} is \emph{highly degenerate}; the state space $\cH_\mu$ is infinite dimensional, while the noise (Brownian motion) is one dimensional. Nevertheless, we can show a weaker version of the strong Feller property for general lifted SEEs (under suitable assumptions). This can be recovered by means of the
\emph{asymptotic log-Harnack inequality}, which is our main topic in the next section.

%%%%%%%%%%%%%%%%%%%%%%%%%%%%%%%%%%
%%%%%%%%%%%%%%%%%%%%%%%%%%%%%%%%%%
%% Section
%%%%%%%%%%%%%%%%%%%%%%%%%%%%%%%%%%
%%%%%%%%%%%%%%%%%%%%%%%%%%%%%%%%%%

\section{Asymptotic log-Harnack inequality}\label{Harnack}

In this section, we investigate asymptotic properties of the transition semigroup $\{P_t\}_{t\geq0}$ associated with the (general) lifted SEE \eqref{lift_eq_SEE} defined in \cref{lift_theo_Markov}. In particular, we derive an \emph{asymptotic log-Harnack inequality} for $\{P_t\}_{t\geq0}$ and show some important consequences.

We impose the following assumptions on the maps $b:\bR^n\to\bR^n$ and $\sigma:\bR^n\to\bR^{n\times d}$.

%% Assumption

\begin{assum}\label{Harnack_assum_coefficients}
There exists a constant $L>0$ such that
\begin{equation*}
	|b(x)-b(\bar{x})|+|\sigma(x)-\sigma(\bar{x})|\leq L|x-\bar{x}|
\end{equation*}
for any $x,\bar{x}\in\bR^n$. Furthermore, $\|\sigma\|_\infty:=\sup_{x\in\bR^n}|\sigma(x)|<\infty$, and for any $x\in\bR^n$, the matrix $\sigma(x)\in\bR^{n\times d}$ has a right inverse $\sigma^{-1}(x)\in\bR^{d\times n}$ such that the map $x\mapsto\sigma^{-1}(x)$ is measurable and $\|\sigma^{-1}\|_\infty:=\sup_{x\in\bR^n}|\sigma^{-1}(x)|<\infty$.
\end{assum}

%% Remark

\begin{rem}
We note that the above assumption requires that the dimension $d$ of Brownian motion $W$ is greater than or equal to the dimension $n$ of solutions of SVIEs, and the matrix $\sigma(x)$ is of rank $n$ for any $x\in\bR^n$.
\end{rem}

In the following, for a function $f:\cH_\mu\to\bR$, we use the notations:
\begin{equation*}
	|\nabla f|(y):=\limsup_{\bar{y}\to y}\frac{|f(y)-f(\bar{y})|}{\|y-\bar{y}\|_{\cH_\mu}},\ y\in\cH_\mu,
\end{equation*}
and $\|\nabla f\|_\infty:=\sup_{y\in\cH_\mu}|\nabla f|(y)$.

%%%%%%%%%%%%%%%%%
%% Subsection
%%%%%%%%%%%%%%%%%

\subsection{Main result and applications}

The following is the main result of this section, which is proved in the next subsection.

%% Theorem

\begin{theo}\label{Harnack_theo_Harnack}
Let \cref{lift_assum_kernel} and \cref{Harnack_assum_coefficients} hold. Then the following \emph{asymptotic log-Harnack inequality} holds:
\begin{equation}\label{Harnack_eq_Harnack}
\begin{split}
	P_t\log f(\bar{y})&\leq\log P_tf(y)+\frac{1}{2}\|\sigma^{-1}\|^2_\infty\Big(1+2L^2\Big(1+\int_\supp r(\theta)\dmu\Big)r(m)^{-2}\Big)\|y-\bar{y}\|^2_{\cH_\mu}\\
	&\hspace{3cm}+r(m)^{-1/2}e^{-\beta t/2}\|y-\bar{y}\|_{\cH_\mu}\|\nabla\log f\|_\infty
\end{split}
\end{equation}
for any $t\geq0$, $y,\bar{y}\in\cH_\mu$ and any $f\in\Bb$ such that $f\geq1$ and $\|\nabla\log f\|_\infty<\infty$, where $\beta:=\inf\supp\geq0$, and $m\geq1$ is an arbitrary constant satisfying
\begin{equation}\label{Harnack_eq_m}
	2L^2\Big(1+\int_\supp r(\theta)\dmu\Big)\int_{\supp\cap[m,\infty)}r(\theta)\dmu\leq1.
\end{equation}
\end{theo}

%% Remark

\begin{rem}
\begin{itemize}
\item[(i)]
Under \cref{lift_assum_kernel}, there exists a constant $m\geq1$ such that \eqref{Harnack_eq_m} holds.
\item[(ii)]
\cref{Harnack_theo_Harnack} holds true even for $\beta:=\inf\supp=0$. However, in order to control the asymptotic decay of the remainder term of \eqref{Harnack_eq_Harnack} as $t\to\infty$, we need $\beta>0$. This excludes the case of the fractional kernel, while the Gamma kernel satisfies this condition; see \cref{lift_exam_kernel}. We need further considerations for the case $\beta=0$, and we leave this problem to the future research.
\end{itemize}
\end{rem}

By means of the general results established in \cite{BaWaYu19,Xu11}, we see that the asymptotic log-Harnack inequality \eqref{Harnack_eq_Harnack} (with $\beta>0$) implies that the Markov semigroup $\{P_t\}_{t\geq0}$ is \emph{asymptotically strong Feller}, \emph{asymptotically irreducible}, and \emph{possesses at most one invariant probability measure}. Before stating the results, let us recall the definition of the asymptotic strong Feller property introduced by Hairer and Mattingly \cite{HaMa06}.

%% Definition

\begin{defi}
Let $E$ be a Polish space endowed with the Borel $\sigma$-field, and $\cP(E)$ the set of probability measures on $E$.
\begin{itemize}
\item
A continuous function $d:E\times E\to[0,\infty)$ is called a \emph{pseudo-metric} on $E$ if $d(x,x)=0$ and $d(x,z)\leq d(x,y)+d(y,z)$ hold for $x,y,z\in E$. An increasing sequence of pseudo-metrics $\{d_k\}_{k\in\bN}$ (i.e., $d_k(\cdot,\cdot)\leq d_{k+1}(\cdot,\cdot)$ for any $k\in\bN$) is said to be a \emph{totally separating system} if $\lim_{k\to\infty}d_k(x,y)=1$ for any $x\neq y$.
\item
For a pseudo-metric $d$, the \emph{transportation cost} (which is also called the $L^1$-Wasserstein distance when $d$ is a metric) is defined by
\begin{equation*}
	\bW^d_1(p_1,p_2):=\inf_{p\in\sC(p_1,p_2)}\int_{E\times E}d(x,y)\,p(\diff x,\diff y),\ p_1,p_2\in\cP(E),
\end{equation*}
where $\sC(p_1,p_2)$ is the set of couplings of $p_1$ and $p_2$, that is, $p\in\sC(p_1,p_2)$ means $p\in\cP(E\times E)$ with $p(\cdot\times E)=p_1$ and $p(E\times\cdot)=p_2$.
\item
A Markov semigroup $\{P_t\}_{t\geq0}$ on $E$ is called \emph{asymptotically strong Feller} at a point $x\in E$ if there exists a totally separating system $\{d_k\}_{k\in\bN}$ of pseudo-metrics and an increasing sequence $\{t_k\}_{k\in\bN}$ of nonnegative numbers with $\lim_{k\to\infty}t_k=\infty$ such that
\begin{equation*}
	\inf_{U\in\cU_x}\limsup_{k\to\infty}\sup_{y\in U}\bW^{d_k}_1(P_{t_k}(x,\cdot),P_{t_k}(y,\cdot))=0,
\end{equation*}
where $\cU_x$ denotes the collection of all open sets of $E$ containing a point $x\in E$, and $P_t(x,A):=P_t\1_A(x)$ for $t\geq0$, $x\in E$ and a measurable set $A\subset E$. $\{P_t\}_{t\geq0}$ is called asymptotically strong Feller if it is asymptotically strong Feller at any $x\in E$.
\end{itemize}
\end{defi}

Now we state important consequences of \cref{Harnack_theo_Harnack}, which follow by applying \cite[Theorem 2.1]{BaWaYu19} to our result.

%% Corollary

\begin{cor}\label{Harnack_cor_application}
Let \cref{lift_assum_kernel} and \cref{Harnack_assum_coefficients} hold. Furthermore, assume that $\beta:=\inf\supp>0$. Fix a constant $m\geq1$ satisfying \eqref{Harnack_eq_m}, and denote
\begin{equation*}
	\Lambda:=\frac{1}{2}\|\sigma^{-1}\|^2_\infty\Big(1+2L^2\Big(1+\int_\supp r(\theta)\dmu\Big)r(m)^{-2}\Big).
\end{equation*}
Then the following hold:
\begin{itemize}
\item[(i)]
\emph{(Gradient estimate)}.
For any $t>0$ and any $f\in\Bb$ with $\|\nabla f\|_\infty<\infty$, it holds that
\begin{equation*}
	|\nabla P_t f|\leq\sqrt{2\Lambda}\sqrt{P_tf^2-(P_tf)^2}+r(m)^{-1/2}e^{-\beta t/2}\|\nabla f\|_\infty.
\end{equation*}
In particular, $\{P_t\}_{t\geq0}$ is asymptotically strong Feller.
\item[(ii)]
\emph{(Asymptotic heat kernel estimate)}.
If $\{P_t\}_{t\geq0}$ has an invariant probability measure $\pi\in\cP(\cH_\mu)$, then for any $y\in\cH_\mu$ and any $f\in\Bb$ satisfying $f\geq0$ and $\|\nabla f\|_\infty<\infty$, it holds that
\begin{equation*}
	\limsup_{t\to\infty}P_tf(y)\leq\log\left(\frac{\int_{\cH_\mu}e^{f(\bar{y})}\,\pi(\diff\bar{y})}{\int_{\cH_\mu}e^{-\Lambda\|y-\bar{y}\|^2_{\cH_\mu}}\,\pi(\diff\bar{y})}\right).
\end{equation*}
Consequently, for any $y\in\cH_\mu$ and any closed set $A\subset\cH_\mu$ with $\pi(A)=0$, it holds that
\begin{equation*}
	\lim_{t\to\infty}P_t(y,A)=0.
\end{equation*}
\item[(iii)]
\emph{(Uniqueness of invariant probability measures)}.
$\{P_t\}_{t\geq0}$ has at most one invariant probability measure.
\item[(iv)]
\emph{(Asymptotic irreducibility)}.
Let $A\subset\cH_\mu$ be a Borel measurable set such that
\begin{equation*}
	\delta(y,A):=\liminf_{t\to\infty}P_t(y,A)>0
\end{equation*}
for some $y\in\cH_\mu$. Then, for any $\bar{y}\in\cH_\mu$ and $\ep>0$, it holds that
\begin{equation*}
	\liminf_{t\to\infty}P_t(\bar{y},A_\ep)>0,
\end{equation*}
where $A_\ep:=\{z\in\cH_\mu\,|\,\inf_{w\in A}\|z-w\|_{\cH_\mu}<\ep\}$. Moreover, for any $\ep_0\in(0,\delta(y,A))$, there exists a constant $t_0>0$ such that
\begin{equation*}
	P_t(\bar{y},A_\ep)>0\ \text{provided}\ t\geq t_0\ \text{and}\ r(m)^{-1/2}e^{-\beta t/2}\|y-\bar{y}\|_{\cH_\mu}<\ep\ep_0.
\end{equation*}
\end{itemize}
\end{cor}

Now we illustrate the idea of the proof of \cref{Harnack_theo_Harnack}. A detailed proof is given in the next subsection. Basically, we follow the asymptotic coupling method adopted in \cite{BaWaYu19}, which we shall explain here from a stochastic control theoretical point of view. First, fix $y,\bar{y}\in\cH_\mu$, and consider a ``control process'' $v$, which is an $\bR^d$-valued predictable process such that the stochastic exponential
\begin{equation*}
	R^v_t:=\exp\Big(-\int^t_0\langle v_s,\diff W_s\rangle-\frac{1}{2}\int^t_0|v_s|^2\,\diff s\Big),\ t\geq0,
\end{equation*}
is well-defined and becomes a uniformly integrable martingale under $\bP$. Then, consider the controlled SEE
\begin{equation*}
	\begin{dcases}
	\diff \bar{Y}^v_t(\theta)=-\theta \bar{Y}^v_t(\theta)\,\diff t+b(\mu[\bar{Y}^v_t])\,\diff t+\sigma(\mu[\bar{Y}^v_t])\,(\diff W_t+v_t\,\diff t),\ t>0,\ \theta\in\supp,\\
	\bar{Y}^v_0(\theta)=\bar{y}(\theta),\ \theta\in\supp.
	\end{dcases}
\end{equation*}
By Girsanov's theorem, the process $\bar{W}^v_t:=W_t+\int^t_0v_s\,\diff s$, $t\geq0$, is a $d$-dimensional Brownian motion under the probability measure $\bQ^v$ on $(\Omega,\cF)$ defined by $\frac{\diff\bQ^v}{\diff\bP}:=R^v_\infty:=\lim_{t\to\infty}R^v_t$. Thus, by means of the Yamada--Watanabe theorem, we see that the law of $\bar{Y}^v$ under $\bQ^v$ is equal to the law of $Y^{\bar{y}}$ under $\bP$. Denote by $\bE$ and $\bE_{\bQ^v}$ the expectations with respect to $\bP$ and $\bQ^v$, respectively. Consequently, for any $t\geq0$ and any $f\in\Bb$ such that $f\geq1$ and $\|\nabla\log f\|_\infty<\infty$, we have
\begin{align*}
	P_t\log f(\bar{y})&=\bE\big[\log f(Y^{\bar{y}}_t)\big]=\bE_{\bQ^v}\big[\log f(\bar{Y}^v_t)\big]=\bE\big[R^v_t\log f(Y^y_t)\big]-\bE_{\bQ^v}\big[\log f(Y^y_t)-\log f(\bar{Y}^v_t)\big]\\
	&\leq\log\bE\big[f(Y^y_t)\big]+\bE\big[R^v_t\log R^v_t\big]+\bE_{\bQ^v}\big[\|Y^y_t-\bar{Y}^v_t\|_{\cH_\mu}\big]\|\nabla\log f\|_\infty\\
	&=\log P_tf(y)+\bE\big[R^v_t\log R^v_t\big]+\bE_{\bQ^v}\big[\|Y^y_t-\bar{Y}^v_t\|_{\cH_\mu}\big]\|\nabla\log f\|_\infty,
\end{align*}
where we used Young's inequality (cf.\ \cite[Lemma 2.4]{ArThWa09}) in the second line. Therefore, the problem is reduced to \emph{finding a control process $v$ which controls the (asymptotic) behaviour of $\bE_{\bQ^v}\big[\|Y^y_t-\bar{Y}^v_t\|_{\cH_\mu}\big]$ with a small relative entropy cost $\bE\big[R^v_t\log R^v_t\big]$}. In the next subsection, we construct such a control process $v$.

%% Remark

\begin{rem}\label{Harnack_rem_previous}
Hong, Li and Liu \cite{HoLiLi20} and Liu \cite{Li20} derived asymptotic log-Harnack inequalities for monotone SPDEs with multiplicative noise. Unfortunately, the results (or the methods) of \cite{HoLiLi20,Li20} cannot be applied to the lifted SEE \eqref{lift_eq_SEE} or its variational form \eqref{lift_eq_variational} by the following reasons:
\begin{itemize}
\item
In \cite[Theorem 2.1]{HoLiLi20}, the assumption that the nonnegative closed operator $-A$ has infinitely many and sufficiently large eigenvalues is crucial. This is related to the so-called ``essential ellipticity condition'' introduced by Hairer and Mattingly \cite{HaMa06}. However, such a spectral condition does not necessarily hold in our framework. Indeed, if the Radon measure $\mu$ is diffusive (e.g., the fractional kernel and the Gamma kernel), then $-A$ has no eigenvalues; see \cref{lift_rem_space} (iii). In addition, \cite[Theorem 2.1]{HoLiLi20} requires the Lipschitz continuity of the (multiplicative) diffusion coefficient $\sigma$ as a map from the state space to the space of Hilbert--Schmidt operator, which is not the case in our framework due to the non-continuity of the integral operator $y\mapsto\mu[y]$ on the Hilbert space $\cH_\mu$.
\item
\cite[Theorem 3.1]{Li20} requires that the diffusion coefficient, which is a Hilbert--Schmidt operator from the noise space to the state space, has a right inverse, and hence surjective. This means that the noise (Brownian motion) has to be infinite dimensional when the state space is infinite dimensional. A difficulty of our framework is that the lifted SEE \eqref{lift_eq_SEE} is \emph{highly degenerate}; the state space $\cH_\mu$ is (typically) infinite dimensional, while the Brownian motion $W$ is finite dimensional.
\end{itemize}
We have to control the asymptotic behaviour of an infinite dimensional state process by a finite dimensional ($d$-dimensional) control process. This is of course a very non-trivial (or even impossible) task in general\footnote{There are some positive results in this direction, but their methods heavily depend on the special structures considered therein. See \hyperref[intro]{Introduction}.}. Nevertheless, by means of the special structure of the lifted SEE \eqref{lift_eq_SEE}, we succeed to obtain a positive result.
\end{rem}

%%%%%%%%%%%%%%
%% Subsection
%%%%%%%%%%%%%%

\subsection{Proof of \cref{Harnack_theo_Harnack}}

Now we prove \cref{Harnack_theo_Harnack}. Suppose that \cref{lift_assum_kernel} and \cref{Harnack_assum_coefficients} hold. Let $y,\bar{y}\in\cH_\mu$ be fixed. Let $Y:=Y^y$ be the mild solution of the lifted SEE \eqref{lift_eq_SEE} with the initial condition $y$. By \cref{lift_theo_bound}, we have
\begin{equation*}
	\bE\Big[\sup_{t\in[0,T]}\|Y_t\|^2_{\cH_\mu}+\int^T_0\|Y_t\|^2_{\cV_\mu}\,\diff t\Big]<\infty
\end{equation*}
for any $T>0$. We introduce the following equation:
\begin{equation}\label{Harnack_eq_coupling}
	\begin{dcases}
	\diff\bar{Y}_t(\theta)=-\theta\bar{Y}_t(\theta)\,\diff t+\Big\{b(\mu[\bar{Y}_t])+\lambda\sigma(\mu[\bar{Y}_t])\sigma^{-1}(\mu[Y_t])\int_\supp r(m\vee\theta')(Y_t(\theta')-\bar{Y}_t(\theta'))\dmud\Big\}\,\diff t\\
	\hspace{6cm}+\sigma(\mu[\bar{Y}_t])\,\diff W_t,\ \ t>0,\ \theta\in\supp,\\
	\bar{Y}_0(\theta)=\bar{y}(\theta),\ \theta\in\supp,
	\end{dcases}
\end{equation}
where $m\geq1$ is an arbitrary constant satisfying \eqref{Harnack_eq_m}, and $\lambda=\lambda(\mu,L,m)\geq1$ is defined by
\begin{equation}\label{Harnack_eq_lambda}
	\lambda:=1+2L^2\Big(1+\int_\supp r(\theta)\dmu\Big)r(m)^{-2}.
\end{equation}
Here, we remark that, by \cref{lift_assum_kernel} (in particular, by the assumptions that $r:[0,\infty)\to(0,\infty)$ is non-increasing and $r\leq1$), we have
\begin{equation}\label{Harnack_eq_r}
	r(m)r(\theta)\leq r(m\vee\theta)\leq r(\theta)
\end{equation}
for any $\theta\geq0$. Thus, the map $y\mapsto\int_\supp r(m\vee\theta)y(\theta)\dmu$ is a bounded linear operator from $\cH_\mu$ to $\bR^n$.

Observe that $\eqref{Harnack_eq_coupling}$ is a generalized SEE of the form \eqref{lift_eq_general-SEE}. The diffusion coefficient $\sigma:\bR^n\to\bR^{n\times d}$ is globally Lipschitz continuous, and the drift coefficient $\tilde{b}:\Omega\times[0,\infty)\times\bR^n\times\cH_\mu\to\bR^n$ defined by
\begin{equation*}
	\tilde{b}(t,x,y):=b(x)+\lambda\sigma(x)\sigma^{-1}(\mu[Y_t])\int_\supp r(m\vee\theta)(Y_t(\theta)-y(\theta))\dmu,\ t\geq0,\ x\in\bR^n,\ y\in\cH_\mu,
\end{equation*}
satisfies \cref{lift_assum_general-SEE} (i) (the measurability condition), (ii) (the linear growth condition), and (iii)' (the local Lipschitz condition). First, the measurability condition is trivial. Second, observe that, for any $t\geq0$, $x\in\bR^n$ and $y\in\cH_\mu$,
\begin{align*}
	|\tilde{b}(t,x,y)|&\leq|b(0)|+|b(x)-b(0)|+\lambda\|\sigma\|_\infty\|\sigma^{-1}\|_\infty\int_\supp r(\theta)\big(|Y_t(\theta)|+|y(\theta)|\big)\dmu\\
	&\leq|b(0)|+L|x|+\lambda\|\sigma\|_\infty\|\sigma^{-1}\|_\infty\Big(\int_\supp r(\theta)\dmu\Big)^{1/2}\big(\|Y_t\|_{\cH_\mu}+\|y\|_{\cH_\mu}\big),
\end{align*}
where we used \eqref{Harnack_eq_r}. Since $\bE\big[\int^T_0\|Y_t\|^2_{\cH_\mu}\,\diff t\big]<\infty$ for any $T>0$, we see that \cref{lift_assum_general-SEE} (ii) (the linear growth condition) holds. And finally, for any $t\geq0$, $x,\bar{x}\in\bR^n$ and $y,\bar{y}\in\cH_\mu$,
\begin{align*}
	|\tilde{b}(t,x,y)-\tilde{b}(t,\bar{x},\bar{y})|&\leq|b(x)-b(\bar{x})|+\lambda|\sigma(x)-\sigma(\bar{x})||\sigma^{-1}(\mu[Y_t])|\int_\supp r(\theta)\big(|Y_t(\theta)|+|y(\theta)|\big)\dmu\\
	&\hspace{0.5cm}+\lambda|\sigma(\bar{x})|\,|\sigma^{-1}(\mu[Y_t])|\int_\supp r(\theta)|y(\theta)-\bar{y}(\theta)|\dmu\\
	&\leq \Big\{1+\lambda\|\sigma^{-1}\|_\infty\Big(\int_\supp r(\theta)\dmu\Big)^{1/2}\big(\|Y_t\|_{\cH_\mu}+\|y\|_{\cH_\mu}\big)\Big\}L|x-\bar{x}|\\
	&\hspace{0.5cm}+\lambda\|\sigma\|_\infty\|\sigma^{-1}\|_\infty\Big(\int_\supp r(\theta)\dmu\Big)^{1/2}\|y-\bar{y}\|_{\cH_\mu}.
\end{align*}
Thus, \cref{lift_assum_general-SEE} (iii)' (the local Lipschitz condition) holds with the stopping times $\tau_k:=\inf\{t\geq0\,|\,\|Y_t\|_{\cH_\mu}>k\}$, $k\in\bN$, and suitable constants $L_k>0$, $k\in\bN$. Therefore, by \cref{lift_theo_local-Lip}, there exists a unique mild solution $\bar{Y}$ of the generalized SEE \eqref{Harnack_eq_coupling}. By \cref{lift_theo_bound}, we have
\begin{equation*}
	\bE\Big[\sup_{t\in[0,T]}\|\bar{Y}_t\|^2_{\cH_\mu}+\int^T_0\|\bar{Y}_t\|^2_{\cV_\mu}\,\diff t\Big]<\infty
\end{equation*}
for any $T>0$. We may assume that $Y$ and $\bar{Y}$ are (jointly measurable) families of $\bR^n$-valued It\^{o} processes; see \cref{lift_rem_Ito}.

Define a control process $v$ by
\begin{equation*}
	v_t:=\lambda\sigma^{-1}(\mu[Y_t])\int_\supp r(m\vee\theta)(Y_t(\theta)-\bar{Y}_t(\theta))\dmu,\ t\geq0.
\end{equation*}
Then $v$ is an $\bR^d$-valued predictable process satisfying $\bE\big[\int^T_0|v_t|^2\,\diff t\big]<\infty$ for any $T>0$. Thus, the stochastic exponential
\begin{equation*}
	R_t:=\exp\Big(-\int^t_0\langle v_s,\diff W_s\rangle-\frac{1}{2}\int^t_0|v_s|^2\,\diff s\Big),\ t\geq0,
\end{equation*}
is well-defined and is a local martingale under the probability measure $\bP$. Furthermore, we have:

%% Lemma

\begin{lemm}\label{Harnack_lemm_R}
The stochastic exponential $R$ is a uniformly integrable martingale under the probability measure $\bP$. Furthermore, it holds that
\begin{equation}\label{Harnack_eq_R}
	\sup_{t\in[0,\infty)}\bE\big[R_t\log R_t\big]\leq\frac{1}{2}\|\sigma^{-1}\|^2_\infty\Big(1+2L^2\Big(1+\int_\supp r(\theta)\dmu\Big)r(m)^{-2}\Big)\|y-\bar{y}\|^2_{\cH_\mu}.
\end{equation}
\end{lemm}

%% Proof

\begin{proof}
For each $k\in\bN$, define
\begin{equation*}
	\tau_k:=\inf\{t\geq0\,|\,\|Y_t-\bar{Y}_t\|_{\cH_\mu}>k\}.
\end{equation*}
Then each $\tau_k$ is a stopping time and satisfies $\tau_k\leq\tau_{k+1}$ for any $k\in\bN$ and $\lim_{k\to\infty}\tau_k=\infty$ a.s. Let $k\in\bN$ be fixed. Since the process $v$ is bounded on $[0,\tau_k]$, the stopped process $R_{\cdot\wedge\tau_k}$ is a martingale under the probability measure $\bP$. Define a probability measure $\bQ_k$ on $(\Omega,\cF)$ by $\frac{\diff\bQ_k}{\diff\bP}:=R_{\tau_k}$. By Girsanov's theorem, the process
\begin{equation*}
	\bar{W}^k_t:=W_t+\int^{t\wedge\tau_k}_0v_s\,\diff s,\ t\geq0,
\end{equation*}
is a $d$-dimensional Brownian motion under $\bQ_k$. Observe that, for each $\theta\in\supp$, the processes $Y(\theta)$ and $\bar{Y}(\theta)$ satisfy
\begin{equation*}
	\diff Y_t(\theta)=-\theta Y_t(\theta)\,\diff t-\lambda\int_\supp r(m\vee\theta')(Y_t(\theta')-\bar{Y}_t(\theta'))\dmud\,\diff t+b(\mu[Y_t])\,\diff t+\sigma(\mu[Y_t])\,\diff\bar{W}^k_t
\end{equation*}
and
\begin{equation*}
	\diff\bar{Y}_t(\theta)=-\theta\bar{Y}_t(\theta)\,\diff t+b(\mu[\bar{Y}_t])\,\diff t+\sigma(\mu[\bar{Y}_t])\,\diff\bar{W}^k_t
\end{equation*}
for $t\in[0,\tau_k]$. Thus, by It\^{o}'s formula, we have
\begin{align*}
	\diff e^{\beta t}|Y_t(\theta)-\bar{Y}_t(\theta)|^2=&-(2\theta-\beta)e^{\beta t}|Y_t(\theta)-\bar{Y}_t(\theta)|^2\,\diff t\\
	&-2\lambda e^{\beta t}\Big\langle Y_t(\theta)-\bar{Y}_t(\theta),\int_\supp r(m\vee\theta')(Y_t(\theta')-\bar{Y}_t(\theta'))\dmud\Big\rangle\,\diff t\\
	&+2e^{\beta t}\langle Y_t(\theta)-\bar{Y}_t(\theta),b(\mu[Y_t])-b(\mu[\bar{Y}_t])\rangle\,\diff t+e^{\beta t}|\sigma(\mu[Y_t])-\sigma(\mu[\bar{Y}_t])|^2\,\diff t\\
	&+2e^{\beta t}\langle Y_t(\theta)-\bar{Y}_t(\theta),(\sigma(\mu[Y_t])-\sigma(\mu[\bar{Y}_t]))\,\diff \bar{W}^k_t\rangle
\end{align*}
for $t\in[0,\tau_k]$, where $\beta:=\inf\supp\geq0$. We integrate both sides with respect to $r(m\vee\theta)\dmu$. Note that, by \eqref{Harnack_eq_r},
\begin{align*}
	&\int_\supp\Big(\int^T_0e^{2\beta t}|Y_t(\theta)-\bar{Y}_t(\theta)|^2|\sigma(\mu[Y_t])-\sigma(\mu[\bar{Y}_t])|^2\,\diff t\Big)^{1/2}r(m\vee\theta)\dmu\\
	&\leq\int_\supp\Big(\int^T_0e^{2\beta t}|Y_t(\theta)-\bar{Y}_t(\theta)|^2|\sigma(\mu[Y_t])-\sigma(\mu[\bar{Y}_t])|^2\,\diff t\Big)^{1/2}r(\theta)\dmu\\
	&\leq e^{\beta T}\Big(\int_\supp r(\theta)\dmu\Big)^{1/2}\Big(\int^T_0\|Y_t-\bar{Y}_t\|^2_{\cH_\mu}|\sigma(\mu[Y_t])-\sigma(\mu[\bar{Y}_t])|^2\,\diff t\Big)^{1/2}\\
	&\leq2T^{1/2}\|\sigma\|_\infty e^{\beta T}\Big(\int_\supp r(\theta)\dmu\Big)^{1/2}\sup_{t\in[0,T]}\|Y_t-\bar{Y}_t\|_{\cH_\mu}\\
	&<\infty\ \ \text{a.s.}
\end{align*}
for any $T>0$. Thus, we can use the stochastic Fubini theorem (cf.\ \cite{Ve12}) and obtain
\begin{align*}
	&\diff e^{\beta t}\int_\supp r(m\vee\theta)|Y_t(\theta)-\bar{Y}_t(\theta)|^2\dmu\\
	&=-e^{\beta t}\int_\supp(2\theta-\beta)r(m\vee\theta)|Y_t(\theta)-\bar{Y}_t(\theta)|^2\dmu\,\diff t\\
	&\hspace{0.5cm}-2\lambda e^{\beta t}\Big|\int_\supp r(m\vee\theta)(Y_t(\theta)-\bar{Y}_t(\theta))\dmu\Big|^2\,\diff t\\
	&\hspace{0.5cm}+2e^{\beta t}\Big\langle\int_\supp r(m\vee\theta)(Y_t(\theta)-\bar{Y}_t(\theta))\dmu,b(\mu[Y_t])-b(\mu[\bar{Y}_t])\Big\rangle\,\diff t\\
	&\hspace{0.5cm}+e^{\beta t}\int_\supp r(m\vee\theta)\dmu\,|\sigma(\mu[Y_t])-\sigma(\mu[\bar{Y}_t])|^2\,\diff t\\
	&\hspace{0.5cm}+2e^{\beta t}\Big\langle\int_\supp r(m\vee\theta)(Y_t(\theta)-\bar{Y}_t(\theta))\dmu,(\sigma(\mu[Y_t])-\sigma(\mu[\bar{Y}_t]))\,\diff\bar{W}^k_t\Big\rangle
\end{align*}
for $t\in[0,\tau_k]$. Noting that $Y_0(\theta)=y(\theta)$ and $\bar{Y}_0(\theta)=\bar{y}(\theta)$ for $\mu$-a.e.\ $\theta\in\supp$, we have
\begin{align*}
	&e^{\beta(t\wedge\tau_k)}\int_\supp r(m\vee\theta)|Y_{t\wedge\tau_k}(\theta)-\bar{Y}_{t\wedge\tau_k}(\theta)|^2\dmu\\
	&+\int^{t\wedge\tau_k}_0e^{\beta s}\int_\supp(2\theta-\beta)r(m\vee\theta)|Y_s(\theta)-\bar{Y}_s(\theta)|^2\dmu\,\diff s\\
	&+2\lambda\int^{t\wedge\tau_k}_0e^{\beta s}\Big|\int_\supp r(m\vee\theta)(Y_s(\theta)-\bar{Y}_s(\theta))\dmu\Big|^2\,\diff s\\
	&=\int_\supp r(m\vee\theta)|y(\theta)-\bar{y}(\theta)|^2\dmu\\
	&\hspace{0.5cm}+\int^{t\wedge\tau_k}_0e^{\beta s}\Big\{2\Big\langle\int_\supp r(m\vee\theta)(Y_s(\theta)-\bar{Y}_s(\theta))\dmu,b(\mu[Y_s])-b(\mu[\bar{Y}_s])\Big\rangle\\
	&\hspace{3cm}+\int_\supp r(m\vee\theta)\dmu\,|\sigma(\mu[Y_s])-\sigma(\mu[\bar{Y}_s])|^2\Big\}\,\diff s\\
	&\hspace{0.5cm}+2\int^{t\wedge\tau_k}_0e^{\beta s}\Big\langle\int_\supp r(m\vee\theta)(Y_s(\theta)-\bar{Y}_s(\theta))\dmu,(\sigma(\mu[Y_s])-\sigma(\mu[\bar{Y}_s]))\,\diff\bar{W}^k_s\Big\rangle
\end{align*}
for any $t\geq0$ a.s. The integrand of the Lebesgue integral in the right-hand side is estimated as follows: By using the Lipschitz conditions for $b$ and $\sigma$ and the inequality \eqref{Harnack_eq_r}, we have
\begin{align*}
	&2\Big\langle\int_\supp r(m\vee\theta)(Y_s(\theta)-\bar{Y}_s(\theta))\dmu,b(\mu[Y_s])-b(\mu[\bar{Y}_s])\Big\rangle\\
	&\hspace{0.5cm}+\int_\supp r(m\vee\theta)\dmu\,|\sigma(\mu[Y_s])-\sigma(\mu[\bar{Y}_s])|^2\\
	&\leq\Big|\int_\supp r(m\vee\theta)(Y_s(\theta)-\bar{Y}_s(\theta))\dmu\Big|^2+|b(\mu[Y_s])-b(\mu[\bar{Y}_s])|^2\\
	&\hspace{1cm}+\int_\supp r(m\vee\theta)\dmu\,|\sigma(\mu[Y_s])-\sigma(\mu[\bar{Y}_s])|^2\\
	&\leq\Big|\int_\supp r(m\vee\theta)(Y_s(\theta)-\bar{Y}_s(\theta))\dmu\Big|^2+L^2\Big(1+\int_\supp r(\theta)\dmu\Big)\,|\mu[Y_s]-\mu[\bar{Y}_s]|^2.
\end{align*}
Now we estimate the term $|\mu[Y_s]-\mu[\bar{Y}_s]|^2$ appearing in the last line above\footnote{This is the most important and technical point in the proof.}. Recall that $Y_s$ and $\bar{Y}_s$ belong to $\cV_\mu$ a.s.\ for a.e.\ $s\geq0$. Since $\frac{r(m\vee\theta)}{r(m)}=1$ for $\theta\in[0,m)$ and $\frac{r(m\vee\theta)}{r(m)}\in(0,1]$ for $\theta\in[m,\infty)$, we have
\begin{align*}
	&|\mu[Y_s]-\mu[\bar{Y}_s]|^2=\Big|\int_\supp(Y_s(\theta)-\bar{Y}_s(\theta))\dmu\Big|^2\\
	&=\Big|\int_\supp\frac{r(m\vee\theta)}{r(m)}(Y_s(\theta)-\bar{Y}_s(\theta))\dmu+\int_\supp\Big(1-\frac{r(m\vee\theta)}{r(m)}\Big)(Y_s(\theta)-\bar{Y}_s(\theta))\dmu\Big|^2\\
	&\leq2r(m)^{-2}\Big|\int_\supp r(m\vee\theta)(Y_s(\theta)-\bar{Y}_s(\theta))\dmu\Big|^2+2\Big(\int_{\supp\cap[m,\infty)}|Y_s(\theta)-\bar{Y}_s(\theta)|\dmu\Big)^2\\
	&\leq2r(m)^{-2}\Big|\int_\supp r(m\vee\theta)(Y_s(\theta)-\bar{Y}_s(\theta))\dmu\Big|^2\\
	&\hspace{1cm}+2\int_{\supp\cap[m,\infty)}\theta^{-1}r(\theta)^{-1}\dmu\int_{\supp\cap[m,\infty)}\theta r(\theta)|Y_s(\theta)-\bar{Y}_s(\theta)|^2\dmu\\
	&\leq2r(m)^{-2}\Big|\int_\supp r(m\vee\theta)(Y_s(\theta)-\bar{Y}_s(\theta))\dmu\Big|^2\\
	&\hspace{1cm}+2\int_{\supp\cap[m,\infty)}r(\theta)\dmu\int_\supp\theta r(m\vee\theta)|Y_s(\theta)-\bar{Y}_s(\theta)|^2\dmu,
\end{align*}
a.s.\ for a.e.\ $s\geq0$, where the last inequality follows from the estimate $\theta^{-1}r(\theta)^{-1}\leq\theta^{-1/2}\leq r(\theta)$ for $\theta\geq1$. Hence, recalling that the constant $m\geq1$ satisfies \eqref{Harnack_eq_m}, we obtain
\begin{align*}
	&e^{\beta(t\wedge\tau_k)}\int_\supp r(m\vee\theta)|Y_{t\wedge\tau_k}(\theta)-\bar{Y}_{t\wedge\tau_k}(\theta)|^2\dmu\\
	&+\int^{t\wedge\tau_k}_0e^{\beta s}\int_\supp(\theta-\beta)r(m\vee\theta)|Y_s(\theta)-\bar{Y}_s(\theta)|^2\dmu\,\diff s\\
	&+\Xi(\mu,L,m,\lambda)\int^{t\wedge\tau_k}_0e^{\beta s}\Big|\int_\supp r(m\vee\theta)(Y_s(\theta)-\bar{Y}_s(\theta))\dmu\Big|^2\,\diff s\\
	&\leq\int_\supp r(m\vee\theta)|y(\theta)-\bar{y}(\theta)|^2\dmu\\
	&\hspace{0.5cm}+2\int^{t\wedge\tau_k}_0e^{\beta s}\Big\langle\int_\supp r(m\vee\theta)(Y_s(\theta)-\bar{Y}_s(\theta))\dmu,(\sigma(\mu[Y_s])-\sigma(\mu[\bar{Y}_s]))\,\diff\bar{W}^k_s\Big\rangle
\end{align*}
for any $t\geq0$ a.s., where
\begin{equation}\label{Harnack_eq_C}
	\Xi(\mu,L,m,\lambda):=2\lambda-1-2L^2\Big(1+\int_\supp r(\theta)\dmu\Big)r(m)^{-2}>0.
\end{equation} Furthermore, noting \eqref{Harnack_eq_r} and $\beta=\inf\supp$, we have
\begin{align*}
	&r(m)e^{\beta(t\wedge\tau_k)}\|Y_{t\wedge\tau_k}-\bar{Y}_{t\wedge\tau_k}\|^2_{\cH_\mu}\\
	&+\Xi(\mu,L,m,\lambda)\int^{t\wedge\tau_k}_0e^{\beta s}\Big|\int_\supp r(m\vee\theta)(Y_s(\theta)-\bar{Y}_s(\theta))\dmu\Big|^2\,\diff s\\
	&\leq\|y-\bar{y}\|^2_{\cH_\mu}+2\int^{t\wedge\tau_k}_0e^{\beta s}\Big\langle\int_\supp r(m\vee\theta)(Y_s(\theta)-\bar{Y}_s(\theta))\dmu,(\sigma(\mu[Y_s])-\sigma(\mu[\bar{Y}_s]))\,\diff\bar{W}^k_s\Big\rangle
\end{align*}
for any $t\geq0$ a.s. Now we take the expectations with respect to $\bQ_k$ (we denote by $\bE_{\bQ_k}$ the expectation under $\bQ_k$). Note that, by the definition of the stopping time $\tau_k$,
\begin{align*}
	&\bE_{\bQ_k}\Big[\int^{t\wedge\tau_k}_0e^{2\beta s}\Big|\int_\supp r(m\vee\theta)(Y_s(\theta)-\bar{Y}_s(\theta))\dmu\Big|^2|\sigma(\mu[Y_s])-\sigma(\mu[\bar{Y}_s])|^2\,\diff s\Big]\\
	&\leq e^{2\beta t}\int_\supp r(\theta)\dmu\,\bE_{\bQ_k}\Big[\int^{t\wedge\tau_k}_0\|Y_s-\bar{Y}_s\|^2_{\cH_\mu}|\sigma(\mu[Y_s])-\sigma(\mu[\bar{Y}_s])|^2\,\diff s\Big]\\
	&\leq4te^{2\beta t}k^2\|\sigma\|^2_\infty\int_\supp r(\theta)\dmu<\infty
\end{align*}
for any $t\geq0$. Thus, the expectation of the stochastic integral under $\bQ_k$ is zero, and we obtain
\begin{equation}\label{Harnack_eq_estimate}
\begin{split}
	&r(m)\bE_{\bQ_k}\big[e^{\beta(t\wedge\tau_k)}\|Y_{t\wedge\tau_k}-\bar{Y}_{t\wedge\tau_k}\|^2_{\cH_\mu}\big]\\
	&+\Xi(\mu,L,m,\lambda)\bE_{\bQ_k}\Big[\int^{t\wedge\tau_k}_0e^{\beta s}\Big|\int_\supp r(m\vee\theta)(Y_s(\theta)-\bar{Y}_s(\theta))\dmu\Big|^2\,\diff s\Big]\\
	&\leq\|y-\bar{y}\|^2_{\cH_\mu}
\end{split}
\end{equation}
for any $t\geq0$.

By the above estimate, noting that $\Xi(\mu,L,m,\lambda)>0$ and $\beta\geq0$, it holds that
\begin{equation*}
	\bE_{\bQ_k}\Big[\int^{t\wedge\tau_k}_0\Big|\int_\supp r(m\vee\theta)(Y_s(\theta)-\bar{Y}_s(\theta))\dmu\Big|^2\,\diff s\Big]\leq\frac{1}{\Xi(\mu,L,m,\lambda)}\|y-\bar{y}\|^2_{\cH_\mu}
\end{equation*}
for any $t\geq0$ and $k\in\bN$. Recall that the process $v$ is bounded on $[0,\tau_k]$, and $\bar{W}^k$ is a $\bQ_k$-Brownian motion. Thus, from the above observations, we have, for any $t\geq0$ and $k\in\bN$,
\begin{align*}
	\bE\big[R_{t\wedge\tau_k}\log R_{t\wedge\tau_k}\big]&=\bE_{\bQ_k}\big[\log R_{t\wedge\tau_k}\big]\\
	&=\bE_{\bQ_k}\Big[-\int^{t\wedge\tau_k}_0\langle v_s,\diff W_s\rangle-\frac{1}{2}\int^{t\wedge\tau_k}_0|v_s|^2\,\diff s\Big]\\
	&=\bE_{\bQ_k}\Big[-\int^{t\wedge\tau_k}_0\langle v_s,\diff \bar{W}^k_s\rangle+\frac{1}{2}\int^{t\wedge\tau_k}_0|v_s|^2\,\diff s\Big]\\
	&=\frac{1}{2}\bE_{\bQ_k}\Big[\int^{t\wedge\tau_k}_0|v_s|^2\,\diff s\Big]\\
	&\leq\frac{1}{2}\lambda^2\|\sigma^{-1}\|^2_\infty\bE_{\bQ_k}\Big[\int^{t\wedge\tau_k}_0\Big|\int_\supp r(m\vee\theta)(Y_s(\theta)-\bar{Y}_s(\theta))\dmu\Big|^2\,\diff s\Big]\\
	&\leq\frac{1}{2}\|\sigma^{-1}\|^2_\infty\frac{\lambda^2}{\Xi(\mu,L,m,\lambda)}\|y-\bar{y}\|^2_{\cH_\mu}.
\end{align*}
Recall also that the constant $\lambda=\lambda(\mu,L,m)$ is defined by \eqref{Harnack_eq_lambda} (which is the minimizer of the last term with respect to $\lambda$ such that $\Xi(\mu,L,m,\lambda)>0$; see \eqref{Harnack_eq_C}). Then we have
\begin{equation*}
	\bE\big[R_{t\wedge\tau_k}\log R_{t\wedge\tau_k}\big]\leq\frac{1}{2}\|\sigma^{-1}\|^2_\infty\Big(1+2L^2\Big(1+\int_\supp r(\theta)\dmu\Big)r(m)^{-2}\Big)\|y-\bar{y}\|^2_{\cH_\mu}
\end{equation*}
for any $t\geq0$ and any $k\in\bN$. In particular, the family $\{R_{t\wedge\tau_k}\}_{t\geq0,k\in\bN}$ is uniformly integrable under the probability measure $\bP$.

Let $0\leq s<t<\infty$ and $A\in\cF_s$ be given. Since each stopped process $R_{\cdot\wedge\tau_k}$ is a martingales under $\bP$, we have $\bE[R_{t\wedge\tau_k}\1_A]=\bE[R_{s\wedge\tau_k}\1_A]$ for any $k\in\bN$. Noting that $\{R_{t\wedge\tau_k}\}_{k\in\bN}$ and $\{R_{s\wedge\tau_k}\}_{k\in\bN}$ are uniformly integrable under $\bP$, by Vitali's convergence theorem, we have
\begin{equation*}
	\bE[R_t\1_A]=\lim_{k\to\infty}\bE[R_{t\wedge\tau_k}\1_A]=\lim_{k\to\infty}\bE[R_{s\wedge\tau_k}\1_A]=\bE[R_s\1_A].
\end{equation*}
Hencce, $R$ is a martingale under $\bP$. Furthermore, Fatou's lemma yields that
\begin{equation*}
	\bE\big[R_t\log R_t\big]\leq\liminf_{k\to\infty}\bE\big[R_{t\wedge\tau_k}\log R_{t\wedge\tau_k}\big]\leq\frac{1}{2}\|\sigma^{-1}\|^2_\infty\Big(1+2L^2\Big(1+\int_\supp r(\theta)\dmu\Big)r(m)^{-2}\Big)\|y-\bar{y}\|^2_{\cH_\mu}
\end{equation*}
for any $t\geq0$. Thus, the estimate \eqref{Harnack_eq_R} holds. Finally, this estimate implies that $R$ is uniformly integrable under $\bP$. This completes the proof.
\end{proof}

%% Proof

\begin{proof}[Proof of \cref{Harnack_theo_Harnack}]
By \cref{Harnack_lemm_R}, we see that the limit $R_\infty:=\lim_{t\to\infty}R_t$ exists in $L^1(\bP)$, and the measure $\bQ$ on $(\Omega,\cF)$ defined by $\frac{\diff\bQ}{\diff\bP}:=R_\infty$ is a probability measure. In this proof, we denote by $\bE$ and $\bE_\bQ$ the expectations under $\bP$ and $\bQ$, respectively. By Girsanov's theorem, the process
\begin{equation*}
	\bar{W}_t:=W_t+\int^t_0v_s\,\diff s,\ t\geq0,
\end{equation*}
is a $d$-dimensional Brownian motion under $\bQ$. Observe that $\bar{Y}$ is a mild solution of the SEE
\begin{equation*}
	\begin{dcases}
	\diff\bar{Y}_t(\theta)=-\theta\bar{Y}_t(\theta)\,\diff t+b(\mu[\bar{Y}_t])\,\diff t+\sigma(\mu[\bar{Y}_t])\,\diff\bar{W}_t,\ t>0,\ \theta\in\supp,\\
	\bar{Y}_0(\theta)=\bar{y}(\theta),\ \theta\in\supp.
	\end{dcases}
\end{equation*}
By the uniqueness in law of the SEE \eqref{lift_eq_SEE}, which follows from the pathwise uniqueness and the Yamada--Watanabe theorem (see \cite{Ku14} for the general form of the Yamada--Watanabe theorem), we see that the law of $\bar{Y}$ under $\bQ$ coincides with the law of $Y^{\bar{y}}$ under $\bP$. Furthermore, by \eqref{Harnack_eq_estimate}, we have (with the same notations defined in the proof of \cref{Harnack_lemm_R})
\begin{equation*}
	\bE_\bQ\big[e^{\beta(t\wedge\tau_k)}\|Y_{t\wedge\tau_k}-\bar{Y}_{t\wedge\tau_k}\|^2_{\cH_\mu}\big]=\bE_{\bQ_k}\big[e^{\beta(t\wedge\tau_k)}\|Y_{t\wedge\tau_k}-\bar{Y}_{t\wedge\tau_k}\|^2_{\cH_\mu}\big]\leq r(m)^{-1}\|y-\bar{y}\|^2_{\cH_\mu}
\end{equation*}
for any $t\geq0$ and any $k\in\bN$. Since $\lim_{k\to\infty}\tau_k=\infty$ a.s., by Fatou's lemma, we obtain
\begin{equation*}
	\bE_\bQ\big[e^{\beta t}\|Y_t-\bar{Y}_t\|^2_{\cH_\mu}\big]\leq\liminf_{k\to\infty}\bE_\bQ\big[e^{\beta(t\wedge\tau_k)}\|Y_{t\wedge\tau_k}-\bar{Y}_{t\wedge\tau_k}\|^2_{\cH_\mu}\big]\leq r(m)^{-1}\|y-\bar{y}\|^2_{\cH_\mu},
\end{equation*}
and thus,
\begin{equation}\label{Harnack_eq_remainder}
	\bE_\bQ\big[\|Y_t-\bar{Y}_t\|_{\cH_\mu}\big]\leq\bE_\bQ\big[\|Y_t-\bar{Y}_t\|^2_{\cH_\mu}\big]^{1/2}\leq r(m)^{-1/2}e^{-\beta t/2}\|y-\bar{y}\|_{\cH_\mu}
\end{equation}
for any $t\geq0$.

Consequently, by \eqref{Harnack_eq_R} and \eqref{Harnack_eq_remainder}, for any $t\geq0$ and any $f\in\Bb$ such that $f\geq1$ and $\|\nabla\log f\|_\infty<\infty$, we have
\begin{align*}
	P_t\log f(\bar{y})&=\bE\big[\log f(Y^{\bar{y}}_t)\big]=\bE_\bQ\big[\log f(\bar{Y}_t)\big]=\bE\big[R_t\log f(Y_t)\big]-\bE_\bQ\big[\log f(Y_t)-\log f(\bar{Y}_t)\big]\\
	&\leq\log\bE\big[f(Y_t)\big]+\bE\big[R_t\log R_t\big]+\bE_\bQ\big[\|Y_t-\bar{Y}_t\|_{\cH_\mu}\big]\|\nabla\log f\|_\infty\\
	&\leq\log P_tf(y)+\frac{1}{2}\|\sigma^{-1}\|^2_\infty\Big(1+2L^2\Big(1+\int_\supp r(\theta)\dmu\Big)r(m)^{-2}\Big)\|y-\bar{y}\|^2_{\cH_\mu}\\
	&\hspace{3cm}+r(m)^{-1/2}e^{-\beta t/2}\|y-\bar{y}\|_{\cH_\mu}\|\nabla\log f\|_\infty.
\end{align*}
This completes the proof of \cref{Harnack_theo_Harnack}.
\end{proof}

%%%%%%%%%%%%%%%%%%%%%%%%%%%%%%%%%%
%%%%%%%%%%%%%%%%%%%%%%%%%%%%%%%%%%
%% Appendix
%%%%%%%%%%%%%%%%%%%%%%%%%%%%%%%%%%
%%%%%%%%%%%%%%%%%%%%%%%%%%%%%%%%%%

\appendix
\setcounter{theo}{0}
\setcounter{equation}{0}

\section{Appendix}\label{appendix}

In this appendix, we provide proofs of the well-posedness of the generalized SEE \eqref{lift_eq_general-SEE}.

%%%%%%%%%%%%%%%%%
%% Subsection
%%%%%%%%%%%%%%%%%

\subsection{Proof of \cref{lift_theo_bound}: A priori bound of mild solutions}\label{appendix_1}

Suppose that \cref{lift_assum_kernel} holds. Let $b:\Omega\times[0,\infty)\times\bR^n\times\cH_\mu\to\bR^n$ and $\sigma:\Omega\times[0,\infty)\times\bR^n\times\cH_\mu\to\bR^{n\times d}$ satisfy \cref{lift_assum_general-SEE} (i) (the measurability condition) and (ii) (the linear growth condition).

%% Proof

\begin{proof}[Proof of \cref{lift_theo_bound}]
Suppose that $Y$ is a mild solution of the generalized SEE \eqref{lift_eq_general-SEE} with the initial condition $\eta\in L^2_{\cF_0}(\cH_\mu)$. In this proof, we denote $X:=\mu[Y]$, which is an $\bR^n$-valued predicable process and satisfies $X_t=\int_\supp Y_t(\theta)\dmu$ a.s.\ for a.e.\ $t>0$. By \cref{lift_lemm_space} (iii), we have $\int^T_0|X_t|^2\,\diff t<\infty$ a.s.\ for any $T>0$. Noting \cref{lift_rem_Ito}, we may assume that, for each $\theta\in\supp$, the process $Y(\theta)$ is an $\bR^n$-valued It\^{o} process satisfying
\begin{equation*}
	\begin{dcases}
	\diff Y_t(\theta)=-\theta Y_t(\theta)\,\diff t+b(t,X_t,Y_t)\,\diff t+\sigma(t,X_t,Y_t)\,\diff W_t,\ t>0,\\
	Y_0(\theta)=\eta(\theta).
	\end{dcases}
\end{equation*}
Let $\lambda>0$ be a constant, which we will identify later. By It\^{o}'s formula, we have
\begin{align*}
	\diff e^{-\lambda t}|Y_t(\theta)|^2=&-\lambda e^{-\lambda t}|Y_t(\theta)|^2\,\diff t-2e^{-\lambda t}\theta|Y_t(\theta)|^2\,\diff t+2e^{-\lambda t}\langle Y_t(\theta),b(t,X_t,Y_t)\rangle\,\diff t+e^{-\lambda t}|\sigma(t,X_t,Y_t)|^2\,\diff t\\
	&+2e^{-\lambda t}\langle Y_t(\theta),\sigma(t,X_t,Y_t)\,\diff W_t\rangle.
\end{align*}
We integrate both sides with respect to $r(\theta)\dmu$. Since
\begin{align*}
	&\int_\supp\Big(\int^T_0e^{-2\lambda t}|Y_t(\theta)|^2|\sigma(t,X_t,Y_t)|^2\,\diff t\Big)^{1/2}\,r(\theta)\dmu\\
	&\leq\Big(\int_\supp r(\theta)\dmu\Big)^{1/2}\Big(\int_\supp\int^T_0e^{-2\lambda t}|Y_t(\theta)|^2|\sigma(t,X_t,Y_t)|^2\,\diff t\,r(\theta)\dmu\Big)^{1/2}\\
	&=\Big(\int_\supp r(\theta)\dmu\Big)^{1/2}\Big(\int^T_0e^{-2\lambda t}\|Y_t\|^2_{\cH_\mu}|\sigma(t,X_t,Y_t)|^2\,\diff t\Big)^{1/2}\\
	&\leq\frac{1}{2}\Big(\int_\supp r(\theta)\dmu\Big)^{1/2}\Big\{\sup_{t\in[0,T]}\|Y_t\|^2_{\cH_\mu}+\int^T_0|\sigma(t,X_t,Y_t)|^2\,\diff t\Big\}\\
	&<\infty\ \text{a.s.}
\end{align*}
for any $T>0$, we can use the stochastic Fubini theorem (cf.\ \cite{Ve12}) and obtain
\begin{align*}
	&\diff e^{-\lambda t}\int_\supp r(\theta)|Y_t(\theta)|^2\dmu\\
	&=-\lambda e^{-\lambda t}\int_\supp r(\theta)|Y_t(\theta)|^2\dmu\,\diff t-2e^{-\lambda t}\int_\supp\theta r(\theta)|Y_t(\theta)|^2\dmu\,\diff t\\
	&\hspace{0.5cm}+2e^{-\lambda t}\Big\langle\int_\supp r(\theta)Y_t(\theta)\dmu,b(t,X_t,Y_t)\Big\rangle\,\diff t+e^{-\lambda t}\int_\supp r(\theta)\dmu\,|\sigma(t,X_t,Y_t)|^2\,\diff t\\
	&\hspace{0.5cm}+2e^{-\lambda t}\Big\langle\int_\supp r(\theta)Y_t(\theta)\dmu,\sigma(t,X_t,Y_t)\,\diff W_t\Big\rangle,
\end{align*}
and thus
\begin{align*}
	e^{-\lambda t}\|Y_t\|^2_{\cH_\mu}&=\|\eta\|^2_{\cH_\mu}+\int^t_0e^{-\lambda s}\Big\{-(\lambda-2)\|Y_s\|^2_{\cH_\mu}-2\|Y_s\|^2_{\cV_\mu}+2\Big\langle\int_\supp r(\theta)Y_s(\theta)\dmu,b(s,X_s,Y_s)\Big\rangle\\
	&\hspace{4cm}+\int_\supp r(\theta)\dmu\,|\sigma(s,X_s,Y_s)|^2\Big\}\,\diff s\\
	&\hspace{1cm}+2\int^t_0e^{-\lambda s}\Big\langle\int_\supp r(\theta)Y_s(\theta)\dmu,\sigma(s,X_s,Y_s)\,\diff W_s\Big\rangle
\end{align*}
for any $t\geq0$ a.s. Noting that
\begin{equation*}
	\Big|\int_\supp r(\theta)Y_s(\theta)\dmu\Big|\leq\Big(\int_\supp r(\theta)\dmu\Big)^{1/2}\|Y_s\|_{\cH_\mu}
\end{equation*}
and using the inequality $2ab\leq a^2+b^2$, we get
\begin{equation}\label{lift_eq_Ito}
\begin{split}
	e^{-\lambda t}\|Y_t\|^2_{\cH_\mu}&\leq\|\eta\|^2_{\cH_\mu}+\int^t_0e^{-\lambda s}\Big\{-(\lambda-3)\|Y_s\|^2_{\cH_\mu}-2\|Y_s\|^2_{\cV_\mu}\\
	&\hspace{3.5cm}+\int_\supp r(\theta)\dmu\,\big(|b(s,X_s,Y_s)|^2+|\sigma(s,X_s,Y_s)|^2\big)\Big\}\,\diff s\\
	&\hspace{1cm}+2\int^t_0e^{-\lambda s}\Big\langle\int_\supp r(\theta)Y_s(\theta)\dmu,\sigma(s,X_s,Y_s)\,\diff W_s\Big\rangle
\end{split}
\end{equation}
for any $t\geq0$ a.s. By \cref{lift_assum_general-SEE} (ii) (the linear growth condition), we have
\begin{equation*}
	|b(s,X_s,Y_s)|^2+|\sigma(s,X_s,Y_s)|^2\leq 3\varphi^2_s+3c^2_\LG\big\{|X_s|^2+\|Y_s\|^2_{\cH_\mu}\big\}.
\end{equation*}
Furthermore, by \cref{lift_lemm_space} (iii), there exists a constant $M=M(\mu,c_\LG)>0$ such that
\begin{equation*}
	|X_s|^2=\Big|\int_\supp Y_s(\theta)\dmu\Big|^2\leq\frac{1}{3c^2_\LG\int_\supp r(\theta)\dmu}\|Y_s\|^2_{\cV_\mu}+M\|Y_s\|^2_{\cH_\mu}
\end{equation*}
a.s.\ for a.e.\ $s>0$. Hence, we obtain
\begin{align*}
	e^{-\lambda t}\|Y_t\|^2_{\cH_\mu}&\leq\|\eta\|^2_{\cH_\mu}+\int^t_0e^{-\lambda s}\Big\{-\Big(\lambda-3-3c^2_\LG(M+1)\int_\supp r(\theta)\dmu\Big)\|Y_s\|^2_{\cH_\mu}-\|Y_s\|^2_{\cV_\mu}\\
	&\hspace{3.5cm}+3\int_\supp r(\theta)\dmu\,\varphi^2_s\Big\}\,\diff s\\
	&\hspace{1cm}+2\int^t_0e^{-\lambda s}\Big\langle\int_\supp r(\theta)Y_s(\theta)\dmu,\sigma(s,X_s,Y_s)\,\diff W_s\Big\rangle
\end{align*}
for any $t\geq0$ a.s. Now we take the constant $\lambda=\lambda(\mu,c_\LG)>0$ by
\begin{equation*}
	\lambda=3+3c^2_\LG(M+1)\int_\supp r(\theta)\dmu.
\end{equation*}
Then we obtain
\begin{equation}\label{lift_eq_estimate0}
\begin{split}
	&e^{-\lambda t}\|Y_t\|^2_{\cH_\mu}+\int^t_0e^{-\lambda s}\|Y_s\|^2_{\cV_\mu}\,\diff s\\
	&\leq\|\eta\|^2_{\cH_\mu}+3\int_\supp r(\theta)\dmu\int^t_0e^{-\lambda s}\varphi^2_s\,\diff s+2\int^t_0e^{-\lambda s}\Big\langle\int_\supp r(\theta)Y_s(\theta)\dmu,\sigma(s,X_s,Y_s)\,\diff W_s\Big\rangle
\end{split}
\end{equation}
for any $t\geq0$ a.s.

Let $T>0$ and a stopping time $\tau$ be fixed. For each $k\in\bN$, define the stopping time
\begin{equation*}
	\tau_k:=\inf\left\{t\geq0\relmiddle|\|Y_t\|_{\cH_\mu}>k\ \text{or}\ \int^t_0|\sigma(s,X_s,Y_s)|^2\,\diff s>k\right\}\wedge\tau.
\end{equation*}
Then we have $\tau_k\leq\tau_{k+1}\leq\tau$ for any $k\in\bN$ and $\lim_{k\to\infty}\tau_k=\tau$ a.s. We fix $k\in\bN$. Noting that the stopped process
\begin{equation*}
	\int^{\cdot\wedge\tau_k}_0e^{-\lambda s}\Big\langle\int_\supp r(\theta)Y_s(\theta)\dmu,\sigma(s,X_s,Y_s)\,\diff W_s\Big\rangle
\end{equation*}
is a martingale, by taking the expectations, we get
\begin{equation}\label{lift_eq_estimate1}
	\bE\Big[\int^{T\wedge\tau_k}_0e^{-\lambda t}\|Y_t\|^2_{\cV_\mu}\,\diff t\Big]\leq\bE\Big[\|\eta\|^2_{\cH_\mu}+3\int_\supp r(\theta)\dmu\int^{T\wedge\tau_k}_0e^{-\lambda t}\varphi^2_t\,\diff t\Big].
\end{equation}
Furthermore, by \eqref{lift_eq_estimate0} and the Burkholder--Davis--Gundy inequality,
\begin{align*}
	\bE\Big[\sup_{t\in[0,T]}e^{-\lambda(t\wedge\tau_k)}\|Y_{t\wedge\tau_k}\|^2_{\cH_\mu}\Big]&\leq\bE\Big[\|\eta\|^2_{\cH_\mu}+3\int_\supp r(\theta)\dmu\int^{T\wedge\tau_k}_0e^{-\lambda t}\varphi^2_t\,\diff t\Big]\\
	&\hspace{0.5cm}+2\bE\Big[\sup_{t\in[0,T]}\Big|\int^{t\wedge\tau_k}_0e^{-\lambda s}\Big\langle\int_\supp r(\theta)Y_s(\theta)\dmu,\sigma(s,X_s,Y_s)\,\diff W_s\Big\rangle\Big|\Big]\\
	&\leq\bE\Big[\|\eta\|^2_{\cH_\mu}+3\int_\supp r(\theta)\dmu\int^{T\wedge\tau_k}_0e^{-\lambda t}\varphi^2_t\,\diff t\Big]\\
	&\hspace{0.5cm}+2c_\BDG\bE\Big[\Big(\int^{T\wedge\tau_k}_0e^{-2\lambda t}\Big|\int_\supp r(\theta)Y_t(\theta)\dmu\Big|^2|\sigma(t,X_t,Y_t)|^2\,\diff t\Big)^{1/2}\Big],
\end{align*}
where $c_\BDG>0$ is the constant appearing in the Burkholder--Davis--Gundy inequality. The last term can be estimated as follows:
\begin{align*}
	&2c_\BDG\bE\Big[\Big(\int^{T\wedge\tau_k}_0e^{-2\lambda t}\Big|\int_\supp r(\theta)Y_t(\theta)\dmu\Big|^2|\sigma(t,X_t,Y_t)|^2\,\diff t\Big)^{1/2}\Big]\\
	&\leq2c_\BDG\Big(\int_\supp r(\theta)\dmu\Big)^{1/2}\bE\Big[\Big(\int^{T\wedge\tau_k}_0e^{-2\lambda t}\|Y_t\|^2_{\cH_\mu}|\sigma(t,X_t,Y_t)|^2\,\diff t\Big)^{1/2}\Big]\\
	&\leq2c_\BDG\Big(\int_\supp r(\theta)\dmu\Big)^{1/2}\bE\Big[\Big(\sup_{t\in[0,T]}e^{-\lambda(t\wedge\tau_k)}\|Y_{t\wedge\tau_k}\|^2_{\cH_\mu}\Big)^{1/2}\Big(\int^{T\wedge\tau_k}_0e^{-\lambda t}|\sigma(t,X_t,Y_t)|^2\,\diff t\Big)^{1/2}\Big]\\
	&\leq\frac{1}{2}\bE\Big[\sup_{t\in[0,T]}e^{-\lambda(t\wedge\tau_k)}\|Y_{t\wedge\tau_k}\|^2_{\cH_\mu}\Big]+2c^2_\BDG\int_\supp r(\theta)\dmu\,\bE\Big[\int^{T\wedge\tau_k}_0e^{-\lambda t}|\sigma(t,X_t,Y_t)|^2\,\diff t\Big].
\end{align*}
Hence,
\begin{align*}
	\bE\Big[\sup_{t\in[0,T]}e^{-\lambda(t\wedge\tau_k)}\|Y_{t\wedge\tau_k}\|^2_{\cH_\mu}\Big]&\leq2\bE\Big[\|\eta\|^2_{\cH_\mu}+3\int_\supp r(\theta)\dmu\int^{T\wedge\tau_k}_0e^{-\lambda t}\varphi^2_t\,\diff t\\
	&\hspace{1cm}+2c^2_\BDG\int_\supp r(\theta)\dmu\int^{T\wedge\tau_k}_0e^{-\lambda t}|\sigma(t,X_t,Y_t)|^2\,\diff t\Big].
\end{align*}
Again by using \cref{lift_assum_general-SEE} (ii) (the linear growth condition), \cref{lift_lemm_space} (iii), and the fact that $\|\cdot\|_{\cH_\mu}\leq\|\cdot\|_{\cV_\mu}$, we have
\begin{align*}
	|\sigma(t,X_t,Y_t)|^2&\leq3\varphi^2_t+3c^2_\LG\big\{|X_t|^2+\|Y_t\|^2_{\cH_\mu}\big\}\\
	&=3\varphi^2_t+3c^2_\LG\Big\{\Big|\int_\supp Y_t(\theta)\dmu\Big|^2+\|Y_t\|^2_{\cH_\mu}\Big\}\\
	&\leq3\varphi^2_t+3c^2_\LG\Big(\int_\supp r(\theta)\dmu+1\Big)\|Y_t\|^2_{\cV_\mu}
\end{align*}
a.s.\ for a.e.\ $t>0$. Thus,
\begin{equation}\label{lift_eq_estimate2}
\begin{split}
	&\bE\Big[\sup_{t\in[0,T]}e^{-\lambda(t\wedge\tau_k)}\|Y_{t\wedge\tau_k}\|^2_{\cH_\mu}\Big]\\
	&\leq2\bE\Big[\|\eta\|^2_{\cH_\mu}+(3+6c^2_\BDG)\int_\supp r(\theta)\dmu\int^{T\wedge\tau_k}_0e^{-\lambda t}\varphi^2_t\,\diff t\\
	&\hspace{1cm}+6c^2_\BDG c^2_\LG\Big(\int_\supp r(\theta)\dmu+1\Big)^2\int^{T\wedge\tau_k}_0e^{-\lambda t}\|Y_t\|^2_{\cV_\mu}\,\diff t\Big].
	\end{split}
\end{equation}
From \eqref{lift_eq_estimate1} and \eqref{lift_eq_estimate2}, we see that there exists a constant $C=C(\mu,c_\LG)>0$ such that
\begin{equation*}
	\bE\Big[\sup_{t\in[0,T]}\|Y_{t\wedge\tau_k}\|^2_{\cH_\mu}+\int^{T\wedge\tau_k}_0\|Y_t\|^2_{\cV_\mu}\,\diff t\Big]\leq Ce^{CT}\bE\Big[\|\eta\|^2_{\cH_\mu}+\int^{T\wedge\tau_k}_0\varphi^2_t\,\diff t\Big]\leq Ce^{CT}\bE\Big[\|\eta\|^2_{\cH_\mu}+\int^{T\wedge\tau}_0\varphi^2_t\,\diff t\Big].
\end{equation*}
Since $\lim_{k\to\infty}\tau_k=\tau$ a.s., by Fatou's lemma, we get the assertion.
\end{proof}

%%%%%%%%%%%%%%%%%
%% Subsection
%%%%%%%%%%%%%%%%%

\subsection{Proof of \cref{lift_theo_global-Lip}: Existence, uniqueness and stability; the global Lipschitz case}\label{appendix_2}

Suppose that \cref{lift_assum_kernel} holds. Let $b:\Omega\times[0,\infty)\times\bR^n\times\cH_\mu\to\bR^n$ and $\sigma:\Omega\times[0,\infty)\times\bR^n\times\cH_\mu\to\bR^{n\times d}$ satisfy \cref{lift_assum_general-SEE} (i) (the measurability condition), (ii) (the linear growth condition) and (iii) (the global Lipschitz condition).

%% Proof

\begin{proof}[Proof of \cref{lift_theo_global-Lip}]
First, we show the existence and uniqueness of the mild solution. Let $\eta\in L^2_{\cF_0}(\cH_\mu)$ be given. Fix $T>0$, and denote by $\cL_T$ the space of equivalent classes of $\cH_\mu$-valued continuous and adapted processes $Y$ on $[0,T]$ such that
\begin{equation*}
	\|Y\|_T:=\bE\Big[\sup_{t\in[0,T]}\|Y_t\|^2_{\cH_\mu}+\int^T_0\|Y_t\|^2_{\cV_\mu}\,\diff t\Big]^{1/2}<\infty.
\end{equation*}
Then $(\cL_T,\|\cdot\|_T)$ is a Banach space. Note that, by \cref{lift_theo_bound}, any mild solution belongs to $\cL_T$. Then, by \cref{lift_assum_general-SEE} (ii) (the linear growth condition) and \cref{lift_lemm_convolution}, the $\cH_\mu$-valued process $\hat{Y}$ defined by
\begin{equation*}
	\hat{Y}_t:=e^{-\cdot t}\eta(\cdot)+\int^t_0e^{-\cdot(t-s)}b(s,\mu[Y_s],Y_s)\,\diff s+\int^t_0e^{-\cdot(t-s)}\sigma(s,\mu[Y_s],Y_s)\,\diff W_s,\ t\in[0,T],
\end{equation*}
belongs to $\cL_T$. Thus, we have a map $\Phi:\cL_T\to\cL_T$ given by $\Phi(Y):=\hat{Y}$. To show the existence and uniqueness of the mild solution of the generalized SEE \eqref{lift_eq_general-SEE}, it suffices to show that $\Phi$ has a unique fixed point in $\cL_T$. For this purpose, let $\lambda,\kappa>0$ be given, and define
\begin{equation*}
	\|Y\|_{T,\lambda,\kappa}:=\bE\Big[\sup_{t\in[0,T]}e^{-\lambda t}\|Y_t\|^2_{\cH_\mu}+\int^T_0e^{-\lambda t}\big\{\kappa\|Y_t\|^2_{\cH_\mu}+\|Y_t\|^2_{\cV_\mu}\big\}\,\diff t\Big]^{1/2}
\end{equation*}
for $Y\in\cL_T$. Clearly, $\|\cdot\|_{T,\lambda,\kappa}$ is a norm on $\cL_T$, and it is equivalent to the original norm $\|\cdot\|_T$. We shall show that, for some $\lambda,\kappa>0$ large enough, the map $\Phi$ is contractive with respect to the norm $\|\cdot\|_{T,\lambda,\kappa}$. Then by the Banach fixed point theorem we see that $\Phi$ has a unique fixed point in $\cL_T$. In the following, we take $\lambda>3$.

To show that $\Phi$ is contractive with respect to $\|\cdot\|_{T,\lambda,\kappa}$ for some $\lambda>3$ and $\kappa>0$, let $Y^1,Y^2\in\cL_T$, and denote $X^i:=\mu[Y^i]$ and $\hat{Y}^i:=\Phi(Y^i)$ for $i=1,2$. Then for $i=1,2$ and for each $\theta\in\supp$, the process $\hat{Y}^i(\theta)$ satisfies
\begin{equation*}
	\begin{dcases}
	\diff\hat{Y}^i_t(\theta)=-\theta\hat{Y}^i_t(\theta)\,\diff t+b(t,X^i_t,Y^i_t)\,\diff t+\sigma(t,X^i_t,Y^i_t)\,\diff W_t,\ t\in[0,T],\\
	\hat{Y}^i_0(\theta)=\eta(\theta).
	\end{dcases}
\end{equation*}
Then, by the same calculations as \eqref{lift_eq_Ito} in the proof of \cref{lift_theo_bound}, we obtain
\begin{align*}
	&e^{-\lambda t}\|\hat{Y}^1_t-\hat{Y}^2_t\|^2_{\cH_\mu}\\
	&\leq\int^t_0e^{-\lambda s}\Big\{-(\lambda-3)\|\hat{Y}^1_s-\hat{Y}^2_s\|^2_{\cH_\mu}-2\|\hat{Y}^1_s-\hat{Y}^2_s\|^2_{\cV_\mu}\\
	&\hspace{1.5cm}+\int_\supp r(\theta)\dmu\,\big(|b(s,X^1_s,Y^1_s)-b(s,X^2_s,Y^2_s)|^2+|\sigma(s,X^1_s,Y^1_s)-\sigma(s,X^2_s,Y^2_s)|^2\big)\Big\}\,\diff s\\
	&\hspace{0.5cm}+2\int^t_0e^{-\lambda s}\Big\langle\int_\supp r(\theta)\big(\hat{Y}^1_s(\theta)-\hat{Y}^2_s(\theta)\big)\dmu,\big(\sigma(s,X^1_s,Y^1_s)-\sigma(s,X^2_s,Y^2_s)\big)\,\diff W_s\Big\rangle
\end{align*}
for $t\in[0,T]$. On the one hand, taking the expectations, we have
\begin{align*}
	&\bE\Big[\int^T_0e^{-\lambda t}\Big\{(\lambda-3)\|\hat{Y}^1_t-\hat{Y}^2_t\|^2_{\cH_\mu}+2\|\hat{Y}^1_t-\hat{Y}^2_t\|^2_{\cV_\mu}\Big\}\,\diff t\Big]\\
	&\leq\int_\supp r(\theta)\dmu\,\bE\Big[\int^T_0e^{-\lambda t}\big\{|b(t,X^1_t,Y^1_t)-b(t,X^2_t,Y^2_t)|^2+|\sigma(t,X^1_t,Y^1_t)-\sigma(t,X^2_t,Y^2_t)|^2\big\}\,\diff t\Big].
\end{align*}
On the other hand, noting that $\lambda>3$, by the Burkholder--Davis--Gundy inequality,
\begin{align*}
	&\bE\Big[\sup_{t\in[0,T]}e^{-\lambda t}\|\hat{Y}^1_t-\hat{Y}^2_t\|^2_{\cH_\mu}\Big]\\
	&\leq\int_\supp r(\theta)\dmu\,\bE\Big[\int^T_0e^{-\lambda t}\big\{|b(t,X^1_t,Y^1_t)-b(t,X^2_t,Y^2_t)|^2+|\sigma(t,X^1_t,Y^1_t)-\sigma(t,X^2_t,Y^2_t)|^2\big\}\,\diff t\Big]\\
	&\hspace{0.5cm}+2c_\BDG\bE\Big[\Big(\int^T_0e^{-2\lambda t}\Big|\int_\supp r(\theta)(\hat{Y}^1_t(\theta)-\hat{Y}^2_t(\theta))\dmu\Big|^2|\sigma(t,X^1_t,Y^1_t)-\sigma(t,X^2_t,Y^2_t)|^2\,\diff t\Big)^{1/2}\Big]\\
	&\leq\int_\supp r(\theta)\dmu\,\bE\Big[\int^T_0e^{-\lambda t}\big\{|b(t,X^1_t,Y^1_t)-b(t,X^2_t,Y^2_t)|^2+|\sigma(t,X^1_t,Y^1_t)-\sigma(t,X^2_t,Y^2_t)|^2\big\}\,\diff t\Big]\\
	&\hspace{0.5cm}+2c^2_\BDG\int_\supp r(\theta)\dmu\,\bE\Big[\int^T_0e^{-\lambda t}|\sigma(t,X^1_t,Y^1_t)-\sigma(t,X^2_t,Y^2_t)|^2\,\diff t\Big]\\
	&\hspace{0.5cm}+\frac{1}{2}\bE\Big[\sup_{t\in[0,T]}e^{-\lambda t}\|\hat{Y}^1_t-\hat{Y}^2_t\|^2_{\cH_\mu}\Big].
\end{align*}
Thus, we get
\begin{align*}
	&\bE\Big[2\sup_{t\in[0,T]}e^{-\lambda t}\|\hat{Y}^1_t-\hat{Y}^2_t\|^2_{\cH_\mu}+\int^T_0e^{-\lambda t}\Big\{(\lambda-3)\|\hat{Y}^1_t-\hat{Y}^2_t\|^2_{\cH_\mu}+2\|\hat{Y}^1_t-\hat{Y}^2_t\|^2_{\cV_\mu}\Big\}\,\diff t\Big]\\
	&\leq(5+8c^2_\BDG)\int_\supp r(\theta)\dmu\,\bE\Big[\int^T_0e^{-\lambda t}\big\{|b(t,X^1_t,Y^1_t)-b(t,X^2_t,Y^2_t)|^2\\
	&\hspace{7cm}+|\sigma(t,X^1_t,Y^1_t)-\sigma(t,X^2_t,Y^2_t)|^2\big\}\,\diff t\Big].
\end{align*}
By \cref{lift_assum_general-SEE} (iii) (the global Lipschitz condition), we have
\begin{equation*}
	|b(t,X^1_t,Y^1_t)-b(t,X^2_t,Y^2_t)|^2+|\sigma(t,X^1_t,Y^1_t)-\sigma(t,X^2_t,Y^2_t)|^2\leq 2L^2\big\{|X^1_t-X^2_t|^2+\|Y^1_t-Y^2_t\|^2_{\cH_\mu}\big\}.
\end{equation*}
Furthermore, by \cref{lift_lemm_space} (iii), there exists a constant $M=M(\mu,L)>0$ such that
\begin{align*}
	|X^1_t-X^2_t|^2&=\Big|\int_\supp\big(Y^1_t(\theta)-Y^2_t(\theta)\big)\dmu\Big|^2\\
	&\leq\frac{1}{2L^2(5+8c^2_\BDG)\int_\supp r(\theta)\dmu}\|Y^1_t-Y^2_t\|^2_{\cV_\mu}+M\|Y^1_t-Y^2_t\|^2_{\cH_\mu}
\end{align*}
a.s.\ for a.e.\ $t\in(0,T]$.
Then we obtain
\begin{align*}
	&\bE\Big[2\sup_{t\in[0,T]}e^{-\lambda t}\|\hat{Y}^1_t-\hat{Y}^2_t\|^2_{\cH_\mu}+\int^T_0e^{-\lambda t}\Big\{(\lambda-3)\|\hat{Y}^1_t-\hat{Y}^2_t\|^2_{\cH_\mu}+2\|\hat{Y}^1_t-\hat{Y}^2_t\|^2_{\cV_\mu}\Big\}\,\diff t\Big]\\
	&\leq\bE\Big[\int^T_0e^{-\lambda t}\Big\{2L^2(M+1)(5+8c^2_\BDG)\int_\supp r(\theta)\dmu\,\|Y^1_t-Y^2_t\|^2_{\cH_\mu}+\|Y^1_t-Y^2_t\|^2_{\cV_\mu}\Big\}\,\diff t\Big].
\end{align*}
Thus, by taking $\kappa=2L^2(M+1)(5+8c^2_\BDG)\int_\supp r(\theta)\dmu>0$ and $\lambda=3+2\kappa>3$, we have
\begin{equation*}
	\|\hat{Y}^1-\hat{Y}^2\|^2_{T,\lambda,\kappa}\leq\frac{1}{2}\|Y^1-Y^2\|^2_{T,\lambda,\kappa}.
\end{equation*}
Therefore, the map $\Phi$ is contractive with respect to $\|\cdot\|_{T,\lambda,\kappa}$, and we see that there exists a unique mild solution of the generalized SEE \eqref{lift_eq_general-SEE}.

The assertion (ii) follows from \cref{lift_theo_bound}. Indeed, if we set $\tilde{b}:\Omega\times[0,\infty)\times\bR^n\times\cH_\mu\to\bR^n$ and $\tilde{\sigma}:\Omega\times[0,\infty)\times\bR^n\times\cH_\mu\to\bR^{n\times d}$ by
\begin{equation*}
	\tilde{b}(t,x,y):=b(t,x+\mu[\bar{Y}_t],y+\bar{Y}_t)-\bar{b}_t\ \text{and}\ \tilde{\sigma}(t,x,y):=\sigma(t,x+\mu[\bar{Y}_t],y+\bar{Y}_t)-\bar{\sigma}_t,
\end{equation*}
then $\tilde{b}$ and $\tilde{\sigma}$ satisfy \cref{lift_assum_general-SEE} (i) (the measurability condition) and (ii) (the linear growth condition) with $c_\LG$ and $\varphi$ replaced by $L$ and $|b(t,\mu[\bar{Y}_t],\bar{Y}_t)-\bar{b}_t|+|\sigma(t,\mu[\bar{Y}_t],\bar{Y}_t)-\bar{\sigma}_t|$, respectively. Since $\tilde{Y}:=Y-\bar{Y}$ is a mild solution of the generalized SEE \eqref{lift_eq_general-SEE} with the coefficients $\tilde{b}$ and $\tilde{\sigma}$ and the initial condition $\tilde{\eta}:=\eta-\bar{\eta}\in L^2_{\cF_0}(\cH_\mu)$, we obtain the assertion (ii) by \cref{lift_theo_bound}. This completes the proof.
\end{proof}

%% Remark

\begin{rem}\label{appendix_rem_fixed-point}
From the above proof, the Banach fixed point theorem yields that the unique mild solution $Y$ to the generalized SEE \eqref{lift_eq_general-SEE} with the initial condition $\eta\in L^2_{\cF_0}(\cH_\mu)$ can be approximated by the iterations of the map $\Phi$ under the norm $\|\cdot\|_{T,\lambda,\kappa}$ with $\lambda,\kappa>0$ specified in the above proof. Namely, by setting $Y^0:=0$ and $Y^k:=\Phi(Y^{k-1})$ for $k\in\bN$, we have $\lim_{k\to\infty}\|Y-Y^k\|_{T,\lambda,\kappa}=0$ for any $T>0$.
\end{rem}

%%%%%%%%%%%%%%%%%
%% Subsection
%%%%%%%%%%%%%%%%%

\subsection{Proof of \cref{lift_theo_local-Lip}: Existence and uniqueness; the local Lipschitz case}\label{appendix_3}

Suppose that \cref{lift_assum_kernel} holds. Let $b:\Omega\times[0,\infty)\times\bR^n\times\cH_\mu\to\bR^n$ and $\sigma:\Omega\times[0,\infty)\times\bR^n\times\cH_\mu\to\bR^{n\times d}$ satisfy \cref{lift_assum_general-SEE} (i) (the measurability condition), (ii) (the linear growth condition) and (iii)' (the local Lipschitz condition).

%% Proof

\begin{proof}[Proof of \cref{lift_theo_local-Lip}]
First, we show the existence of a mild solution. Let $\eta\in L^2_{\cF_0}(\cH_\mu)$ be fixed. For each $k\in\bN$, define $b_k:\Omega\times[0,\infty)\times\bR^n\times\cH_\mu\to\bR^n$ and $\sigma_k:\Omega\times[0,\infty)\times\bR^n\times\cH_\mu\to\bR^{n\times d}$ by
\begin{equation*}
	b_k(t,x,y):=b(t,x,\pi_k(y))\1_{[0,\tau_k)}(t)\ \text{and}\ \sigma_k(t,x,y):=\sigma(t,x,\pi_k(y))\1_{[0,\tau_k)}(t),
\end{equation*}
where $\pi_k:\cH_\mu\to\cH_\mu$ is defined by
\begin{equation*}
	\pi_k(y):=
	\begin{dcases}
	y\ &\text{if $\|y\|_{\cH_\mu}\leq k$},\\
	\frac{k}{\|y\|_{\cH_\mu}}y\ &\text{if $\|y\|_{\cH_\mu}>k$}.
	\end{dcases}
\end{equation*}
Then $b_k$ and $\sigma_k$ satisfy \cref{lift_assum_general-SEE} (i) (the measurability condition), (ii) (the linear growth condition) with the common $c_\LG$ and $\varphi$, and (iii) (the global Lipschitz condition) with the constant $L$ replaced by $L_k$. Thus, by \cref{lift_theo_global-Lip}, there exists a unique mild solution $Y^k$ to the generalized SEE \eqref{lift_eq_general-SEE} with the coefficients $b_k$ and $\sigma_k$ and the initial condition $\eta$. Furthermore, by \cref{lift_theo_bound}, we have
\begin{equation}\label{lift_eq_uniform-bound}
	\bE\Big[\sup_{t\in[0,T]}\|Y^k_t\|^2_{\cH_\mu}+\int^T_0\|Y^k_t\|^2_{\cV_\mu}\,\diff t\Big]\leq Ce^{CT}\bE\Big[\|\eta\|^2_{\cH_\mu}+\int^T_0\varphi^2_t\,\diff t\Big]<\infty
\end{equation}
for any $T>0$ and any $k\in\bN$, where the constant $C>0$ depends only on the measure $\mu$ and the linear growth constant $c_\LG$. For each $k\in\bN$, define a stopping time $\hat{\tau}_k$ by
\begin{equation*}
	\hat{\tau}_k:=\inf\{t\geq0\,|\,\|Y^k_t\|_{\cH_\mu}>k\}\wedge\tau_k.
\end{equation*}
Then by \cref{lift_theo_global-Lip} (ii), for any $k_1,k_2\in\bN$ with $k_1>k_2$ and any $T>0$, we have
\begin{align*}
	&\bE\Big[\sup_{t\in[0,T]}\|Y^{k_1}_{t\wedge\hat{\tau}_{k_2}}-Y^{k_2}_{t\wedge\hat{\tau}_{k_2}}\|^2_{\cH_\mu}+\int^{T\wedge\hat{\tau}_{k_2}}_0\|Y^{k_1}_t-Y^{k_2}_t\|^2_{\cV_\mu}\,\diff t\Big]\\
	&\leq C_{k_1}e^{C_{k_1}T}\bE\Big[\int^{T\wedge\hat{\tau}_{k_2}}_0\Big\{\big|b_{k_1}(t,\mu[Y^{k_2}_t],Y^{k_2}_t)-b_{k_2}(t,\mu[Y^{k_2}_t],Y^{k_2}_t)\big|^2\\
	&\hspace{5cm}+\big|\sigma_{k_1}(t,\mu[Y^{k_2}_t],Y^{k_2}_t)-\sigma_{k_2}(t,\mu[Y^{k_2}_t],Y^{k_2}_t)\big|^2\Big\}\,\diff t\Big],
\end{align*}
for some constant $C_{k_1}>0$ which depends only on $\mu$ and $L_{k_1}$. By the definitions of $b_k$, $\sigma_k$ and $\hat{\tau}_k$, noting that $k_1>k_2$, we see that the right-hand side is equal to zero. Therefore, we see that
\begin{equation*}
	Y^{k_1}_{t\wedge\hat{\tau}_{k_2}}=Y^{k_2}_{t\wedge\hat{\tau}_{k_2}}\ \text{for any $t\geq0$ a.s.}
\end{equation*}
This further implies that $\hat{\tau}_{k_1}\geq\hat{\tau}_{k_2}$ a.s. Indeed, for any $T\geq0$, $\hat{\tau}_{k_2}\geq T$ implies that $\tau_{k_1}\geq\tau_{k_2}\geq T$ by the assumption of $\{\tau_k\}_{k\in\bN}$, and
\begin{equation*}
	k_1>k_2\geq\sup_{t\in[0,T]}\|Y^{k_2}_t\|_{\cH_\mu}=\sup_{t\in[0,T]}\|Y^{k_1}_t\|_{\cH_\mu},
\end{equation*}
which implies that $\inf\{t\geq0\,|\,\|Y^{k_1}_t\|_{\cH_\mu}>k_1\}\geq T$, and hence $\hat{\tau}_{k_1}\geq T$. Thus, the sequence $\{\hat{\tau}_k\}_{k\in\bN}$ of stopping times is non-decreasing. Furthermore, for any $T>0$,
\begin{align*}
	\bP\big(\lim_{k\to\infty}\hat{\tau}_k\leq T\big)&=\lim_{k\to\infty}\bP\big(\hat{\tau}_k\leq T\big)=\lim_{k\to\infty}\bP\Big(\tau_k\leq T\ \text{or}\ \sup_{t\in[0,T]}\|Y^k_t\|_{\cH_\mu}>k\Big)\\
	&\leq\limsup_{k\to\infty}\bP\big(\tau_k\leq T\big)+\limsup_{k\to\infty}\bP\Big(\sup_{t\in[0,T]}\|Y^k_t\|_{\cH_\mu}>k\Big)\\
	&\leq\limsup_{k\to\infty}\frac{1}{k^2}\bE\Big[\sup_{t\in[0,T]}\|Y^k_t\|^2_{\cH_\mu}\Big],
\end{align*}
and the last term is zero due to the uniform bound \eqref{lift_eq_uniform-bound}. This implies that $\lim_{k\to\infty}\hat{\tau}_k=\infty$ a.s. Now we define
\begin{equation*}
	Y_t:=
	\begin{dcases}
	\lim_{k\to\infty}Y^k_t\ &\text{if the limit exists},\\
	0\ &\text{otherwise},
	\end{dcases}
\end{equation*}
where the limit is taken in $\cH_\mu$. Then we see that $Y$ is an $\cH_\mu$-valued adapted process and satisfies
\begin{equation*}
	Y_{t\wedge\hat{\tau}_k}=Y^k_{t\wedge\hat{\tau}_k}\ \text{for any $t\geq0$ a.s.}
\end{equation*}
for any $k\in\bN$. Since each $Y^k$ is $\cH_\mu$-continuous and $\lim_{k\to\infty}\hat{\tau}_k=\infty$ a.s., we see that $Y$ is $\cH_\mu$-continuous a.s. Furthermore, by Fatou's lemma and \eqref{lift_eq_uniform-bound},
\begin{align*}
	\bE\Big[\sup_{t\in[0,T]}\|Y_t\|^2_{\cH_\mu}+\int^T_0\|Y_t\|^2_{\cV_\mu}\,\diff t\Big]&\leq\liminf_{k\to\infty}\bE\Big[\sup_{t\in[0,T]}\|Y^k_{t\wedge\hat{\tau}_k}\|^2_{\cH_\mu}+\int^{T\wedge\hat{\tau}_k}_0\|Y^k_t\|^2_{\cV_\mu}\,\diff t\Big]\\
	&\leq Ce^{CT}\bE\Big[\|\eta\|^2_{\cH_\mu}+\int^T_0\varphi^2_t\,\diff t\Big]<\infty
\end{align*}
for any $T>0$. In particular, by \cref{lift_lemm_space} (iii) and \cref{lift_assum_general-SEE} (ii) (the linear growth condition), we have
\begin{equation*}
	\bE\Big[\int^T_0\Big\{|b(t,\mu[Y_t],Y_t)|+|\sigma(t,\mu[Y_t],Y_t)|^2\Big\}\,\diff t\Big]<\infty
\end{equation*}
for any $T>0$. Then, by the definitions of $b_k$, $\sigma_k$ and $\hat{\tau}_k$, we have
\begin{align*}
	Y_t\1_{\{t\leq\hat{\tau}_k\}}&=Y^k_t\1_{\{t\leq\hat{\tau}_k\}}\\
	&=\Big\{e^{-\cdot t}\eta(\cdot)+\int^t_0e^{-\cdot(t-s)}b_k(s,\mu[Y^k_s],Y^k_s)\,\diff s+\int^t_0e^{-\cdot(t-s)}\sigma_k(s,\mu[Y^k_s],Y^k_s)\,\diff W_s\Big\}\1_{\{t\leq\hat{\tau}_k\}}\\
	&=\Big\{e^{-\cdot t}\eta(\cdot)+\int^t_0e^{-\cdot(t-s)}b(s,\mu[Y_s],Y_s)\,\diff s+\int^t_0e^{-\cdot(t-s)}\sigma(s,\mu[Y_s],Y_s)\,\diff W_s\Big\}\1_{\{t\leq\hat{\tau}_k\}}
\end{align*}
a.s.\ for any $t\geq0$ and any $k\in\bN$. By taking the limit $k\to\infty$, we see that
\begin{equation*}
	Y_t=e^{-\cdot t}\eta(\cdot)+\int^t_0e^{-\cdot(t-s)}b(s,\mu[Y_s],Y_s)\,\diff s+\int^t_0e^{-\cdot(t-s)}\sigma(s,\mu[Y_s],Y_s)\,\diff W_s
\end{equation*}
a.s.\ for any $t\geq0$. Therefore, $Y$ is a mild solution of the generalized SEE \eqref{lift_eq_general-SEE}.

Next, we prove the uniqueness. Let $\bar{Y}$ be another mild solution of the generalized SEE \eqref{lift_eq_general-SEE}, and define
\begin{equation*}
	\bar{\tau}_k:=\inf\{t\geq0\,|\,\|\bar{Y}_t\|_{\cH_\mu}>k\}\wedge\hat{\tau}_k
\end{equation*}
for any $k\in\bN$. Then $\{\bar{\tau}_k\}_{k\in\bN}$ is a sequence of stopping times such that $\bar{\tau}_k\leq\bar{\tau}_{k+1}$ for any $k\in\bN$ and $\lim_{k\to\infty}\bar{\tau}_k=\infty$ a.s. By the construction of $Y$ and \cref{lift_theo_global-Lip} (ii), for any $T>0$ and $k\in\bN$,
\begin{align*}
	&\bE\Big[\sup_{t\in[0,T]}\|Y_{t\wedge\bar{\tau}_k}-\bar{Y}_{t\wedge\bar{\tau}_k}\|^2_{\cH_\mu}+\int^{T\wedge\bar{\tau}_k}_0\|Y_t-\bar{Y}_t\|^2_{\cV_\mu}\,\diff t\Big]\\
	&=\bE\Big[\sup_{t\in[0,T]}\|Y^k_{t\wedge\bar{\tau}_k}-\bar{Y}_{t\wedge\bar{\tau}_k}\|^2_{\cH_\mu}+\int^{T\wedge\bar{\tau}_k}_0\|Y^k_t-\bar{Y}_t\|^2_{\cV_\mu}\,\diff t\Big]\\
	&\leq  C_ke^{C_kT}\bE\Big[\int^{T\wedge\bar{\tau}_k}_0\Big\{|b_k(t,\mu[\bar{Y}_t],\bar{Y}_t)-b(t,\mu[\bar{Y}_t],\bar{Y}_t)|^2+|\sigma_k(t,\mu[\bar{Y}_t],\bar{Y}_t)-\sigma(t,\mu[\bar{Y}_t],\bar{Y}_t)|^2\Big\}\,\diff t\Big]
\end{align*}
for some constant $C_k>0$ depending only on $\mu$ and $L_k$. By the definitions of $b_k$, $\sigma_k$ and $\bar{\tau}_k$, we see that the last term is zero. Thus, $\bar{Y}_{t\wedge\bar{\tau}_k}=Y_{t\wedge\bar{\tau}_k}$ for any $t\geq0$ a.s.\ for any $k\in\bN$. Since $\lim_{k\to\infty}\bar{\tau}_k=\infty$ a.s., we see that $\bar{Y}_t=Y_t$ for any $t\geq0$ a.s. This proves the uniqueness.
\end{proof}

%%%%%%%%%%%%%%%%%%%%%%%%%%%%
%%%%%% Acknowledgements
%%%%%%%%%%%%%%%%%%%%%%%%%%%%

\section*{Acknowledgments}
The author would like to thank Professor Bin Xie for insightful discussions, and Professor Masanori Hino for pointing out an error in \cref{Harnack_lemm_R} in the previous version.

%%%%%%%%%%%%%%%%%%%%%%%%%%%%
%%%%%% References
%%%%%%%%%%%%%%%%%%%%%%%%%%%%


\begin{thebibliography}{99}

\bibitem{AbiJa19}
E.\ Abi Jaber,
Lifting the Heston model,
{\it Quant.\ Finance},
19, 1995--2013,
2019.

\bibitem{AbiJaCuLaPu21}
E.\ Abi Jaber, C.\ Cuchiero, M.\ Larsson, and S.\ Pulido,
A weak solution theory for stochastic Volterra equations of convolution type,
{\it Ann.\ Appl.\ Probab.},
31(6), 2924--2952,
2021.

\bibitem{AbiJaEu19}
E.\ Abi Jaber and O.\ El Euch,
Multifactor approximation of rough volatility models,
{\it SIAM J.\ Financ.\ Math.},
10(2), 309--349,
2019.

\bibitem{AbiJaMiPh21}
E.\ Abi Jaber, E.\ Miller, and H.\ Pham,
Linear-Quadratic control for a class of stochastic Volterra equations: solvability and approximation,
{\it Ann.\ Appl.\ Probab.},
31(5), 2244--2274,
2021.

\bibitem{AlKe21}
A.\ Alfonsi and A.\ Kebaier,
Approximation of stochastic Volterra equations with kernels of completely monotone type,
{\it preprint},
arXiv:2102.13505,
2021.

\bibitem{ArThWa09}
M.\ Arnaudon, A.\ Thalmaier, and F.-Y.\ Wang,
Gradient estimates and Harnack inequalities on non-compact Riemannian manifolds,
{\it Stochastic Process.\ Appl.},
119(10), 3653--3670,
2009.

\bibitem{BaWaYu19}
J.\ Bao, F.-Y.\ Wang, and C.\ Yuan,
Asymptotic log-Harnack inequality and applications for stochastic systems of infinite memory,
{\it Stochastic Process.\ Appl.},
129(11), 4576--4596,
2019.

\bibitem{BaBeVe11}
O.\ Barndorff-Nielsen, F.\ Benth, and A.\ Veraart,
Ambit processes and stochastic partial differential equations,
In {\it Advanced Mathematical Methods for Finance, eds G.\ Di Nunno and B.\ {\O}ksendal},
35--74,
Springer, Berlin, Heidelberg.
2011.

\bibitem{BaBr23}
C.\ Bayer and S.\ Breneis,
Markovian approximations of stochastic Volterra equations with the fractional kernel,
{\it Quant. Finance},
23(1), 53--70,
2023.

\bibitem{BeDeKr22}
F.E.\ Benth, N.\ Detering, and P.\ Kr\"{u}hner,
Stochastic Volterra integral equations and a class of first-order stochastic partial differential equations,
{\it Stochastics},
94(7), 1054--1076,
2022.

\bibitem{BeMi80a}
M.A.\ Berger and V.J.\ Mizel,
Volterra equations with It\^o integrals---I,
{\it J. Integral Equations},
2(3), 187--245,
1980.

\bibitem{BeMi80b}
M.A.\ Berger and V.J.\ Mizel,
Volterra equations with It\^o integrals---II,
{\it J. Integral Equations},
2(4), 319--337,
1980.

\bibitem{CaCu98}
P.\ Carmona and L.\ Coutin,
Fractional Brownian motion and the Markov property,
{\it Electron.\ Commun.\ Probab.},
3, 95--107,
1998.

\bibitem{CuTe19}
C.\ Cuchiero and J.\ Teichmann,
Markovian lifts of positive semidefinite affine Volterra-type processes,
{\it Decis.\ Econ.\ Finance},
42, 407--448,
2019.

\bibitem{CuTe20}
C.\ Cuchiero and J.\ Teichmann,
Generalized Feller processes and Markovian lifts of stochastic Volterra processes: the affine case,
{\it J.\ Evol.\ Equ.},
20, 1301--1348,
2020.

\bibitem{DaPrZa96}
G.\ Da Prato and J.\ Zabczyk,
{\it Ergodicity for Infinite Dimensional Systems},
London Mathematical Society Lecture Note Series,
Cambridge University Press, Cambridge,
1996.

\bibitem{DaPrZa14}
G.\ Da Prato and J.\ Zabczyk,
{\it Stochastic Equations in Infinite Dimensions}, 2nd edition,
Cambridge University Press, Cambridge,
2014.

\bibitem{FrJi22}
M.\ Friesen and P.\ Jin,
Volterra square-root process: Stationarity and regularity of the law,
{\it Ann.\ Appl.\ Probab.}
(to appear),
arXiv:2203.08677,
2022.

\bibitem{GaJaRo18}
J.\ Gatheral, T.\ Jaisson, and M.\ Rosenbaum,
Volatility is rough,
{\it Quant.\ Finance},
18(6), 933--949,
2018.

\bibitem{GaMa10}
L.\ Gawarecki and V.\ Mandrekar,
{\it Stochastic Differential Equations in Infinite Dimensions: with Applications to Stochastic Partial Differential Equations},
Springer, Berlin,
2010.

\bibitem{GoXi20}
L.\ Gouden\`{e}ge and B.\ Xie,
Ergodicity of stochastic Cahn--Hilliard equations with logarithmic potentials driven by degenerate or nondegenerate noises,
{\it J.\ Differ.\ Equ.},
269(9), 6988--7014,
2020.

\bibitem{GrLoSt90}
G.\ Gripenberg, S.O.\ Londen, and O.\ Staffans,
{\it Volterra Integral and Functional Equations},
volume 34 of {\it Encyclopedia of Mathematics and its Applications},
Cambridge University Press, Cambridge,
1990.

\bibitem{HaMa06}
M.\ Hairer and J.C.\ Mattingly,
Ergodicity of the 2D Navier--Stokes equations with degenerate stochastic forcing,
{\it Ann.\ Math.},
164, 993--1032,
2006.

\bibitem{HaSt19}
P.\ Harms and D.\ Stefanovits,
Affine representations of fractional processes with applications in mathematical finance,
{\it Stochastic Process.\ Appl.},
129(4), 1185--1228,
2019.

\bibitem{HoLiLi20}
W.\ Hong, S.\ Li, and W.\ Liu,
Asymptotic log-Harnack inequality and applications for SPDE with degenerate multiplicative noise,
{\it Stat.\ Probab.\ Lett.},
164, 108810,
2020.

\bibitem{HoLiLi21a}
W.\ Hong, S.\ Li, and W.\ Liu,
Asymptotic log-Harnack inequality and ergodicity for 3D Leray-$\alpha$ model with degenerate type noise,
{\it Potential Anal.},
55, 477--490,
2021.

\bibitem{HoLiLi21b}
W.\ Hong, S.\ Li, and W.\ Liu,
Asymptotic log-Harnack inequality and applications for stochastic 2D hydrodynamical-type systems with degenerate noise,
{\it J.\ Evol.\ Equ.},
21, 419--440,
2021.

%\bibitem{JaPaSp22}
%A.\ Jacquier, A.\ Pannier, and K.\ Spiliopoulos,
%On the ergodic behaviour of affine Volterra processes,
%{\it preprint},
%arXiv:2204.05270,
%2022.

\bibitem{KrRo79}
N.V.\ Krylov and B.L.\ Rozovskii,
Stochastic evolution equations,
{\it Itogi Nauki i Tekhniki, Seria Sovremiennyie Problemy Matematiki},
14, 71--146,
1979 (in Russian);
English translation in {\it J.\ Sov.\ Math.},
16, 1233--1277,
1981.

\bibitem{Ku14}
T.\ Kurtz,
Weak and strong solutions of general stochastic models,
{\it Electron.\ Commun.\ Probab.},
19(58), 1--16,
2014.

\bibitem{LiRo10}
W.\ Liu and M.\ R\"{o}ckner,
SPDE in Hilbert space with locally monotone coefficients,
{\it J.\ Funct.\ Anal.},
259, 2902--2922,
2010.

\bibitem{Li20}
Z.\ Liu,
Asymptotic log-Harnack inequality for monotone SPDE with multiplicative noise,
{\it preprint},
arXiv:2007.13080,
2020.

\bibitem{Pr05}
P.E.\ Protter,
{\it Stochastic Integration and Differential Equations}, 2nd edition,
Springer, Berlin,
2005.

\bibitem{SaKiMa87}
S.G.\ Samko, A.A.\ Kilbas, and O.I.\ Marichev,
{\it Fractional Integrals and Derivatives, Theory and Applications},
Gordon and Breach Science Publishers, Yverdon, Switzerland,
1987.

\bibitem{Sc06}
J.\ Schmiegel,
Self-scaling tumor growth,
{\it Phys.\ A: Stat.\ Mech.\ Appl.},
367(C), 509--524,
2006.

\bibitem{Ve12}
M.\ Veraar,
The stochastic Fubini theorem revisited,
{\it Stochastics},
84(4), 543--551,
2012.

\bibitem{Wa97}
F.-Y.\ Wang,
Logarithmic Sobolev inequalities on noncompact Riemannian manifolds,
{\it Probab.\ Theory Related Fields},
109, 417--424,
1997.

\bibitem{Wa10}
F.-Y.\ Wang,
Harnack inequalities on manifolds with boundary and applications,
{\it J.\ Math.\ Pures Appl.},
94, 304--321,
2010.

\bibitem{Wa13}
F.-Y.\ Wang,
{\it Harnack Inequalities and Applications for Stochastic Partial Differential Equations},
Springer, Berlin,
2013.

\bibitem{WaWuYiZh22}
Y.\ Wang, F.\ Wu, G.\ Yin, and C.\ Zhu,
Stochastic functional differential equations with infinite delay under non-Lipschitz coefficients: Existence and uniqueness, Markov property, ergodicity, and asymptotic log-Harnack inequality,
{\it Stochastic Process.\ Appl.},
149, 1--38,
2022.

\bibitem{Xu11}
L.\ Xu,
A modified log-Harnack inequality and asymptotically strong Feller property,
{\it J.\ Evol.\ Equ.},
11, 925--942,
2011.

\end{thebibliography}
\end{document}